\documentclass[11pt,reqno]{amsart}
\usepackage{amssymb,amscd,amsfonts,amsbsy,enumerate}
\usepackage{latexsym}
\usepackage{exscale}
\usepackage{amsmath,amsthm,amsfonts}
\usepackage{mathrsfs}
\usepackage{epsf,epsfig}

\usepackage[colorlinks=true,linkcolor=blue,citecolor=red,urlcolor=red]{hyperref} 

\usepackage{esint} 

\usepackage{color}

\numberwithin{equation}{section}

\def\XXint#1#2#3{{\setbox0=\hbox{$#1{#2#3}{\int}$}
     \vcenter{\hbox{$#2#3$}}\kern-.5\wd0}}

\allowdisplaybreaks

\newtheorem{theorem}{Theorem}[section]
\newtheorem{lemma}[theorem]{Lemma}
\newtheorem{corollary}[theorem]{Corollary}
\newtheorem{proposition}[theorem]{Proposition}
\newtheorem{definition}[theorem]{Definition}

\theoremstyle{remark}

\begin{document}
\allowdisplaybreaks

\title[Riesz Transform Characterizations of $H^1$ and {\rm BMO}]
{Riesz Transform Characterizations of $H^1$ and {\rm BMO} on Ahlfors Regular Sets with Small Oscillations}

\author{Dorina Mitrea}
\address{Dorina Mitrea
\\
Department of Mathematics
\\
Baylor University
\\
Sid Richardson Bldg., 1410 S.~4th Street
\\
Waco, TX 76706, USA} \email{Dorina\_\,Mitrea@baylor.edu}

\author{Irina Mitrea}
\address{Irina Mitrea
\\
Department of Mathematics
\\
Temple University\!
\\
1805\,N.\,Broad\,Street
\\
Philadelphia, PA 19122, USA} \email{imitrea@temple.edu}

\author{Marius Mitrea\\ }
\address{Marius Mitrea
\\
Department of Mathematics
\\
Baylor University
\\
Sid Richardson Bldg., 1410 S.~4th Street
\\
Waco, TX 76706, USA} \email{Marius\_\,Mitrea@baylor.edu}

\thanks{The first author has been supported in part by Simons Foundation grant $\#\,$958374. 
The second author has been supported in part by Simons Foundation grant $\#\,$00002820.
The third author has been supported in part by the Simons Foundation grant $\#\,$637481.}




\begin{abstract}
We employ the Riesz transform as a means for describing geometric properties of sets in ${\mathbb{R}}^n$,
and study the extent to which they can be used to characterize function spaces defined on said sets.
In particular, characterizations of the end-point spaces on the Lebesgue scale $L^p$ with $1<p<\infty$, 
namely the Hardy space $H^1$ and the John-Nirenberg space {\rm BMO}, are produced in terms of the 
Riesz transforms on Ahlfors regular sets in ${\mathbb{R}}^n$ with small oscillations (quantified in terms 
of the {\rm BMO} nature of the outward unit normal). These generalize the celebrated results of 
C.~Fefferman and E.~Stein in the flat Euclidean setting. 
\end{abstract}

\maketitle

\allowdisplaybreaks

\centerline{To appear in Annales de L'Institut Fourier}

\section{Introduction}
\label{S-1}

The conventional wisdom is that the Riesz transforms are sensitive to regularity (broadly understood), 
in a retrievable manner. This feature makes the Riesz transforms useful tools both for studying the geometry 
of subsets of the Euclidean ambient and for characterizing the regularity of functions defined on said sets. 

To illustrate this philosophy, let us first elaborate on the role the Riesz transforms play in connection with 
geometric measure theory. Bring in a brand of Riesz transforms whose definition places less stringent demands on 
the underlying set. Specifically, given a closed Ahlfors regular set $\Sigma\subseteq{\mathbb{R}}^n$ along with some 
exponent $\alpha\in(0,1)$, denote by ${\mathscr{C}}_c^{\alpha}(\Sigma)$ the space of H\"older functions of order 
$\alpha$ with compact support in $\Sigma$. This is a Banach space, and we let  $\big({\mathscr{C}}_c^{\alpha}(\Sigma)\big)^\ast$ 
stand for its dual. Then, for each $j\in\{1,\dots,n\}$, one defines the $j$-th distributional Riesz transform as the operator 
\begin{equation}\label{yrf56f-RRR}
{\mathbb{R}}_j:{\mathscr{C}}_c^{\alpha}(\Sigma)\longrightarrow\big({\mathscr{C}}_c^{\alpha}(\Sigma)\big)^\ast
\end{equation}
with the property that for every $f,g\in{\mathscr{C}}_c^{\alpha}(\Sigma)$ one has
\begin{equation}\label{yrf56f-RRR.1}
\big\langle{\mathbb{R}}_jf,g\big\rangle=\frac{1}{\omega_{n-1}}\int_{\Sigma}\int_{\Sigma}
\frac{x_j-y_j}{|x-y|^{n}}\big[f(y)g(x)-f(x)g(y)\big]\,d{\mathcal{H}}^{\,n-1}(y)\,d{\mathcal{H}}^{\,n-1}(x)
\end{equation}
where, in this context, $\langle\cdot,\cdot\rangle$ stands for the natural pairing between the spaces 
$\big({\mathscr{C}}_c^{\,\alpha}(\Sigma)\big)^\ast$ and ${\mathscr{C}}_c^{\,\alpha}(\Sigma)$.

This variety of Riesz transforms turns out to encode a surprising amount of geometric information, 
and may be used to characterize the regularity of the underlying set. One example stems from the work 
of F.~Nazarov, X.~Tolsa, A.~Volberg from \cite{NTV}. In terms of the Riesz transforms \eqref{yrf56f-RRR}-\eqref{yrf56f-RRR.1}, 
the main result in \cite{NTV} may be rephrased as follows: under the background assumption that $\Sigma$ is an Ahlfors 
regular subset of ${\mathbb{R}}^n$ and with $\sigma:={\mathcal{H}}^{\,n-1}\lfloor\Sigma$, one has
\begin{equation}\label{Mabb88}
\parbox{9.40cm}{$\Sigma$ is a {\rm UR} (aka, uniformly rectifiable; cf. Definition~\ref{Def-unif.rect}) set if and only if 
${\mathbb{R}}_j1\in{\rm BMO}(\Sigma,\sigma)$ for every $j\in\{1,\dots,n\}$.}
\end{equation}
Hence, within the class of Ahlfors regular subsets of ${\mathbb{R}}^n$, the lone membership 
of the ${\mathbb{R}}_j1$'s to the John-Nirenberg space ${\rm BMO}$ characterizes uniform 
rectifiability (this refines earlier work of G.~David and S.~Semmes; cf. \cite{DS1991}).

As another example, we have the following theorem:

\begin{theorem}\label{TTYTaf.}
Let $\Omega\subseteq\mathbb{R}^n$ be an Ahlfors regular domain with compact boundary. 
Then for each $\alpha\in(0,1)$ the following claims are equivalent:
\begin{enumerate}
\item[{\rm (a)}] $\Omega$ is a domain of class ${\mathscr{C}}^{1,\alpha}$
{\rm (}aka Lyapunov domain of order $\alpha${\rm )};
\item[{\rm (b)}] the distributional Riesz transforms, defined as in 
\eqref{yrf56f-RRR}-\eqref{yrf56f-RRR.1} with $\Sigma:=\partial\Omega$, satisfy
\begin{equation}\label{eq:RIESZ33}
{\mathbb{R}}_j1\in{\mathscr{C}}^{\alpha}(\partial\Omega)\,\text{ for }\,1\leq j\leq n.
\end{equation}
\end{enumerate}
\end{theorem}

\begin{proof}
This has been proved in \cite[Theorem~1.1]{MMV} under the additional topological assumption that 
$\partial(\,\overline{\Omega}\,)=\partial\Omega$ which, as seen in \cite[(5.10.52), p.\,467]{GHA.I}, 
is superfluous since $\Omega$ is an Ahlfors regular domain to begin with. 
\end{proof}

It has been shown in \cite{HoMiTa07} that if $\Omega\subseteq\mathbb{R}^n$ is an Ahlfors regular domain with compact boundary, 
then $\Omega$ being a ${\mathscr{C}}^1$ domain is equivalent with $\nu\in\big[{\mathscr{C}}^0(\partial\Omega)\big]^n$
(where $\nu$ denotes the geometric measure theoretic outward unit normal to $\Omega$). This being said, the limiting case
$\alpha=0$ of the equivalence {\rm (a)}$\Leftrightarrow${\rm (b)} in Theorem~\ref{TTYTaf.} requires replacing the space of continuous
functions by the (larger) Sarason space ${\rm VMO}$, of functions of vanishing mean oscillations on $\partial\Omega$, 
viewed as a space of homogeneous type, in the sense of Coifman-Weiss, when equipped with the measure 
$\sigma:={\mathcal{H}}^{\,n-1}\lfloor\partial\Omega$ and the Euclidean distance). Specifically, we have the following result:

\begin{theorem}\label{TYTFfa.668}
Let $\Omega\subseteq\mathbb{R}^n$ be an Ahlfors regular domain with compact boundary. 
Abbreviate $\sigma:={\mathcal{H}}^{\,n-1}\lfloor\partial\Omega$ and denote by $\nu$ the 
geometric measure theoretic outward unit normal to $\Omega$. Then 
\begin{equation}\label{eq:iugf.6r4}
\nu\in\big[{\rm VMO}(\partial\Omega,\sigma)\big]^n
\Longleftrightarrow{\mathbb{R}}_j1\in{\rm VMO}(\partial\Omega,\sigma)\,\text{ for }\,1\leq j\leq n.
\end{equation}
\end{theorem}

\begin{proof}
This follows by combining \cite[Theorem~1.4, p.\,107]{MMV} with Theorem~\ref{6trrf.TT.ccc.WACO.222.NEW.AR}.
 \end{proof}

The class of Ahlfors regular domains with compact boundaries and whose geometric measure theoretic outward unit normals have vanishing mean oscillations
has been introduced in \cite[Definition~3.4.1, p.\,158]{GHA.V}, where the label infinitesimally flat {\rm AR} domain has been used; 
see Definition~\ref{t7aab-RFCa.WACO.1}. As remarked in \cite[(3.4.4), p.\,158]{GHA.V}, 
\begin{equation}\label{eq:sffs.WACO.D+M}
\parbox{9.50cm}{\it the class of infinitesimally flat {\rm AR} domains should be thought of as a sharp 
{\rm (}i.e., significantly more inclusive, yet exhibiting similar properties{\rm )} geometric measure 
theoretic version of the class of ${\mathscr{C}}^1$ domains with compact boundary.}
\end{equation}
In addition to bounded ${\mathscr{C}}^1$ domains, examples of infinitesimally flat {\rm AR} domains include planar chord-arc 
domains with vanishing constant and bounded ${\rm VMO}_1$ domains (cf. the discussion on \cite[pp.\,161--173]{GHA.V}).
Also, work in \cite{HoMiTa07} shows that the class of infinitesimally flat {\rm AR} domains is invariant under 
${\mathscr{C}}^1$ diffeomorphisms. 

In view of this discussion, the equivalence in \eqref{eq:iugf.6r4} gives the following result, which may be thought of 
as a version of \eqref{Mabb88} involving more regularity:

\begin{corollary}\label{YRYfuHHJH}
If $\Omega\subseteq\mathbb{R}^n$ is an open set with a compact 
Ahlfors regular boundary and if one sets $\sigma:={\mathcal{H}}^{\,n-1}\lfloor\partial\Omega$ then
\begin{equation}\label{eq:rDC}
\begin{array}{c}
\text{${\mathbb{R}}_j1\in{\rm VMO}(\partial\Omega,\sigma)$ for all $j\in\{1,\dots,n\}$}
\\[4pt]
\text{if and only if $\Omega$ is an infinitesimally flat {\rm AR} domain}.
\end{array}
\end{equation}
\end{corollary}

In addition to the ``distributional'' Riesz transforms \eqref{yrf56f-RRR}-\eqref{yrf56f-RRR.1}, one can also define
``principal-value'' Riesz transforms on a {\rm UR} set $\Sigma\subseteq{\mathbb{R}}^n$. 
These are the singular integral operators $\{R_j\}_{1\leq j\leq n}$ where, for each $j\in\{1,\dots,n\}$,
each function $f\in L^1\big(\Sigma,\frac{{\mathcal{H}}^{n-1}(x)}{1+|x|^{n-1}}\big)$, and 
for ${\mathcal{H}}^{\,n-1}$-a.e. $x\in\Sigma$,
\begin{equation}\label{Cau-RRj}
R_jf(x):=\lim_{\varepsilon\to 0^{+}}\frac{2}{\omega_{n-1}}
\int\limits_{\substack{y\in\Sigma\\ |x-y|>\varepsilon}}
\frac{x_j-y_j}{|x-y|^n}f(y)\,d{\mathcal{H}}^{\,n-1}(y).
\end{equation}
Regarding the relationship between the two brands of Riesz transforms considered above, it is instructive to note that 
(see the discussion in \cite[p.\,104]{MMV}):
\begin{equation}\label{eq:}
\parbox{10.00cm}{if  $\Sigma\subseteq{\mathbb{R}}^n$ is a compact {\rm UR} set then for each $j\in\{1,\dots,n\}$ 
the $j$-th principal-value Riesz transform $R_j$ from \eqref{Cau-RRj} and the $j$-th distributional Riesz transform
${\mathbb{R}}_j$ from \eqref{yrf56f-RRR}-\eqref{yrf56f-RRR.1} agree on ${\mathscr{C}}^{\alpha}(\partial\Omega)$ 
for each $\alpha\in(0,1)$; hence, $R_j1={\mathbb{R}}_j1$ for $1\leq j\leq n$.}
\end{equation}

In the case when $\Sigma\subseteq{\mathbb{R}}^n$ is an unbounded {\rm UR} set, one may alter the definition in \eqref{Cau-RRj} as to
allow the resulting operators to act meaningfully on bounded functions on $\Sigma$. Following \cite[(2.3.31), p.\,352]{GHA.III}, 
if $\sigma:={\mathcal{H}}^{n-1}\lfloor\Sigma$, then for each $j\in\{1,\dots,n\}$ we define the modified $j$-th Riesz transform 
$R^{{}^{\rm mod}}_j$ on $\Sigma$ as the principal-value singular integral operator acting on any given function 
$f\in L^1\big(\Sigma\,,\,\tfrac{\sigma(x)}{1+|x|^n}\big)$ at $\sigma$-a.e. point in $\Sigma$ according to 
\begin{equation}\label{eq:MOD.RZ.1}
R^{{}^{\rm mod}}_jf(x):=\lim\limits_{\varepsilon\to 0^{+}}\int_\Sigma
\big\{k_{j,\varepsilon}(x-y)-k_{j,1}(-y)\big\}f(y)\,d\sigma(y),
\end{equation}
where for each $\varepsilon>0$ we have set 
\begin{equation}\label{eq:MOD.RZ.2}
k_{j,\varepsilon}(z):=\frac{2}{\omega_{n-1}}\frac{z_j}{|z|^n}\cdot
{\mathbf{1}}_{{\mathbb{R}}^n\setminus\overline{B(0,\varepsilon)}}(z)
\,\,\text{ for each }\,\,z\in{\mathbb{R}}^n\setminus\{0\}.
\end{equation}

When $n\in{\mathbb{N}}$ with $n\geq 2$ and the set $\Sigma$ is the hyperplane ${\mathbb{R}}^{n-1}\times\{0\}$ canonically identified 
with the $(n-1)$-dimensional Euclidean space ${\mathbb{R}}^{n-1}$, the definition made in \eqref{Cau-RRj} yields 
the $j$-th Riesz transform in ${\mathbb{R}}^{n-1}$, i.e., the singular integral operator acting on any 
$f\in L^1\big({\mathbb{R}}^{n-1}\,,\,\frac{dx'}{1+|x'|^{n-1}}\big)$ according to 
\begin{align}\label{eq:Rieszdef.INT}
(R_jf)(x'):=\lim\limits_{\varepsilon\to 0^{+}}\frac{2}{\omega_{n-1}}
\int\limits_{\substack{y'\in{\mathbb{R}}^{n-1}\\ |x'-y'|>\varepsilon}}\frac{x_j-y_j}{|x'-y'|^n}f(y')\,dy'
\end{align} 
for ${\mathcal{L}}^{n-1}$-a.e. $x'\in{\mathbb{R}}^{n-1}$.
In this setting, it is well known (see \cite[(4.11), p.\,284]{GCRF85} and \cite[Corollary~1, p.\,221]{St70}; cf. also 
\cite[\S4.3, p.\,123]{Stein93}) that the standard Hardy space\footnote{in ${\mathbb{R}}^{n-1}$, equipped with the 
$(n-1)$-dimensional Lebesgue measure ${\mathcal{L}}^{n-1}$} $H^1({\mathbb{R}}^{n-1},{\mathcal{L}}^{n-1})$ may be described as
\begin{align}\label{u6gGVV.1}
&H^1({\mathbb{R}}^{n-1},{\mathcal{L}}^{n-1})\\[2pt]
&\quad=\big\{f\in L^1({\mathbb{R}}^{n-1},{\mathcal{L}}^{n-1}):\,
R_jf\in L^1({\mathbb{R}}^{n-1},{\mathcal{L}}^{n-1}),\,\,1\leq j\leq n-1\big\}.
\nonumber
\end{align}
In a nutshell, the Riesz transforms characterize the Hardy space $H^1$. One direction in which this result has been extended is 
understanding what other classes of operators may be used in place of the Riesz transforms. This question was asked by C.~Fefferman 
in the survey paper \cite{Fe76} where he suggested that any nondegenerate system of very smooth singular integral operators might 
do the job. In its original formulation, this conjecture turned out to be false, as the counterexample found by J.~Garcia-Cuerva 
in \cite{GC79} shows. A necessary (rank $2$) condition was identified by S.~Janson in \cite{Ja77}, where he also conjectured that 
this is sufficient. The latter property was eventually established by A.~Uchiyama in \cite{Uc82}, thus completing the task of 
identifying optimal classes of singular integral operators (of convolution type) which characterize the Hardy space $H^1$ 
in the Euclidean setting. 

When $\Sigma$ is the (full) flat Euclidean space, the template described in \eqref{eq:MOD.RZ.1}-\eqref{eq:MOD.RZ.2} yields the 
$j$-th modified Riesz transform in ${\mathbb{R}}^{n-1}$, i.e., the singular integral operator acting on any given function 
$f\in L^1\big({\mathbb{R}}^{n-1}\,,\,\frac{dx'}{1+|x'|^n}\big)$ at ${\mathcal{L}}^{n-1}$-a.e. point $x'\in{\mathbb{R}}^{n-1}$ 
according to 
\begin{align}\label{eq:Rieszdef.INT.MOD}
\big(R^{{}^{\rm mod}}_jf\big)(x'):=\lim\limits_{\varepsilon\to 0^{+}}\frac{2}{\omega_{n-1}}
&\int\limits_{{\mathbb{R}}^{n-1}}\Bigg\{\frac{x_j-y_j}{|x'-y'|^n}
{\mathbf{1}}_{{\mathbb{R}}^{n-1}\setminus\overline{B_{n-1}(x',\varepsilon)}}(y')
\\[-2pt]
&-\frac{-y_j}{|-y'|^n}{\mathbf{1}}_{{\mathbb{R}}^{n-1}\setminus\overline{B_{n-1}(0',1)}}(y')\Bigg\}f(y')\,dy'
\nonumber
\end{align} 
where, generally speaking, $B_{n-1}(z',r)$ is the $(n-1)$-dimensional ball in ${\mathbb{R}}^{n-1}$ centered at 
$z'\in{\mathbb{R}}^{n-1}$ and of radius $r$. This happens to be a well-defined linear and bounded operator on 
${\mathrm{BMO}}({\mathbb{R}}^{n-1},{\mathcal{L}}^{n-1})$,
(cf. \cite{FeSt72}, \cite[(5.3.70), p.\,666]{GHA.IV}), and the following characterization of 
${\rm BMO}({\mathbb{R}}^{n-1},{\mathcal{L}}^{n-1})$ may be deduced from \cite[Theorem~2, p.\,587]{Fe}:
\begin{equation}\label{u6gGVV.2}
{\rm BMO}({\mathbb{R}}^{n-1},{\mathcal{L}}^{n-1})=\Big\{g_0+\sum_{j=1}^{n-1}R^{{}^{\rm mod}}_jg_j:
g_0,g_1,\dots,g_n\in L^\infty({\mathbb{R}}^{n-1},{\mathcal{L}}^{n-1})\Big\}.
\end{equation}
In a nutshell, {\it bounded functions and modified Riesz transforms of bounded functions span} {\rm BMO}.
In this vein, we also recall from \cite[(5.3.95), p.\,669]{GHA.IV} the characterization of the Sarason space 
${\rm VMO}({\mathbb{R}}^{n-1},{\mathcal{L}}^{n-1})$ using modified Riesz transforms to the effect that 
\begin{equation}\label{eq:675466475.WACO.1}
\begin{array}{c}
\text{for any given function $f\in{\rm BMO}({\mathbb{R}}^{n-1},{\mathcal{L}}^{n-1})$ one has}
\\[6pt]
f\in{\rm VMO}({\mathbb{R}}^{n-1},{\mathcal{L}}^{n-1})\Leftrightarrow
\,
\begin{cases}
R^{{}^{\rm mod}}_jf\in{\rm VMO}({\mathbb{R}}^{n-1},{\mathcal{L}}^{n-1})
\\[2pt]
\text{ for all }\,j\in\{1,\dots,n-1\}.
\end{cases}
\end{array}
\end{equation}
Other results of this flavor may be found in \cite{MaMiMiMi}. 
Here we wish to mention that, as shown in \cite[Proposition~5.3.2, p.\,670]{GHA.IV}, 
\begin{equation}\label{eq:675466475.WACO.agag.tt}
\begin{array}{c}
\text{for any given function $f\in{\rm BMO}(S^{n-1},{\mathcal{H}}^{\,n-1})$ one has}
\\[6pt]
f\in{\rm VMO}(S^{n-1},{\mathcal{H}}^{\,n-1})\Leftrightarrow
\,
\begin{cases}
R_jf\in{\rm VMO}(S^{n-1},{\mathcal{H}}^{\,n-1})
\\[2pt]
\text{ for all }\,j\in\{1,\dots,n\}.
\end{cases}
\end{array}
\end{equation}

The above discussion brings to the forefront the question of determining whether results similar to 
the characterizations \eqref{u6gGVV.1}, \eqref{u6gGVV.2}, \eqref{eq:675466475.WACO.1}, \eqref{eq:675466475.WACO.agag.tt} 
are true for more general ``surfaces'' (replacing the horizontal plane and the unit sphere). 
The goal here is to answer this question in relation to \eqref{u6gGVV.1} and \eqref{u6gGVV.2}, now working 
with Riesz transforms considered on Ahlfors regular sets with small oscillations, quantified in terms of the {\rm BMO} nature of the 
geometric measure theoretic outward unit normal (see Definition~\ref{t7aab-RFCa.WACO.1} and Definition~\ref{def:USKT}), 
in place of the boundary of the upper half-space or the unit ball, as was the case earlier. 

First, we have the following analogue of \eqref{u6gGVV.1} on bounded ``infinitesimally flat surfaces'' which refines
the result proved by E.~Fabes and C.~Kenig in \cite[Theorem~2.6. p.\,15]{FaKe81} for compact ${\mathscr{C}}^1$ surfaces
(cf. \eqref{eq:sffs.WACO.D+M}; for domains with unbounded boundaries see Theorem~\ref{u6g5rfffc-CCC.1111.UNB} below).

\begin{theorem}\label{u6g5rfffc-CCC.1111}
If $\Omega\subset\mathbb{R}^{n}$ is an infinitesimally flat {\rm AR} domain and $\sigma:={\mathcal{H}}^{\,n-1}\lfloor\partial\Omega$, 
then the Hardy space $H^{1}$ on $\partial\Omega$ {\rm (}regarded as a space of homogeneous type when equipped with the doubling 
measure $\sigma$ and the Euclidean distance; cf., e.g., {\rm\cite{AM15}}{\rm )} may be described as follows:
\begin{equation}\label{7ggVV}
H^{1}(\partial\Omega,\sigma)=\big\{f\in L^{1}(\partial\Omega,\sigma):\,R_jf\in L^{1}(\partial\Omega,\sigma)
\,\text{ for }\,1\leq j\leq n\big\}.
\end{equation}
\end{theorem}

Theorem~\ref{u6g5rfffc-CCC.1111} is in fact implied by a more general result established in the body of this paper, 
specifically Theorem~\ref{yrFCC-TT.aat.TRa}, which encompasses a variety of related corollaries that are of independent interest.
For example, in the same geometric setting as in Theorem~\ref{u6g5rfffc-CCC.1111}, for any family of functions
$f_1,\dots,f_n\in L^1(\partial\Omega,\sigma)$ one has:
\begin{equation}\label{D+M.FEB}
\begin{array}{c}
\text{$f_1,\dots,f_n\in H^1(\partial\Omega,\sigma)$ if and only if}
\\[0pt]
\text{$\sum\limits_{j=1}^nR_jf_j\in L^1(\partial\Omega,\sigma)$ and $R_jf_k-R_kf_j\in L^1(\partial\Omega,\sigma)$ 
for $1\leq j<k\leq n$}.
\end{array}
\end{equation}
Indeed, this readily follows from Theorem~\ref{yrFCC-TT.aat.TRa} applied to the Clifford algebra-valued function 
$f:=\sum\limits_{j=1}^nf_je_j$. The aforementioned characterization of $H^1$ for families of functions is new even when 
$\Omega$ is the unit ball.  

Another way of understanding Theorem~\ref{u6g5rfffc-CCC.1111} is through the lens of mapping properties of 
principal-value integral operators of convolution type 
\begin{equation}\label{eq:gy7tucfy.A1}
T_kf(x):=\lim\limits_{\varepsilon\to 0^{+}}\int\limits_{\substack{y\in\Sigma\\ |x-y|>\varepsilon}}k(x-y)f(y)\,d\sigma(y),\quad 
f\in L^1(\Sigma,\sigma),\quad x\in\Sigma,
\end{equation}
where $\Sigma\subseteq{\mathbb{R}}^n$ is a compact {\rm UR} set, and 
$k\in{\mathscr{C}}^\infty({\mathbb{R}}^n\setminus\{0\})$ is odd and positive homogeneous of degree $1-n$.
The fact that each $T_k$ maps $L^1(\Sigma,\sigma)$ into $L^{1,\infty}(\Sigma,\sigma)$ (cf. \cite[(2.3.19), p.\,350]{GHA.III}) 
invites the question of identifying the largest subspace of $L^1(\Sigma,\sigma)$ which is mapped by any such $T_k$ into 
$L^1(\Sigma,\sigma)$. What Theorem~\ref{u6g5rfffc-CCC.1111} implies (together with the fact that each $T_k$ maps 
$H^1(\Sigma,\sigma)$ into $L^1(\Sigma,\sigma)$; cf. \cite[(2.3.27), p.\,351]{GHA.III}) is that the answer to this question 
is the Hardy space $H^1(\Sigma,\sigma)$ if $\Sigma:=\partial\Omega$ with $\Omega\subset\mathbb{R}^{n}$ an infinitesimally flat 
{\rm AR} domain. Succinctly put, $H^1$ satisfies a natural ``maximality'' property, when considered on compact 
infinitesimally flat surfaces.

Second, we have the following analogue of \eqref{u6gGVV.2} on compact infinitesimally flat surfaces.

\begin{theorem}\label{u6g5rfffc-CCC}
If $\Omega\subset\mathbb{R}^{n}$ is an infinitesimally flat {\rm AR} domain
and $\sigma:={\mathcal{H}}^{\,n-1}\lfloor\partial\Omega$, then the John-Nirenberg space ${\rm BMO}$ 
on $\partial\Omega$ {\rm (}regarded as a space of homogeneous type when equipped with the doubling measure 
$\sigma$ and the Euclidean distance{\rm )} may be described as follows:
\begin{equation}\label{utggG-TRFF.rt.A-CCC}
{\rm BMO}(\partial\Omega,\sigma)=\Big\{f_0+\sum_{j=1}^n R_jf_j:\,
f_0,f_1,\dots,f_n\in L^\infty(\partial\Omega,\sigma)\Big\}.
\end{equation}
\end{theorem}

Here is another perspective on the above result. Given a compact {\rm UR} set $\Sigma\subseteq{\mathbb{R}}^n$, a natural question 
is to determine the smallest enlargement of $L^\infty(\Sigma,\sigma)$ to a space into which all singular integral operators as 
in \eqref{eq:gy7tucfy.A1} take values when acting from $L^\infty(\Sigma,\sigma)$. Theorem~\ref{u6g5rfffc-CCC}, in concert with the fact 
that each $T_k$ maps $L^\infty(\Sigma,\sigma)$ into ${\rm BMO}(\Sigma,\sigma)$ (cf. \cite[(2.3.39), p.\,353]{GHA.III}) shows that the John-Nirenberg space 
${\rm BMO}(\Sigma,\sigma)$ is the minimizer when $\Sigma:=\partial\Omega$ with $\Omega\subset\mathbb{R}^{n}$ an infinitesimally flat {\rm AR} domain.
Hence, ${\rm BMO}$ satisfies a natural ``minimality'' property, when considered on compact infinitesimally flat surfaces.

In relation to these results we remark that, thanks to work in \cite{GHA.V} and the approach developed here, 
the characterizations in Theorems~\ref{u6g5rfffc-CCC.1111}-\ref{u6g5rfffc-CCC} 
continue to hold for the more inclusive class of $\delta$-oscillating {\rm AR} domains with $\delta\in(0,1)$ sufficiently small. 
The latter class of domains has been introduced in \cite[Definition~3.4.1, p.\,158]{GHA.V} as the collection 
of Ahlfors regular domains $\Omega\subset\mathbb{R}^{n}$ with a compact boundary satisfying  
\begin{equation}\label{eq:Rieuyeyer}
{\rm dist}\,\Big(\nu\,,\,\big[{\rm VMO}(\partial\Omega,\sigma)\big]^n\Big)<\delta
\end{equation}
where $\sigma:={\mathcal{H}}^{\,n-1}\lfloor\partial\Omega$, the symbol $\nu$ denotes the geometric measure theoretic outward unit normal to $\Omega$, 
and the distance in the left-hand side of \eqref{eq:Rieuyeyer} is measured in $\big[{\rm BMO}(\partial\Omega,\sigma)\big]^n$.

Moreover, we are able to establish analogous results to those described in Theorems~\ref{u6g5rfffc-CCC.1111}-\ref{u6g5rfffc-CCC}
in the class of Ahlfors regular domains $\Omega\subset\mathbb{R}^{n}$ whose topological boundaries are not compact, now asking, 
in place of \eqref{eq:Rieuyeyer}, that the geometric measure theoretic outward unit normal $\nu$ to $\Omega$ satisfies 
\begin{equation}\label{eq:Rieuyeyer.UNB}
\|\nu\|_{[{\rm BMO}(\partial\Omega,\sigma)]^n}<\delta
\end{equation}
for some $\delta\in(0,1)$ sufficiently small (relative to the dimension $n$ and the Ahlfors regularity constants of $\partial\Omega$).
Ahlfors regular domains satisfying \eqref{eq:Rieuyeyer.UNB} are going to be referred to as $\delta$-{\rm AR} domains (cf. Definition~\ref{def:USKT}). 
They have been introduced and studied in \cite{GHA.V} and \cite{M5} where it has been shown that condition \eqref{eq:Rieuyeyer.UNB} with 
$\delta\in(0,1)$ sufficiently small has profound implications for the geometry and topology of the set $\Omega$. Among other things, this implies that
$\Omega$ is both a connected two-sided {\rm NTA} domain in the sense of Jerison-Kenig \cite{JeKe82}, and a {\rm UR}  
domain in the sense of Hofmann-Mitrea-Taylor \cite{HoMiTa10}, whose topological boundary $\partial\Omega$ is an unbounded connected set. 
See \cite[Theorem~3.1.5, pp.\,118-119]{GHA.V} and \cite[Theorem~2.5, p.\,86]{M5}. Remarkably, it has been shown in \cite[Example~2.7, pp.\,92-95]{M5} 
that $\delta$-{\rm AR} domains can develop spiral points on the boundary. Such singularities indicate that domains in this class may not necessarily 
be locally the graphs of functions; in particular, there are $\delta$-{\rm AR} domains which fail to be Lipschitz. The interested reader is referred to 
\cite[Chapter~3]{GHA.V} and \cite[\S2.4]{M5} for a more thorough discussion. Here we only mention that condition \eqref{eq:Rieuyeyer.UNB} is satisfied 
by Lipschitz domains with small Lipschitz constant, upper-graph ${\rm BMO}_1$-domains with small constant, and chord-arc 
domains with an unbounded boundary having a small constant in the plane. Heuristically speaking, 
\begin{equation}\label{tFFf.123}
\parbox{9.40cm}{\textit{the class of $\delta$-AR domains is the sharp version, from a geometric measure
theoretic point of view, of the category of upper-graph Lipschitz domains with small Lipschitz constants.}}
\end{equation}
Specifically, for the category of $\delta$-{\rm AR} domains we prove the following theorem:

\begin{theorem}\label{u6g5rfffc-CCC.1111.UNB}
If $\Omega\subset\mathbb{R}^{n}$ is a $\delta$-{\rm AR} domain in $\mathbb{R}^{n}$ with $\delta\in(0,1)$ sufficiently small 
{\rm (}relative to $n$ and the Ahlfors regularity constants of $\partial\Omega${\rm )} and $\sigma:={\mathcal{H}}^{\,n-1}\lfloor\partial\Omega$, 
then the Hardy space $H^{1}$ on $\partial\Omega$ may be described as 
\begin{equation}\label{7ggVV.UNB}
H^{1}(\partial\Omega,\sigma)=\big\{f\in L^{1}(\partial\Omega,\sigma):\,R_jf\in L^{1}(\partial\Omega,\sigma)
\,\text{ for }\,1\leq j\leq n\big\},
\end{equation}
where $R_j$ stands for the $j$-th Riesz transform on $\partial\Omega$ {\rm (}cf. \eqref{Cau-RRj}{\rm )},  
while the John-Nirenberg space {\rm BMO} of functions of bounded mean oscillations on $\partial\Omega$ may be described as
\begin{equation}\label{utggG-TRFF.rt.A-CCC.UNB}
{\rm BMO}(\partial\Omega,\sigma)=\Big\{f_0+\sum_{j=1}^n R^{{}^{\rm mod}}_jf_j:\,
f_0,f_1,\dots,f_n\in L^\infty(\partial\Omega,\sigma)\Big\},
\end{equation}
where the modified $j$-th Riesz transform $R^{{}^{\rm mod}}_j$ is defined as in \eqref{eq:MOD.RZ.1}-\eqref{eq:MOD.RZ.2}.
\end{theorem}

The above theorem is a faithful generalization of the classical Fefferman-Stein representations in \cite{FeSt72}, 
as this reduces precisely to \eqref{u6gGVV.1} and \eqref{u6gGVV.2} in the ``flat'' case, i.e., when 
the set $\Omega$ is the upper-half space ${\mathbb{R}}^n_{+}$. Its proof is given at the end of \S\ref{S-4}.
As before, Theorem~\ref{u6g5rfffc-CCC.1111.UNB} implies that the Hardy space $H^1$ satisfies a natural ``maximality'' property, 
while the John-Nirenberg space ${\rm BMO}$ satisfies a natural ``minimality'' property, when considered on $\delta$-{\rm AR} surfaces with $\delta\in(0,1)$ 
sufficiently small. We also wish to remark that the characterization of $H^1$ given in \eqref{D+M.FEB} for families of functions
$f_1,\dots,f_n\in L^1(\partial\Omega,\sigma)$ continues to hold in the geometric setting considered in 
Theorem~\ref{u6g5rfffc-CCC.1111.UNB}. This is a new result even in the flat Euclidean setting; 
see Theorem~\ref{yrFCC-TT.aat.TRa.UNB} for a more general statement of this flavor. 

A cursory inspection of the proof of \eqref{u6gGVV.1} in \cite[(4.11), p.\,284]{GCRF85} reveals that this relies on the following 
key ingredients\footnote{the proof in \cite[Corollary~1, p.\,221]{St70} makes use of additional more specialized properties of 
harmonic functions in the upper-half space}:

(1) Convolution with a (smooth, decaying) approximation to the identity in ${\mathbb{R}}^{n-1}$ 
yields a family of linear operators uniformly bounded in $L^1({\mathbb{R}}^{n-1},{\mathcal{L}}^{n-1})$, convergent to the identity 
pointwise in $L^1({\mathbb{R}}^{n-1},{\mathcal{L}}^{n-1})$ as the scale tends to zero. 

(2) Given any $p\in(1,\infty)$, for a harmonic function in the upper-half space ${\mathbb{R}}^n_{+}$ its uniform $L^p$-integrability on all  
hyperplanes parallel to the boundary is equivalent with the $p$-th power integrability of the nontangential maximal operator of said function. 

(3) The Hardy space $H^1({\mathbb{R}}^{n-1},{\mathcal{L}}^{n-1})$ may be characterized in terms of nontangential traces of harmonic functions 
in the upper-half space ${\mathbb{R}}^n_{+}$ whose nontangential maximal operator belongs to $L^1({\mathbb{R}}^{n-1},{\mathcal{L}}^{n-1})$.

(4) The subharmonicity of certain subunitary powers of the modulus of any vector-valued function whose components satisfy the 
Generalized Cauchy-Riemann system in the upper-half space ${\mathbb{R}}^n_{+}$. 

(5) Riesz transforms map the Hardy space $H^1({\mathbb{R}}^{n-1},{\mathcal{L}}^{n-1})$ boundedly into $L^1({\mathbb{R}}^{n-1},{\mathcal{L}}^{n-1})$.

(6) For any $p\in(1,\infty)$ one can solve the $L^p$-Dirichlet boundary value problem for the Laplacian in the upper-half space ${\mathbb{R}}^n_{+}$ by 
simply convolving the boundary datum with the (classical) Poisson kernel. 

(7) Given an integrability exponent $p\in(1,\infty)$, any nonnegative subharmonic function in ${\mathbb{R}}^n_{+}$ which is uniformly $L^p$-integrable
on all hyperplanes parallel to the boundary may be dominated pointwise in ${\mathbb{R}}^n_{+}$ by the convolution of the Poisson kernel with a 
certain function in $L^p({\mathbb{R}}^{n-1},{\mathcal{L}}^{n-1})$.

Replacing the upper-half space ${\mathbb{R}}^n_{+}$ with an open set $\Omega\subseteq{\mathbb{R}}^n$, on which only general quantitative conditions 
(of a geometric measure theoretic nature) are imposed on its boundary, renders the tools described in (1)-(3) above hopelessly ineffective. 
Fortunately, property (4) survives intact, as this is of a purely algebraic/PDE (hence non-geometric) nature. 
Property (5) points to $\partial\Omega$ as a favorable geometric environment for singular integral operators of Calder\'on-Zygmund type. 
From the seminal work of G.~David and S.~Semmes \cite{DS1991} we therefore know that uniform rectifiability of $\partial\Omega$ is going to 
be a {\it sine qua non} condition.
As far as property (6) is concerned, one now has to circumvent the use of the Poisson kernel and find alternative means of solving the 
$L^p$-Dirichlet boundary value problem for the Laplacian in the set $\Omega$. From the breakthrough work of Dahlberg-Kenig in \cite{DaKe87} we know that in 
generic Lipschitz domains this restricts the integrability exponent $p$ to the interval $(2-\varepsilon,\infty)$ for some small 
$\varepsilon=\varepsilon(\Omega)\in(0,1)$. However, we presently require $p$ to be the reciprocal of the subunitary exponent appearing in (4), 
and this forces $p$ to stay close to $1$. The pioneering work of Fabes-Jodeit-Rivi\`ere \cite{FJR} guarantees that 
this may be accommodated by demanding that $\Omega$ is a bounded domain of class ${\mathscr{C}}^1$. Alternatively, one may ask 
for $\Omega$ to be an upper-graph Lipschitz domain with a suitably small Lipschitz constant. In view of \eqref{eq:sffs.WACO.D+M} and \eqref{tFFf.123}, 
it thus becomes apparent that focusing on infinitesimally flat {\rm AR} domains (withing the category of open sets with compact boundaries), 
as well as on $\delta$-{\rm AR} domains with $\delta\in(0,1)$ suitably small (withing the category of open sets with unbounded boundaries) offers a
fighting chance. Reassuringly, all these domains have uniformly rectifiable boundaries, so the prospect of having a version of (5) remains in play. 
We ultimately succeed, with a substantial amount of work going into salvaging a usable version of (7); see Proposition~\ref{yrFCC} and
Proposition~\ref{yrFCC.UND}. At the core of our approach is the ability to solve the $L^p$-Dirichlet boundary value problem for the Laplacian for any given
exponent $p\in(1,\infty)$ in any infinitesimally flat {\rm AR} domain, as well as in any $\delta$-{\rm AR} domain by appropriately restricting the size of the
parameter $\delta$, controlling the oscillations of the outward unit normal; see \cite[Chapter~8]{GHA.V}. 
In the bigger picture, the work undertaken here contributes a new perspective to the timeless question of determining the type of geometry 
guaranteeing a desired analytical property. 

\vskip 0.08in
\noindent{\it Acknowledgments.} The authors thank the referee for their insightful report, which has improved the article.

\section{Background Material}
\setcounter{equation}{0}
\label{S-2}

Throughout, ${\mathcal{L}}^n$ stands for the $n$-dimensional Lebesgue measure in ${\mathbb{R}}^n$, and 
${\mathcal{H}}^{n-1}$ denotes the $(n-1)$-dimensional Hausdorff measure in ${\mathbb{R}}^n$.\
We begin by recalling the following useful special family of cutoff functions constructed  
in \cite[Lemma~6.1.2, p.\,496]{GHA.I}.

\begin{lemma}\label{LhkC}
Let $\Omega$ be an open, nonempty, proper subset of ${\mathbb{R}}^n$, and for each
$\varepsilon>0$ introduce the {\rm (}open, one-sided{\rm )} collar neighborhood 
${\mathcal{O}}_{\,\varepsilon}$ of $\partial\Omega$ by setting
\begin{equation}\label{Tay-3}
{\mathcal{O}}_{\varepsilon}:=\big\{x\in\Omega:\,{\rm dist}\,(x,\partial\Omega)<\varepsilon\big\}.
\end{equation}

Then there exist a number $N>1$ and a family of functions $\{\Phi_{\varepsilon}\}_{\varepsilon>0}$ 
satisfying the following properties for each $\varepsilon>0$:
\begin{align}\label{VXT-12}
&\Phi_\varepsilon\in{\mathscr{C}}^{\,\infty}(\Omega),\,\,
{\rm supp}\,\,\Phi_\varepsilon\subseteq\Omega\setminus{\mathcal{O}}_{\varepsilon/N},
\,\,0\leq\Phi_\varepsilon\leq 1,\,\,\Phi_\varepsilon\equiv 1\,\text{ on }\,\Omega\setminus{\mathcal{O}}_{\varepsilon},
\\[4pt]
&\text{and }\,\sup_{x\in\Omega}\big|(\partial^\alpha\Phi_\varepsilon)(x)\big|\leq C_\alpha\,\varepsilon^{-|\alpha|}\,
\text{ for each multi-index }\,\alpha\in{\mathbb{N}}^n_0.
\label{VXT-13}
\end{align}
\end{lemma}

Let us now fix an arbitrary, open, nonempty, proper subset $\Omega$ of ${\mathbb{R}}^n$. Given any $\kappa\in(0,\infty)$, 
define the nontangential approach regions (to $\partial\Omega$ from within $\Omega$) of aperture parameter $\kappa$ by setting
\begin{equation}\label{NT-FF1}
\Gamma_{\kappa}(x):=\big\{y\in\Omega:\,|x-y|<(1+\kappa)\,{\rm dist}\,(y,\partial\Omega)\big\},\quad\forall\,x\in\partial\Omega.
\end{equation}
If $u:\Omega\to{\mathbb{R}}$ is an arbitrary Lebesgue measurable function 
define the {\tt nontangential} {\tt maximal} {\tt function} of $u$ with aperture $\kappa$ as 
\begin{equation}\label{LDG-2Rd.NNN}
{\mathcal{N}}_{\kappa}u:\partial\Omega\longrightarrow[0,+\infty],\quad
({\mathcal{N}}_{\kappa}u)(x):=\|u\|_{L^\infty(\Gamma_{\kappa}(x),{\mathcal{L}}^n)}\,
\text{ for all }\,x\in\partial\Omega.
\end{equation}
More generally, if $u:\Omega\to{\mathbb{R}}$ is a Lebesgue measurable function 
and $E\subseteq\Omega$ is a Lebesgue measurable set, we denote by ${\mathcal{N}}^E_{\kappa}u$ 
the nontangential maximal function of $u$ restricted to $E$, i.e., 
\begin{equation}\label{LDG-2AA}
\begin{array}{c}
{\mathcal{N}}^E_{\kappa}u:\partial\Omega\longrightarrow [0,+\infty]\,\,\text{ defined as}
\\[6pt]
({\mathcal{N}}^E_{\kappa}u)(x):=\|u\|_{L^\infty(\Gamma_{\kappa}(x)\cap E,{\mathcal{L}}^n)}\,\,\text{ for each }\,\,x\in\partial\Omega.
\end{array}
\end{equation}
Also, given any Lebesgue measurable function $u:\Omega\to{\mathbb{R}}$, we agree to abbreviate 
\begin{equation}\label{LDG-2Rd.NNN.eppp}
{\mathcal{N}}^\varepsilon_{\kappa}u:={\mathcal{N}}^{{\mathcal{O}}_{\varepsilon}}_{\kappa}u\,\,\text{ for each $\varepsilon>0$.}
\end{equation}

Continue to assume that $\Omega$ is an arbitrary open subset of ${\mathbb{R}}^n$ and for each $\kappa\in(0,\infty)$ define
\begin{equation}\label{UHha8yt4-1}
A_\kappa(\partial\Omega):=\big\{x\in\partial\Omega:\,x\in\overline{\Gamma_\kappa(x)}\,\big\},
\end{equation}
where the bar denotes topological closure. Informally, $A_\kappa(\partial\Omega)$ consists of those boundary 
points which are ``accessible'' in a nontangential fashion (specifically, from within nontangential 
approach regions with aperture $\kappa$). Following \cite[Definition~8.8.5, p.\,781]{GHA.I}, we make the following definition.

\begin{definition}\label{TdaBBa-MM}
Given a nonempty, open, proper subset $\Omega$ of ${\mathbb{R}}^n$, define its nontangentially accessible 
boundary as
\begin{equation}\label{ACC.ann-1}
\partial_{{}_{\rm nta}}\Omega:=\bigcap_{\kappa>0}A_{\kappa}(\partial\Omega)
=\big\{x\in\partial\Omega:\,x\in\overline{\Gamma_\kappa(x)}\,\text{ for each }\,\kappa>0\big\}.
\end{equation}
\end{definition}

Thus, by design, 
\begin{equation}\label{UTRFFa-ygg}
\partial_{{}_{\rm nta}}\Omega\subseteq A_\kappa(\partial\Omega)\,\,\text{ for any }\,\,\kappa>0.
\end{equation} 
The result below, implied by \cite[Proposition~8.8.6, p.\,782]{GHA.I}, describes a setting in which the nontangentially 
accessible boundary covers, up to an ${\mathcal{H}}^{n-1}$-null-set, the geometric measure theoretic 
boundary. Before stating it, recall that a closed set $\Sigma\subseteq{\mathbb{R}}^n$ is called an Ahlfors regular set if there exists a constant 
$C\in(1,\infty)$ with the property that for each $x\in\Sigma$ and each $r\in\big(0\,,\,2\,{\rm diam}\,\Sigma\big)$ one has
\begin{equation}\label{aamsnfsoidq.U}
C^{-1}r^{n-1}\leq\mathcal{H}^{n-1}\big(\Sigma\cap B(x,r)\big)\leq C r^{n-1}.
\end{equation}

\begin{proposition}\label{TdaBBa-DD}
If $\Omega$ is an open nonempty proper subset of ${\mathbb{R}}^n$ with the property that $\partial\Omega$ 
is an Ahlfors regular set, then 
\begin{equation}\label{UTRFFa}
{\mathcal{H}}^{n-1}\big(\partial_\ast\Omega\setminus\partial_{{}_{\rm nta}}\Omega\big)=0.
\end{equation} 
\end{proposition}

Next, we review the notion of nontangential boundary limit (or trace), following the exposition in \cite[\S8.9, p.\,786]{GHA.I}.

\begin{definition}\label{Dntl-1aa}
Let $\Omega\subset{\mathbb{R}}^n$ be an open set, and let $u$ be a Lebesgue measurable function defined ${\mathcal{L}}^n$-a.e. 
in $\Omega$. Fix an aperture parameter $\kappa>0$ and consider a point $x\in A_\kappa(\partial\Omega)$.
In this setting, whenever there exists a Lebesgue measurable set $N(x)\subset\Gamma_\kappa(x)$ 
with ${\mathcal{L}}^n(N(x))=0$ and such that the limit 
\begin{equation}\label{van-hTF}
\lim\limits_{(\Gamma_\kappa(x)\setminus N(x))\ni y\to x}u(y)
\end{equation}
exists in ${\mathbb{C}}^M$, the actual value of the limit is denoted by 
$\big(u\big|^{{}^{\kappa-{\rm n.t.}}}_{\partial\Omega}\big)(x)$.  
\end{definition}

In the context of the above definition, whenever $x\in A_\kappa(\partial\Omega)$ and the $\kappa$-nontangential limit
of $u$ at $x$ exists, it follows that for each $\varepsilon>0$ we have
\begin{equation}\label{abxi-gEE.u6f}
\left|\big(u\big|^{{}^{\kappa-{\rm n.t.}}}_{\partial\Omega}\big)(x)\right|
\leq\big({\mathcal{N}}^{\varepsilon}_{\kappa}u\big)(x)\leq\big({\mathcal{N}}_{\kappa}u\big)(x). 
\end{equation}

The following result is implied by combining \cite[Proposition~8.8.6, p.\,782]{GHA.I}, 
\cite[Lemma~8.9.2, p.\,789]{GHA.I}, and \cite[Proposition~8.9.5, pp.\,793--794]{GHA.I}.

\begin{proposition}\label{nont-ind-11-PP}
Let $\Omega\subseteq{\mathbb{R}}^n$ be an open set with a lower Ahlfors regular topological boundary satisfying 
$\mathcal{H}^{n-1}(\partial\Omega\setminus\partial_\ast\Omega)=0$, and such that 
$\sigma:={\mathcal{H}}^{n-1}\lfloor\partial\Omega$ is a doubling measure. 
Fix $\kappa>0$ and suppose $u:\Omega\to{\mathbb{R}}$ is a Lebesgue measurable function with the property that
\begin{equation}\label{u7hhb-i55-5-PP}
\parbox{7.60cm}{the nontangential limit $\big(u\big|^{{}^{\kappa-{\rm n.t.}}}_{\partial\Omega}\big)(x)$
exists {\rm(}in ${\mathbb{R}}${\rm )} for $\sigma$-a.e. point $x$ belonging to the set $\partial\Omega$.}
\end{equation}

Then 
\begin{equation}\label{u7hhb}
\parbox{8.20cm}{$f:=u\big|^{{}^{\kappa-{\rm n.t.}}}_{\partial\Omega}$ is a $\sigma$-measurable function on $\partial\Omega$, and
$\big({\mathcal{N}}_{\kappa}^{\,\varepsilon}u\big)(x)\rightarrow|f(x)|$ as $\varepsilon\to 0^{+}$ 
for $\sigma$-a.e. $x\in\partial\Omega$.}
\end{equation}
Moreover,  
\begin{equation}\label{nkc-CC-PP}
\parbox{8.70cm}{if ${\mathcal{N}}^{\delta}_{\kappa}u\in L^p(\partial\Omega,\sigma)$ 
for some $p\in(0,\infty)$ and $\delta>0$, then $f\in L^p(\partial\Omega,\sigma)$
and ${\mathcal{N}}_{\kappa}^{\varepsilon}u\longrightarrow|f|$ in 
$L^p(\partial\Omega,\sigma)$ as $\varepsilon\to 0^{+}$.}
\end{equation}
\end{proposition}

Here is another useful result (implied by \cite[Proposition~8.6.10, p.\,745]{GHA.I}):

\begin{proposition}\label{kiGa-615}
Fix $n\in{\mathbb{N}}$ and suppose $\Omega$ is an open nonempty proper subset of ${\mathbb{R}}^n$ with the property that  
$\partial\Omega$ is an Ahlfors regular set. Also, consider the measure $\sigma:={\mathcal{H}}^{n-1}\lfloor{\partial\Omega}$. 
Then for each $p\in(0,\infty)$ and $\kappa\in(0,\infty)$ there exists some 
$C=C(\partial\Omega,n,\kappa,p)\in(0,\infty)$ depending only on $n$, $\kappa$, $p$, 
and the Ahlfors regularity constants of $\partial\Omega$, with the property that, if 
\begin{align}\label{TWRUiga-3}
\varepsilon\in\big(0,\varepsilon_{\Omega,\kappa}\big)\,\,\text{ with }\,\,
\varepsilon_{\Omega,\kappa}:=\frac{{\rm diam}\,(\partial\Omega)}{n(2+\sqrt{n}\,)(3+2\kappa)}\in(0,+\infty],
\end{align}
for each ${\mathcal{L}}^n$-measurable function $u:\Omega\to{\mathbb{C}}$ one has
\begin{equation}\label{Tay-5}
\Big(\int_{{\mathcal{O}}_{\,\varepsilon}}|u|^p\,d{\mathcal{L}}^n\Big)^{1/p}\leq C\,\varepsilon^{1/p}\cdot
\big\|{\mathcal{N}}_{\kappa}^{\varepsilon}u\big\|_{L^p(\partial\Omega,\sigma)}.
\end{equation}
\end{proposition}

Given a Lebesgue measurable set $\Omega\subset\mathbb{R}^n$, define its geometric measure theoretic boundary as
\begin{align}\label{aakudgha.skdn}
&\hskip -0.10in
\partial_\ast\Omega:=\Big\{x\in\mathbb{R}^n:\,\limsup_{r\rightarrow 0^{+}}
\frac{\mathcal{L}^n\big(\Omega\cap B(x,r)\big)}{r^n}> 0\text{ and } 
\\[4pt]
&\hskip 1.20in
\limsup_{r\rightarrow 0^{+}} 
\frac{\mathcal{L}^n\big((\mathbb{R}^n\setminus\Omega)\cap B(x,r)\big)}{r^n}>0\Big\}.
\nonumber
\end{align}
Also, call $\Omega\subset\mathbb{R}^n$ a set of locally finite perimeter if 
$\Omega$ is Lebesgue measurable and
\begin{align}\label{aakudgha.skdn.2}
{\mathcal{H}}^{n-1}(\partial_\ast\Omega\cap K)<+\infty
\,\,\text{ for each compact }\,\,K\subset\mathbb{R}^n.
\end{align}

A central result in Geometric Measure Theory asserts that if $\Omega\subset\mathbb{R}^n$ 
is a set of locally finite perimeter then for ${\mathcal{H}}^{n-1}$-a.e. point $x\in\partial_\ast\Omega$ 
there exists a unique vector $\nu(x)\in S^{n-1}$, henceforth referred to as
the geometric measure theoretic outward unit normal to $\Omega$, with the property that 
\begin{align}\label{UN-1}
\lim\limits_{r\rightarrow 0^{+}}\,&
\frac{{\mathcal{L}}^n\big(B(x,r)\cap\{y\in\Omega:\,(y-x)\cdot\nu(x)>0\}\big)}
{{\mathcal{L}}^n\big(B(x,r)\big)}=0\mbox{ and }
\nonumber\\[6pt]
\lim\limits_{r\rightarrow 0^{+}}\,&
\frac{{\mathcal{L}}^n\big(B(x,r)\cap 
\{y\in{\mathbb{R}}^n\setminus\Omega:\,(y-x)\cdot\nu(x)<0\}\big)}{{\mathcal{L}}^n\big(B(x,r)\big)}=0
\end{align}
and, with $\mathbf{1}_\Omega$ denoting the characteristic function of $\Omega$, 
\begin{align}\label{UN-2}
\nabla\mathbf{1}_\Omega=-\nu\,{\mathcal{H}}^{n-1}\lfloor\partial_\ast\Omega
\end{align}
where the gradient of the locally integrable function $\mathbf{1}_\Omega$ is considered 
in the sense of distributions. In particular, unraveling definitions yields the following 
version of the Divergence Theorem due to De Giorgi-Federer (cf. \cite{Fed96}, \cite{EvGa92}).

\begin{theorem}\label{WRTP22}
Assume $\Omega\subseteq{\mathbb{R}}^n$ is a set of locally finite perimeter. Denote 
by $\nu$ its geometric measure theoretic outward unit normal and abbreviate
$\sigma:={\mathcal{H}}^{n-1}\lfloor\partial\Omega$. Then for each vector field 
$\vec{F}\in\big[{\mathscr{C}}_c^1({\mathbb{R}}^n)\big]^n$ one has
\begin{equation}\label{WRTP22-1}
\int_\Omega\big({\rm div}\,\vec{F}\big)\big|_{\Omega}\,d{\mathcal{L}}^n
=\int_{\partial_\ast\Omega}\nu\cdot\big(\vec{F}\big|_{\partial_\ast\Omega}\big)\,d\sigma.
\end{equation}
\end{theorem}

A version of Theorem~\ref{WRTP22} for singular vector fields has been established in \cite[\S1]{GHA.I}. 
To state it, the reader is reminded that, given an open set $\Omega\subseteq{\mathbb{R}}^n$, 
by ${\mathcal{D}}'(\Omega)$ we denote the space of distributions in $\Omega$, and we let 
${\mathscr{E}}'(\Omega)$ stand for the subspace of ${\mathcal{D}}'(\Omega)$ consisting of 
those distributions in $\Omega$ which have compact support. We shall also let ${\mathscr{C}}^{\infty}_b(\Omega)$ 
stand for the space of smooth and bounded functions in $\Omega$ and denote by  
$\left({\mathscr{C}}^{\,\infty}_b(\Omega)\right)^\ast$ its algebraic dual (for more on this see \cite[\S4.6, p.\,329]{GHA.I}).

\begin{theorem}\label{BSA-V42-M+D}
Fix $n\in{\mathbb{N}}$ and let $\Omega$ be an open nonempty bounded subset of ${\mathbb{R}}^n$ with an Ahlfors regular boundary. 
In particular, $\Omega$ is a set of locally finite perimeter and, if $\sigma:={\mathcal{H}}^{n-1}\lfloor\partial\Omega$, then  
its geometric measure theoretic outward unit normal $\nu$ is defined $\sigma$-a.e. on $\partial_\ast\Omega$ {\rm (}which, up to 
a $\sigma$-nullset, is contained in $\partial_{{}_{\rm nta}}\Omega${\rm )}. Fix $\kappa\in(0,\infty)$ and assume that the vector 
field 
\begin{equation}\label{PDS-bv67-M+ahb}
\vec{F}=(F_1,\dots,F_n)\in\big[{\mathcal{D}}'(\Omega)\big]^n
\end{equation}
satisfies the following conditions:
\begin{equation}\label{PDS-bv67-M+D}
\begin{array}{c}
\text{there exists a compact set $K$ contained in $\Omega$ such that}
\\[6pt]
\vec{F}\big|_{\Omega\setminus K}\in\big[L^1_{\,\rm loc}(\Omega\setminus K,{\mathcal{L}}^n)\big]^n
\,\,\text{ and }\,\,
{\mathcal{N}}^{\Omega\setminus K}_{\kappa}\vec{F}\in L^1(\partial\Omega,\sigma),
\end{array}
\end{equation}
the pointwise nontangential boundary trace 
\begin{equation}\label{PDS-bv67-M+D.ihf}
\vec{F}\big|^{{}^{\kappa-{\rm n.t.}}}_{\partial\Omega}
=\Big(F_1\big|^{{}^{\kappa-{\rm n.t.}}}_{\partial\Omega},\dots,
F_n\big|^{{}^{\kappa-{\rm n.t.}}}_{\partial\Omega}\Big)
\,\text{ exists {\rm (}in ${\mathbb{C}}^n${\rm )} $\sigma$-a.e. on }\,\,\partial_{{}_{\rm nta}}\Omega
\end{equation}
and the divergence of $\vec{F}$, considered in the sense of distributions in $\Omega$, is 
the sum {\rm (}in ${\mathcal{D}}'(\Omega)${\rm )} of an absolutely integrable function in $\Omega$  
and a compactly supported distribution in $\Omega$, i.e., 
\begin{equation}\label{PDS-bv67-M+D.2}
{\rm div}\,\vec{F}\in L^1(\Omega,{\mathcal{L}}^n)+{\mathscr{E}}'(\Omega)\subseteq\left({\mathscr{C}}^{\,\infty}_b(\Omega)\right)^\ast.
\end{equation}

Then for any $\kappa'>0$ the nontangential trace $\vec{F}\big|^{{}^{\kappa'-{\rm n.t.}}}_{\partial\Omega}$ 
exists $\sigma$-a.e. on $\partial_{{}_{\rm nta}}\Omega$ and is actually independent of $\kappa'$.
Also, with the dependence on the aperture parameter dropped, one has
\begin{equation}\label{BDS-Bvx-M+D}
{}_{\left({\mathscr{C}}^{\infty}_b(\Omega)\right)^\ast}
\big({\rm div}\,\vec{F}\,,\,1\big){}_{{\mathscr{C}}^{\infty}_b(\Omega)}
=\int_{\partial_\ast\Omega}\nu\cdot
\big(\vec{F}\,\big|^{{}^{\rm n.t.}}_{\partial\Omega}\big)\,d\sigma.
\end{equation}
\end{theorem}

Moving on, following \cite{HoMiTa10}, we make the following definition.

\begin{definition}\label{def:AR}
Call $\Omega\subset\mathbb{R}^n$ an Ahlfors regular domain provided $\Omega$ is an open set with the property that 
$\partial\Omega$ is an Ahlfors regular set and 
\begin{equation}\label{7ggg}
\mathcal{H}^{n-1}(\partial\Omega\setminus\partial_\ast\Omega)=0.
\end{equation}
\end{definition}

Since any Ahlfors regular domain $\Omega\subset\mathbb{R}^n$ is a set of locally finite perimeter, the condition 
$\mathcal{H}^{n-1}(\partial\Omega\setminus\partial_\ast\Omega)=0$, stipulated in Definition~\ref{def:AR}, 
simply ensures that the geometric measure theoretic outward unit normal $\nu$ to $\Omega$ is defined at 
${\mathcal{H}}^{n-1}$-a.e. point of the topological boundary $\partial\Omega$.

We continue by introducing the notion of uniform rectifiability of G.~David and S.~Semmes. 
The following is a slight variant of the original definition in \cite{DS1991}. 

\begin{definition}\label{Def-unif.rect}
Call $\Sigma\subset{\mathbb{R}}^{n}$ a uniformly rectifiable {\rm (}{\rm UR}{\rm )} set
provided $\Sigma$ is closed, Ahlfors regular, and has Big Pieces of Lipschitz Images. 
The latter property signifies the existence of $\varepsilon>0$ and $M\in(0,\infty)$ such that, 
for each location $x\in\Sigma$ and each scale $r\in(0,{\rm diam}\,\Sigma)$, one can find a 
Lipschitz map $\Phi:B_{n-1}(0',r)\rightarrow{\mathbb{R}}^{n}$ {\rm (}as before, $B_{n-1}(0',r{\rm )}$ stands for the 
$(n-1)$-dimensional ball of radius $r$ centered at the origin $0'$ in ${\mathbb{R}}^{n-1}$), 
having Lipschitz constant $\leq M$, and with the property that
\begin{equation}\label{3.1.9a}
{\mathcal{H}}^{n-1}\Big(\Sigma\cap B(x,r)\cap\Phi\big(B_{n-1}(0',r)\big)\Big)\geq\varepsilon r^{n-1}.
\end{equation}
\end{definition}

In general, {\rm UR} sets can be quite wild, e.g., may have infinitely many spirals, 
holes, or handles, though not without certain restrictions. 
According to a deep result of G.~David and S.~Semmes (cf. \cite[Theorem, pp.\,10-14]{DS1991}),
\begin{equation}\label{u65f65i}
\parbox{9.70cm}{given a closed set $\Sigma\subseteq{\mathbb{R}}^n$ which is Ahlfors regular,
then $\Sigma$ is a {\rm UR} set if and only if the truncated 
singular integral operator $T_{k,\varepsilon}f(x):=\int_{y\in\Sigma\setminus B(x,\varepsilon)}
k(x-y)f(y)\,{\mathcal{H}}^{n-1}(y)$, for $x\in\Sigma$, is bounded on $L^2(\Sigma,{\mathcal{H}}^{n-1})$ 
with norm majorized by a constant independent of $\varepsilon>0$, if the kernel 
$k\in{\mathscr{C}}^{\infty}({\mathbb{R}}^n\setminus\{0\})$ is odd and satisfies
$\sup_{x\in{\mathbb{R}}^n\setminus\{0\}}|x|^{n-1+\ell}|(\nabla^\ell k)(x)|<+\infty$ 
for all $\ell\in{\mathbb{N}}_0$.}
\end{equation}
In a nutshell, the category of {\rm UR} sets constitutes the most general geometric setting where 
the theory of Calder\'on-Zygmund operators is comparable in power and scope with its classical 
counterpart in the entire Euclidean setting. 

As in \cite{HoMiTa10}, we also make the following definition.

\begin{definition}\label{def:UR}
Call $\Omega\subset\mathbb{R}^n$ a {\rm UR} domain provided $\Omega$ is an Ahlfors regular domain
{\rm (}in the sense of Definition~\ref{def:AR}{\rm )} and $\partial\Omega$ is a {\rm UR} set {\rm (}in the sense of Definition~\ref{Def-unif.rect}{\rm )}.
\end{definition}

More plainly, $\Omega$ is a {\rm UR} domain provided $\Omega$ is an open subset of $\mathbb{R}^n$ whose topological boundary is 
a {\rm UR} set, and whose geometric measure theoretic boundary $\partial_\ast\Omega$ has full ($(n-1)$-dimensional Hausdorff) measure in $\partial\Omega$, 
i.e., $\mathcal{H}^{n-1}(\partial\Omega\setminus\partial_\ast\Omega)=0$. Since any {\rm UR} domain $\Omega\subset\mathbb{R}^n$ 
is a set of locally finite perimeter, the latter condition simply ensures that the geometric measure theoretic outward unit normal 
$\nu$ to $\Omega$ is defined at ${\mathcal{H}}^{n-1}$-a.e. point of the topological boundary $\partial\Omega$.

The following definition (adopted from \cite[Definition~3.4.1, p.\,158]{GHA.V}) is central to our work.

\begin{definition}\label{t7aab-RFCa.WACO.1}
Fix $n\in{\mathbb{N}}$ with $n\geq 2$ and consider an Ahlfors regular domain $\Omega\subseteq{\mathbb{R}}^n$ 
with compact boundary. Abbreviate $\sigma:={\mathcal{H}}^{\,n-1}\lfloor\partial\Omega$ and denote by $\nu$ 
the geometric measure theoretic outward unit normal to $\Omega$. Also, let ${\rm VMO}(\partial\Omega,\sigma)$ stand for 
Sarason's space of functions with vanishing mean oscillations on $\partial\Omega$. Call $\Omega$ an 
{\tt infinitesimally} {\tt flat} {\rm AR} {\tt domain} {\rm (}or {\tt inf}-{\tt flat} {\rm AR}, for short{\rm )} 
provided\footnote{i.e., the Gauss map has vanishing mean oscillations}
\begin{equation}\label{eq:614151r.WACO.b.DIDA.sNEW.DRESS.2}
\nu\in\big[{\rm VMO}(\partial\Omega,\sigma)\big]^n.
\end{equation}
\end{definition}

We emphasize that, by definition, any infinitesimally flat {\rm AR} domain is an open set with compact boundary, 
so it is either bounded, or an exterior domain (i.e., the complement of a compact set).
One of the relevant aspects of Definition~\ref{t7aab-RFCa.WACO.1} is that, for an open set with compact boundary, 
being an infinitesimally flat {\rm AR} domain says much about its geometry. This is made precise in the theorem below, 
which is implied by \cite[Theorem~3.4.2, pp.\,160--161]{GHA.V}. 

\begin{theorem}\label{6trrf.TT.ccc.WACO.222.NEW.AR}
Suppose $\Omega\subseteq{\mathbb{R}}^n$ is an infinitesimally flat {\rm AR} domain. 
Then $\Omega$ is a two-sided {\rm NTA} domain, in the sense of Jerison-Kenig {\rm\cite{JeKe82}}.
Moreover, the set $\Omega$ is a {\rm UR} domain, and $\Omega$ satisfies a two-sided local John condition 
{\rm (}in the sense of {\rm\cite[Definition~3.12, p.\,2634]{HoMiTa10}}{\rm )}.  
In addition, $\Omega$ has finitely many connected components, which are separated {\rm (}i.e., have mutually disjoint 
closures{\rm )}, and the same is true for ${\mathbb{R}}^n\setminus\overline{\Omega}$. In particular, 
any connected component of $\Omega$ and ${\mathbb{R}}^n\setminus\overline{\Omega}$ is itself a two-sided {\rm NTA} domain. 
Finally, ${\mathbb{R}}^n\setminus\overline{\Omega}$ is also an infinitesimally flat {\rm AR} domain, sharing a common topological 
boundary with $\Omega$, and whose geometric measure theoretic outward unit normal is the opposite of that for $\Omega$.
\end{theorem}

Moving on, consider an arbitrary $N\times M$ homogeneous first-order system with constant complex coefficients 
in ${\mathbb{R}}^n$ (where $N,M\in{\mathbb{N}}$ are arbitrary)
\begin{equation}\label{u6TGFaf-y65t.1-HHH-FAT.DDD}
D=\Big(\sum_{j=1}^n a^{\alpha\beta}_j\partial_j\Big)_{\substack{1\leq\alpha\leq N\\ 1\leq\beta\leq M}}
\end{equation} 
and recall that its principal symbol is defined as the $N\times M$ matrix
\begin{equation}\label{u6TGFaf-y65t.1SSS-HHH-FAT}
{\rm Sym}\,(D\,;\xi):=i\Big(\sum_{j=1}^n a^{\alpha\beta}_j\xi_j\Big)_{\substack{1\leq\alpha\leq N\\ 1\leq\beta\leq M}}
\,\text{ for }\,\xi=(\xi_1,\dots,\xi_n)\in{\mathbb{R}}^n.
\end{equation}
Then $D$ is said to be {\tt injectively} {\tt elliptic} provided the mapping
\begin{equation}\label{u6TGFaf-y65t.1SSS-HHH-FAT.2}
{\rm Sym}\,(D\,;\xi):{\mathbb{C}}^M\longrightarrow{\mathbb{C}}^N
\,\,\text{ is injective whenever }\,\,\xi\in{\mathbb{R}}^n\setminus\{0\}.
\end{equation}

This notion is relevant in the context of the Fatou-type result recalled below from \cite[Theorem~3.1.6, pp.\,652--656]{GHA.III}.

\begin{theorem}\label{Main-T1-AAFF.HAR.2-FAT.x2}
Fix $n\in{\mathbb{N}}$ with $n\geq 2$, and suppose $\Omega\subseteq{\mathbb{R}}^{n}$ is an 
arbitrary {\rm UR} domain. Abbreviate $\sigma:={\mathcal{H}}^{\,n-1}\lfloor\partial\Omega$ 
and denote by $\nu$ the geometric measure theoretic outward unit normal to $\Omega$. 
Next, consider an injectively elliptic homogeneous first-order $N\times M$ system $D$ with 
constant complex coefficients in ${\mathbb{R}}^n$ {\rm (}where $N,M\in{\mathbb{N}}${\rm )}, 
and suppose $u:\Omega\to{\mathbb{C}}^M$ is a vector-valued function with Lebesgue measurable 
components satisfying, for some aperture parameter $\kappa\in(0,\infty)$,
\begin{equation}\label{u6ggBly5-HAR.HHH-FAT.x2}
\begin{array}{c}
{\mathcal{N}}_{\kappa}u\in L^p(\partial\Omega,\sigma)\,\,\text{ with }\,\,p\in\big(\tfrac{n-1}{n}\,,\infty\big),
\\[6pt]
\text{and }\,\,Du=0\,\,\text{ in }\,\,[{\mathcal{D}}\,'(\Omega)]^N.
\end{array}
\end{equation}
Then the nontangential boundary trace
\begin{equation}\label{u6ggBly5-HAR.HHH-FAT.iii.x2}
\begin{array}{c}
u\big|^{{}^{\kappa-{\rm n.t.}}}_{\partial\Omega}
\,\,\text{ exists {\rm (}in ${\mathbb{C}}^M${\rm )} at $\sigma$-a.e. point on }\,\,\partial\Omega\,\,\text{ and}
\\[4pt]
\text{the function $u\big|^{{}^{\kappa-{\rm n.t.}}}_{\partial\Omega}$ is $\sigma$-measurable on $\partial\Omega${\rm ;}}
\\[4pt]
\text{also, $u\big|^{{}^{\kappa-{\rm n.t.}}}_{\partial\Omega}$ belongs to the space 
$\big[L^p(\partial\Omega,\sigma)\big]^M$}
\\[4pt]
\text{and one has }\,\,\big\|u\big|^{{}^{\kappa-{\rm n.t.}}}_{\partial\Omega}\big\|_{[L^p(\partial\Omega,\sigma)]^M}
\leq\|{\mathcal{N}}_{\kappa}u\|_{L^p(\partial\Omega,\sigma)}.
\end{array}
\end{equation}
Moreover, for any other given aperture parameter $\kappa'\in(0,\infty)$ the nontangential 
boundary trace $u\big|^{{}^{\kappa'-{\rm n.t.}}}_{\partial\Omega}$ also exists, and agrees with 
$u\big|^{{}^{\kappa-{\rm n.t.}}}_{\partial\Omega}$, at $\sigma$-a.e. point on $\partial\Omega$. 
\end{theorem}

The {\tt Clifford} {\tt algebra} with $n$ imaginary units is
the minimal enlargement of ${\mathbb{R}}^{n}$ to a unitary real algebra $({\mathcal{C}}\!\ell_{n},+,\odot)$, 
which is not generated (as an algebra) by any proper subspace of ${\mathbb{R}}^{n}$, and such that
\begin{equation}\label{X-sqr}
x\odot x=-|x|^2\,\,\text{ for each }\,\,x\in{\mathbb{R}}^{n}\hookrightarrow{\mathcal{C}}\!\ell_{n}.
\end{equation}
This identity is equivalent to the demand that, if $\{{\mathbf{e}}_j\}_{1\leq j\leq n}$ is
the standard orthonormal basis in ${\mathbb{R}}^{n}$, then
\begin{equation}\label{im-e}
{\mathbf{e}}_j\odot{\mathbf{e}}_j=-1\,\,\text{ and }\,\,
{\mathbf{e}}_j\odot{\mathbf{e}}_k=-{\mathbf{e}}_k\odot{\mathbf{e}}_j
\,\,\text{ whenever }\,1\leq j\neq k\leq n.
\end{equation}
In particular, identifying the canonical basis $\{{\mathbf{e}}_j\}_{1\leq j\leq n}$ 
from ${\mathbb{R}}^{n}$ with the $n$ imaginary units generating ${\mathcal{C}}\!\ell_{n}$, yields the embedding
\begin{equation}\label{embed}
{\mathbb{R}}^{n}\hookrightarrow{\mathcal{C}}\!\ell_{n},\qquad
{\mathbb{R}}^{n}\ni x=(x_1,\dots,x_{n})\equiv
\sum_{j=1}^{n}x_j\,{\mathbf{e}}_j\in{\mathcal{C}}\!\ell_{n}.
\end{equation}
For more on this topic the reader is referred to \cite{BDS}, \cite{GiMu91}, \cite[\S6.4, p.\,520]{GHA.I}, \cite{Mi}.
An example of a basic differential operator in this context is offered by the Dirac operator 
\begin{equation}\label{eq:iGfva}
D:=\sum_{j=1}^n{\mathbf{e}}_j\odot\partial_j.
\end{equation}
This is injectively elliptic and acts naturally on Clifford algebra-valued functions. 
Traditionally, ${\mathcal{C}}\!\ell_{n}$-valued functions which are null-solutions of $D$ are 
called monogenic. 

Thus, the Dirac operator $D$ from \eqref{eq:iGfva} naturally presents itself as a candidate for the 
(homogeneous, constant coefficient) first-order operator with injective principal symbol intervening 
in the statement of Theorem~\ref{Main-T1-AAFF.HAR.2-FAT.x2}. For this choice, Theorem~\ref{Main-T1-AAFF.HAR.2-FAT.x2} 
then implies the following Fatou-type result:
\begin{equation}\label{iah-a8-DDir}
\parbox{9.70cm}{for any monogenic function $u$ in a {\rm UR} domain $\Omega\subseteq{\mathbb{R}}^n$ 
satisfying $\int_{\partial\Omega}\big({\mathcal{N}}_{\kappa}u\big)^p\,d{\mathcal{H}}^{n-1}<+\infty$,
with $\kappa\in(0,\infty)$ and $p\in\big(\tfrac{n-1}{n}\,,\,\infty\big)$, the trace
$u\big|^{{}^{\kappa-{\rm n.t.}}}_{\partial\Omega}$ exists in ${\mathcal{C}}\!\ell_{n}$ at 
${\mathcal{H}}^{n-1}$-a.e. point on $\partial\Omega$.}
\end{equation}

\section{Proofs of Main Results: Domains with Compact Boundaries}
\setcounter{equation}{0}
\label{S-3}
 
We begin by establishing a foundational comparison result, which refines work by E.~Fabes and C.~Kenig in \cite{FaKe81}. 
This may be thought of as a global, geometrically inclusive, version of a well-known local characterization of subharmonicity. 

\begin{proposition}\label{yrFCC}
Fix $n\in{\mathbb{N}}$, $n\geq 2$, and let $\Omega$ be an infinitesimally 
flat {\rm AR} domain in $\mathbb{R}^{n}$ {\rm (}cf. Definition~\ref{t7aab-RFCa.WACO.1}{\rm )}. 
Abbreviate $\sigma:={\mathcal{H}}^{\,n-1}\lfloor\partial\Omega$, and fix some aperture 
parameter $\kappa>0$ along with an integrability exponent $p\in(1,\infty)$. In this setting, let $u$ be 
a harmonic function in $\Omega$ satisfying ${\mathcal{N}}_{\kappa}u\in L^p(\partial\Omega,\sigma)$, and consider
a continuous subharmonic function $w$ in $\Omega$ with the property that ${\mathcal{N}}_{\kappa}w\in L^p(\partial\Omega,\sigma)$ 
and\,\footnote{the fact that the nontangential trace $u\big|^{{}^{\kappa-{\rm n.t.}}}_{\partial\Omega}$ exists 
at $\sigma$-a.e. point $\partial\Omega$ is a consequence of \cite[Proposition~5.5.3, p.\,856]{GHA.III}}
\begin{equation}\label{ih6g6gfr5}
\Big(w\big|^{{}^{\kappa-{\rm n.t.}}}_{\partial\Omega}\Big)(x)\,\text{ exists and equals }\,
\Big(u\big|^{{}^{\kappa-{\rm n.t.}}}_{\partial\Omega}\Big)(x)\,\text{ for $\sigma$-a.e. }\,x\in\partial\Omega.
\end{equation} 
In the case when $\Omega$ is an exterior domain\footnote{i.e., the complement of a compact set} also assume  
\begin{equation}\label{utrFC-XXX-2D.777.WACO.FDD}
\begin{array}{c}
\text{the function $w$ satisfies $w(x)=o(1)$ as $|x|\to\infty$, and}
\\[6pt]
\text{if $n=2$ there exists $c\in{\mathbb{R}}$ such that }\,\,
u(x)=c+O(|x|^{-1})\,\,\text{ as }\,\,|x|\to\infty,
\\[6pt]
\text{while in the case when $n\geq 3$ one assumes $u(x)=O(|x|^{2-n})$ as $|x|\to\infty$}.
\end{array}
\end{equation}

Then $w\leq u$ in $\Omega$.
\end{proposition}

\begin{proof}
The proof combines ideas from \cite[Lemma~2.2, p.\,12]{FaKe81} and \cite{GHA.I}-\cite{GHA.V}.
To get started, for each $j\in{\mathbb{N}}$ define $\Omega_j:=\{x\in\Omega:\,{\rm dist}(x,\partial\Omega)>1/j\}$. 
A standard argument involving a (nonnegative) mollifier yields a sequence $\{w_j\}_{j\in{\mathbb{N}}}$ 
such that for each $j\in{\mathbb{N}}$ the function $w_j$ belongs to ${\mathscr{C}}^{\infty}(\Omega_j)$, is subharmonic in $\Omega_j$, 
and $w_j\to w$ uniformly on compact subsets of $\Omega$ as $j\to\infty$. From the subaverage property and Taylor's formula we also conclude that
\begin{equation}\label{eq:166TYaf}
(\Delta w_j)(x)\geq 0\,\,\text{ for each $j\in{\mathbb{N}}$ and each $x\in\Omega_j$}.
\end{equation}

Next, having fixed an arbitrary point $x_0\in\Omega$,  
construct a Green function $G_{\Omega}(\cdot,x_0)$ for the Laplacian in $\Omega$ with pole at $x_0$ satisfying
\begin{equation}\label{Dir-Lap4}
\begin{array}{c}
G_{\Omega}(\cdot,x_0)\in{\mathscr{C}}^{\infty}(\Omega\setminus\{x_0\}),\quad
\Delta G_{\Omega}(\cdot,x_0)=-\delta_{x_0}\text{ in }{\mathcal{D}}'(\Omega),
\\[4pt]
G_{\Omega}(\cdot,x_0)\Big|^{{}^{\kappa-{\rm n.t.}}}_{\partial\Omega}=0,
\end{array}
\end{equation}
and
\begin{equation}\label{Dir-Lap5}
\begin{array}{c}
{\mathcal{N}}_{\kappa}^{\varepsilon}\big(\nabla G_{\Omega}(\cdot,x_0)\big),\,
{\mathcal{N}}_{\kappa}^{\varepsilon}\big(G_{\Omega}(\cdot,x_0)\big)\in L^{p'}(\partial\Omega,\sigma)
\\[4pt]
\text{whenever }\,
0<\varepsilon<\tfrac{1}{4}\,{\rm dist}\,(x_0,\partial\Omega),
\end{array}
\end{equation}
where $p'\in(1,\infty)$ is such that $1/p+1/p'=1$. To be specific, bring in the standard fundamental 
solution $E_\Delta$ for the Laplacian, defined at each $x\in{\mathbb{R}}^n\setminus\{0\}$ as 
\begin{equation}\label{j7ggGYg.EEE}
E_{\Delta}(x):=\left\{
\begin{array}{ll}
\displaystyle\frac{1}{\omega_{n-1}(2-n)}\frac{1}{|x|^{n-2}} & \text{ if }\,\,n\geq 3,
\\[18pt]
\displaystyle\frac{1}{2\pi}\ln\,|x| & \text{ if }\,\,n=2,
\end{array}
\right. 
\end{equation}
where $\omega_{n-1}$ stands for the area of the unit sphere $S^{n-1}$ in ${\mathbb{R}}^n$.
Denote by ${\mathscr{S}}$ the boundary-to-domain harmonic single layer operator on $\partial\Omega$, 
acting on each $f\in L^1(\partial\Omega,\sigma)$ by 
\begin{equation}\label{eq:SSha.1}
{\mathscr{S}}f(x):=\int_{\partial\Omega}E_{\Delta}(x-y)f(y)\,d\sigma(y)\,\,\text{ for each }\,\,x\in\Omega,
\end{equation}
and also consider its boundary-to-boundary version, i.e.,  
\begin{equation}\label{eq:SSha.2}
Sf(x):=\int_{\partial\Omega}E_{\Delta}(x-y)f(y)\,d\sigma(y)\,\,\text{ for }\,\,x\in\partial\Omega.
\end{equation}
In addition, fix an auxiliary point $x_\ast\in{\mathbb{R}}^n\setminus\overline{\Omega}$. Then, as noted in \cite[(4.3.14), p.\,794]{GHA.III}, 
if $n\geq 3$ we may take 
\begin{equation}\label{eq:SSha.3}
G_{\Omega}(\cdot,x_0):=-E_{\Delta}(\cdot-x_0)+E_{\Delta}(\cdot-x_\ast)-{\mathscr{S}}f\,\,\text{ in }\,\,\Omega,
\end{equation}
where $f\in L^{p'}(\partial\Omega,\sigma)$ is chosen so that 
\begin{equation}\label{eq:SSha.4}
Sf=-E_{\Delta}(\cdot-x_0)+E_{\Delta}(\cdot-x_\ast)\,\,\text{ on }\,\,\partial\Omega.
\end{equation}
The fact that such a function $f$ exists is a consequence of \cite[Theorem~7.1.9, pp.\,436--437]{GHA.V} (which makes essential 
use of the fact that $\Omega$ is an infinitesimally flat {\rm AR} domain). When $n=2$, in place of \eqref{eq:SSha.3}-\eqref{eq:SSha.4} we take
\begin{equation}\label{eq:SSha.5}
G_{\Omega}(\cdot,x_0):=-E_{\Delta}(\cdot-x_0)+E_{\Delta}(\cdot-x_\ast)-{\mathscr{S}}f+c\,\,\text{ in }\,\,\Omega,
\end{equation}
where $f\in L^{p'}(\partial\Omega,\sigma)$ and $c\in{\mathbb{R}}$ are chosen so that 
\begin{equation}\label{eq:SSha.6}
\int_{\partial\Omega}f\,d\sigma=0\,\,\text{ and }\,\,Sf-c=-E_{\Delta}(\cdot-x_0)+E_{\Delta}(\cdot-x_\ast)\,\,\text{ on }\,\,\partial\Omega.
\end{equation}
That such function $f$ and constant $c$ exist is guaranteed by \cite[Theorem~7.1.10, pp.\,439--440]{GHA.V} (which, again, 
makes critical use of the fact that $\Omega$ is an infinitesimally flat {\rm AR} domain). In all cases, 
thanks to the Calder\'on-Zygmund theory for the harmonic single layer operator in {\rm UR} domains (cf. Theorem~\ref{6trrf.TT.ccc.WACO.222.NEW.AR}
and \cite[Theorem~1.5.1, pp.\,86--102]{GHA.IV}) the Green function just constructed enjoys all properties listed in 
\eqref{Dir-Lap4}-\eqref{Dir-Lap5}. In addition, from \eqref{eq:SSha.3}-\eqref{eq:SSha.6}, \eqref{j7ggGYg.EEE}, 
the Mean Value Theorem, \eqref{eq:SSha.1}, and \cite[Theorem~1.5.9, pp.\,151--153]{GHA.III}
we see that in the case when $\Omega$ is an exterior domain this Green function exhibits the following behavior as $|x|\to\infty$:
\begin{equation}\label{utrFC-XXX-2D.777.WACO.m}
\begin{array}{c}
G_{\Omega}(x,x_0)=c+O(|x|^{-1})\,\,\text{ and }\,\,(\nabla_x G)(x,x_0)=O(|x|^{-2})\,\,\text{ if $n=2$},
\\[6pt]
G_{\Omega}(x,x_0)=O(|x|^{2-n})\,\,\text{ and }\,\,(\nabla_x G)(x,x_0)=O(|x|^{1-n})\,\,\text{ if $n\geq 3$}.
\end{array}
\end{equation}

From the uniqueness in the $L^{p'}$-Dirichlet Problem for the Laplacian in the infinitesimally flat {\rm AR} domain $\Omega$ 
(cf. \cite[Theorem~8.1.6, pp.\,594--596]{GHA.V}) we then see that there is only one function $G_{\Omega}(\cdot,x_0)$ as 
in \eqref{Dir-Lap4}-\eqref{Dir-Lap5} and \eqref{utrFC-XXX-2D.777.WACO.m}. In \cite[\S5]{GHA.III} a Green function 
has been constructed when $\Omega$ is bounded which satisfies all the properties listed in \eqref{Dir-Lap4} and \eqref{utrFC-XXX-2D.777.WACO.m}, 
as well as 
\begin{equation}\label{eq:UIfaf.A}
G_{\Omega}(\cdot,x_0)>0\,\,\text{ in }\,\,\Omega\setminus\{x_0\}.
\end{equation}
In light of the aforementioned uniqueness result we conclude that the Green function constructed in \eqref{eq:SSha.3}-\eqref{eq:SSha.6}
also enjoys the property in \eqref{eq:UIfaf.A} if $\Omega$ is bounded. We can extend this result to exterior domains using the Kelvin transform.
Specifically, assume now $\Omega$ is an exterior domain in ${\mathbb{R}}^n$. Via a translation, arrange matters so that $0\notin\Omega$. 
Also, fix an arbitrary pole $x_0\in\Omega$. Then 
\begin{equation}\label{eq:TYFG.A1}
\widetilde{\Omega}:=\Big\{\tfrac{x}{|x|^2}:\,x\in\Omega\Big\}\cup\{0\}
\end{equation}
is a bounded infinitesimally flat {\rm AR} domain with the property that 
\begin{equation}\label{eq:TYFG.A2}
\partial\widetilde{\Omega}:=\Big\{\tfrac{x}{|x|^2}:\,x\in\partial\Omega\Big\}\,\,\text{ and }\,\,\tfrac{x_0}{|x_0|^2}\in\widetilde{\Omega}.
\end{equation}
Define 
\begin{equation}\label{eq:TYFG.A3}
G(x):=|x_0|^{2-n}|x|^{2-n}G_{\widetilde{\Omega}}\Big(\tfrac{x}{|x|^2},\tfrac{x_0}{|x_0|^2}\Big)\,\,\text{ for each }\,\,x\in\Omega\setminus\{x_0\},
\end{equation}
where $G_{\widetilde{\Omega}}\Big(\cdot\,,\tfrac{x_0}{|x_0|^2}\Big)$ is the Green function for the Laplacian in the bounded domain $\widetilde{\Omega}$ 
with pole at $\tfrac{x_0}{|x_0|^2}\in\widetilde{\Omega}$, as constructed above. In view of \eqref{Dir-Lap4}-\eqref{Dir-Lap5}, this satisfies
\begin{align}\label{Dir-Lap4.F}
& G\in{\mathscr{C}}^{\infty}\big(\widetilde{\Omega}\setminus\{x_0/|x_0|^2\}\big),\quad
G\big|^{{}^{\kappa-{\rm n.t.}}}_{\partial\widetilde{\Omega}}=0,\,\,\text{ and}
\\[4pt]
& \widetilde{\mathcal{N}}_{\kappa}^{\varepsilon}G\,\,\text{ belongs to }\,\,
L^{p'}(\partial\widetilde{\Omega},\widetilde{\sigma})\,\,\text{ if $\varepsilon>0$ is small},
\label{Dir-Lap5.F}
\end{align}
where the objects decorated with tilde are considered in relation to $\widetilde{\Omega}$ much as their counterparts have been introduced for $\Omega$.
Moreover, since the Kelvin transform preserves harmonicity, we have $\Delta G=0$ in $\widetilde{\Omega}\setminus\{x_0/|x_0|^2\}$ and 
\eqref{utrFC-XXX-2D.777.WACO.m} implies that   
\begin{equation}\label{utrFC-XXX-2D.777.WACO.m.F}
\begin{array}{c}
\text{if $n=2$ then $\lim\limits_{x\to\infty}G(x)$ exists, hence one has}
\\[6pt]
G(x)=c+O(|x|^{-1})\,\,\text{ for some $c\in{\mathbb{R}}$ as $|x|\to\infty$, while}
\\[6pt]
\text{if $n\geq 3$ then }\,\,G(x)=O(|x|^{2-n})\,\,\text{ as }\,\,|x|\to\infty.
\end{array}
\end{equation}
Moreover, the function  
\begin{equation}\label{eq:LKHGagf6frcf.1}
h:=G_{\widetilde{\Omega}}\Big(\cdot\,,\tfrac{x_0}{|x_0|^2}\Big)+E_\Delta\Big(\cdot-\tfrac{x_0}{|x_0|^2}\Big)\,\,\text{ in }\,\,\widetilde{\Omega}
\end{equation}
is harmonic in $\widetilde{\Omega}$, hence 
\begin{equation}\label{eq:LKHGagf6frcf.1b}
\Omega\ni x\longmapsto|x|^{2-n}h\Big(\tfrac{x}{|x|^2}\Big)\in{\mathbb{R}}\,\,\text{ is a harmonic function in }\,\,\Omega.
\end{equation}
Collectively, \eqref{eq:TYFG.A3} and \eqref{eq:LKHGagf6frcf.1} imply that for each $x\in\Omega$ we have
\begin{equation}\label{eq:LKHGagf6frcf.2}
G(x)=-|x_0|^{2-n}|x|^{2-n}E_\Delta\Big(\tfrac{x}{|x|^2}-\tfrac{x_0}{|x_0|^2}\Big)+|x_0|^{2-n}\Big[|x|^{2-n}h\Big(\tfrac{x}{|x|^2}\Big)\Big].
\end{equation}
In relation to this, observe that since for each $x,y\in{\mathbb{R}}^n\setminus\{0\}$ we may compute 
\begin{align}\label{eq:jGFFaf.P1}
|x|^2|y|^2\Big|\tfrac{x}{|x|^2}-\tfrac{y}{|y|^2}\Big|^2
&=|x|^2|y|^2\Big\langle\tfrac{x}{|x|^2}-\tfrac{y}{|y|^2}\,,\,\tfrac{x}{|x|^2}-\tfrac{y}{|y|^2}\Big\rangle
\nonumber\\[4pt]
&=|x|^2|y|^2\Big\{\tfrac{1}{|x|^2}-\tfrac{2\langle x,y\rangle}{|x|^2|y|^2}+\frac{1}{|y|^2}\Big\}
\nonumber\\[4pt]
&=|x|^2-2\langle x,y\rangle+|y|^2=|x-y|^2,
\end{align}
we have $|x||y|\Big|\tfrac{x}{|x|^2}-\tfrac{y}{|y|^2}\Big|=|x-y|$, hence 
\begin{align}\label{eq:jGFFaf.P2}
&|x_0|^{2-n}|x|^{2-n}E_\Delta\Big(\tfrac{x}{|x|^2}-\tfrac{x_0}{|x_0|^2}\Big)
\\[4pt]
&\qquad=\left\{
\begin{array}{ll}
E_\Delta(x-x_0) & \text{ if }\,\,n\geq 3,
\\[18pt]
E_\Delta(x-x_0)-E_\Delta(x)-E_\Delta(x_0) & \text{ if }\,\,n=2,
\end{array}
\right. 
\nonumber
\end{align}
for each $x\in\Omega$. From \eqref{eq:LKHGagf6frcf.2}, \eqref{eq:jGFFaf.P2}, and \eqref{eq:LKHGagf6frcf.1b} we then see that 
\begin{equation}\label{eq:jGFFaf.P3}
\Delta G=-\delta_{x_0}\,\,\text{ in }\,\,{\mathcal{D}}'(\Omega).
\end{equation}
On account of \eqref{Dir-Lap4.F}, \eqref{Dir-Lap5.F}, \eqref{utrFC-XXX-2D.777.WACO.m.F}, \eqref{eq:jGFFaf.P3}, and 
the uniqueness in the $L^{p'}$-Dirichlet Problem for the Laplacian in the infinitesimally flat {\rm AR} domain $\Omega$ 
(cf. \cite[Theorem~8.1.6, pp.\,594--596]{GHA.V}) we then conclude that $G$ defined in \eqref{eq:TYFG.A3} is indeed the 
Green function for the Laplacian with pole at $x_0$ constructed as in \eqref{eq:SSha.3}-\eqref{eq:SSha.6} in relation to the current exterior domain $\Omega$.
Based on this and \eqref{eq:UIfaf.A} (written for $\widetilde{\Omega}$ in place of $\Omega$ and $x_0/|x_0|^2$ in place of $x_0$) we then deduce that
\begin{equation}\label{eq:igegT}
G_\Omega(x,x_0)=|x_0|^{2-n}|x|^{2-n}G_{\widetilde{\Omega}}\Big(\tfrac{x}{|x|^2},\tfrac{x_0}{|x_0|^2}\Big)>0
\,\,\text{ for each }\,\,x\in\Omega\setminus\{x_0\}.
\end{equation}
This establishes \eqref{eq:UIfaf.A} in the class of unbounded infinitesimally flat {\rm AR} domains.
Simply put, 
\begin{equation}\label{eq:JHGgws.7trF}
\parbox{11.00cm}{the positivity property \eqref{eq:UIfaf.A} is true for any infinitesimally flat {\rm AR} domain $\Omega$ 
(regardless of whether the domain in question is bounded or unbounded).}
\end{equation}

For further reference, let us also note here that from Theorem~\ref{6trrf.TT.ccc.WACO.222.NEW.AR} and \cite[Proposition~8.9.17, pp.\,816-817]{GHA.I}, 
or \cite[Lemma~5.4.4, p.\,855]{GHA.III}, we know that there exist a number $\lambda\in(1,\infty)$ and an aperture parameter 
$\widetilde\kappa>\kappa$ with the property that for each truncation parameter $\varepsilon\in\big(0,{\rm dist}\,(x_0,\partial\Omega)\big)$ one has
\begin{align}\label{Int-Es-G.1-yy.2ecb.CHI}
{\mathcal{N}}^\varepsilon_{\kappa}\big(G_{\Omega}(\cdot,x_0)\big)(z)\leq\varepsilon\cdot
{\mathcal{N}}^{\lambda\varepsilon}_{\widetilde\kappa}\big(\nabla G_{\Omega}(\cdot,x_0)\big)(z)\,\,\text{ for $\sigma$-a.e. }\,\,z\in\partial\Omega.
\end{align}

At this stage, the proof branches out, depending on whether $\Omega$ is bounded, or $\Omega$ is an exterior domain.
Let us first treat the case when $\Omega$ is bounded. Recall the family of functions 
$\{\Phi_{\varepsilon}\}_{\varepsilon>0}$ introduced in Lemma~\ref{LhkC}.
Fix some $0<\varepsilon<\tfrac{1}{4}\,{\rm dist}\,(x_0,\partial\Omega)$ and for each $j\in{\mathbb{N}}$ sufficiently large
(e.g., $j>N/\varepsilon$ will do, if $N$ is as in Lemma~\ref{LhkC}) consider the vector field 
\begin{align}\label{Dir-Lap8.aaTa.1}
\vec{F}(y) &:=\Phi_\varepsilon(y)w_j(y)\nabla_yG_{\Omega}(y,x_0)+w_j(y)G_{\Omega}(y,x_0)(\nabla\Phi_\varepsilon)(y)
\nonumber\\[4pt]
&\quad -\Phi_\varepsilon(y)G_{\Omega}(y,x_0)(\nabla w_j)(y),\qquad\forall\,y\in\Omega\setminus\{x_0\}.
\end{align}
Bearing in mind \eqref{Dir-Lap4}-\eqref{Dir-Lap5}, that $\Phi_\varepsilon\in{\mathscr{C}}^{\infty}(\Omega)$,
${\rm supp}\,\,\Phi_\varepsilon\subseteq\Omega\setminus{\mathcal{O}}_{\varepsilon/N}\subseteq\Omega_j$ and 
$w_j\in{\mathscr{C}}^{\infty}(\Omega_j)$, it follows that $\vec{F}$ is well defined.
Also, since $\Phi_\varepsilon\equiv 0$ near $\partial\Omega$ we have 
\begin{align}\label{Dir-Lap8.aaTa.2}
\begin{array}{c}
\vec{F}\in\big[L^1_{\,\rm loc}(\Omega,{\mathcal{L}}^n)\big]^n,\quad 
\vec{F}\big|^{{}^{\kappa-{\rm n.t.}}}_{\partial\Omega}=0
\,\,\text{ at every point in }\,\,\partial_{{}_{\rm nta}}\Omega,
\\[6pt]
\text{and if $\rho>0$ is small then }\,\,{\mathcal{N}}_{\kappa}^{\rho}\vec{F}=0
\,\,\text{ at every point in }\,\,\partial\Omega.
\end{array}
\end{align}
Also, with the divergence computed in the sense of distributions in the bounded open set $\Omega$, 
\begin{align}\label{Dir-Lap8.aaTa.3}
{\rm div}\,\vec{F} &=2\langle\nabla G_{\Omega}(\cdot,x_0),\nabla\Phi_\varepsilon\rangle w_j
+G_{\Omega}(\cdot,x_0)(\Delta\Phi_\varepsilon)w_j
\nonumber\\[4pt]
&\quad -G_{\Omega}(\cdot,x_0)\Phi_\varepsilon(\Delta w_j)-w_j(x_0)\delta_{x_0}
\in L^1(\Omega,{\mathcal{L}}^n)+{\mathscr{E}}'(\Omega).
\end{align}
Granted these, Theorem~\ref{BSA-V42-M+D} applies and the Divergence Formula \eqref{BDS-Bvx-M+D} 
for the vector field $\vec{F}$ from \eqref{Dir-Lap8.aaTa.1} presently permits us to conclude that
\begin{align}\label{Dir-Lap8}
w_j(x_0) &=2\int_{\Omega}\langle\nabla_yG_{\Omega}(y,x_0),(\nabla\Phi_\varepsilon)(y)\rangle w_j(y)\,dy
\\[4pt]
&\quad+\int_{\Omega}G_{\Omega}(y,x_0)(\Delta\Phi_\varepsilon)(y)w_j(y)\,dy
\nonumber\\[4pt]
&\quad-\int_{\Omega}G_{\Omega}(y,x_0)\Phi_\varepsilon(y)(\Delta w_j)(y)\,dy.
\nonumber
\end{align}
Thanks to \eqref{eq:166TYaf}, \eqref{eq:JHGgws.7trF}, and the fact that $\Phi_\varepsilon\geq 0$ (cf. Lemma~\ref{LhkC}), 
this implies
\begin{align}\label{Dir-Lap8.pLa}
w_j(x_0)&\leq 2\int_{\Omega}\langle\nabla_yG_{\Omega}(y,x_0),(\nabla\Phi_\varepsilon)(y)\rangle w_j(y)\,dy
\\[4pt]
&\quad+\int_{\Omega}G_{\Omega}(y,x_0)(\Delta\Phi_\varepsilon)(y)w_j(y)\,dy
\nonumber
\end{align}
which, after passing to limit as $j\to\infty$, yields 
\begin{align}\label{Dir-Lap8.pL}
w(x_0)&\leq 2\int_{\Omega}\langle\nabla_yG_{\Omega}(y,x_0),(\nabla\Phi_\varepsilon)(y)\rangle w(y)\,dy
\\[4pt]
&\quad+\int_{\Omega}G_{\Omega}(y,x_0)(\Delta\Phi_\varepsilon)(y)w(y)\,dy.
\nonumber
\end{align}
The same reasoning that has produced \eqref{Dir-Lap8} now used with $w_j$ replaced by $u$ also gives, 
upon recalling that $\Delta u=0$ in $\Omega$, that 
\begin{align}\label{Dir-Lap8.ww}
u(x_0)&=2\int_{\Omega}\langle\nabla_yG_{\Omega}(y,x_0),(\nabla\Phi_\varepsilon)(y)\rangle u(y)\,dy
\\[4pt]
&\quad+\int_{\Omega}G_{\Omega}(y,x_0)(\Delta\Phi_\varepsilon)(y)u(y)\,dy.
\nonumber
\end{align}
Introducing
\begin{align}\label{Dir-Lap8.ww.2}
\eta:=w-u\,\,\text{ in }\,\,\Omega
\end{align}
it follows that
\begin{equation}\label{ih6g6gfr5.BBB}
{\mathcal{N}}_{\kappa}\eta\in L^p(\partial\Omega,\sigma)\,\,\text{ and }\,\,
\eta\big|^{{}^{\kappa-{\rm n.t.}}}_{\partial\Omega}=0\,\,\text{ at $\sigma$-a.e. point on }\,\,\partial\Omega.
\end{equation} 
These properties are going to be relevant shortly. For now, combining \eqref{Dir-Lap8.pL} 
with \eqref{Dir-Lap8.ww} we arrive at
\begin{align}\label{Dir-Lap8.pLb}
\eta(x_0)\leq{\rm I}_\varepsilon+{\rm II}_\varepsilon\,\,\text{ for each small }\,\,\varepsilon>0
\end{align}
where, for each $\varepsilon>0$ small, we have set
\begin{align}\label{Dir-Lap8.pL.22}
{\rm I}_\varepsilon &:=2\int_{\Omega}\langle\nabla_yG_{\Omega}(y,x_0),(\nabla\Phi_\varepsilon)(y)\rangle\eta(y)\,dy,
\\[6pt]
{\rm II}_\varepsilon &:=\int_{\Omega}G_{\Omega}(y,x_0)(\Delta\Phi_\varepsilon)(y)\eta(y)\,dy.
\label{Dir-Lap8.pL.3}
\end{align}
Upon recalling that 
${\mathcal{O}}_{\varepsilon}:=\big\{y\in\Omega:\,{\rm dist}\,(y,\partial\Omega)<\varepsilon\big\}$, 
it follows that for each $\varepsilon>0$ small we may invoke \eqref{VXT-12}-\eqref{VXT-13}, 
Proposition~\ref{kiGa-615} (with $p:=1$), and H\"older's inequality to estimate
\begin{align}\label{Dir-Lap9}
|{\rm I}_\varepsilon| & \leq\frac{C}{\varepsilon}
\int_{{\mathcal{O}}_\varepsilon}|\nabla_yG_{\Omega}(y,x_0)||\eta(y)|\,dy
\leq C\int_{\partial\Omega}{\mathcal{N}}_{\kappa}^{\varepsilon}\big(\nabla G_{\Omega}(\cdot,x_0)\eta\big)\,d\sigma
\nonumber\\[4pt]
&\leq C\int_{\partial\Omega}{\mathcal{N}}_{\kappa}^{\varepsilon}\big(\nabla G_{\Omega}(\cdot,x_0)\big)
{\mathcal{N}}_{\kappa}^\varepsilon\eta\,d\sigma
\nonumber\\[4pt]
& \leq C\big\|{\mathcal{N}}_{\kappa}^{\varepsilon}\big(\nabla G_{\Omega}(\cdot,x_0)\big)\big\|_{L^{p'}(\partial\Omega,\sigma)}
\|{\mathcal{N}}^{\varepsilon}_{\kappa}\eta\|_{L^{p}(\partial\Omega,\sigma)},
\end{align}
where $C\in(0,\infty)$ is independent of $\varepsilon$. 
Since \eqref{Dir-Lap5} and \eqref{ih6g6gfr5.BBB} imply (bearing \eqref{nkc-CC-PP} in mind)
\begin{equation}\label{Dir-Lap10}
\begin{array}{c}
\big\|{\mathcal{N}}_{\kappa}^{\varepsilon}\big(\nabla G_{\Omega}(\cdot,x_0)\big)\big\|_{L^{p'}(\partial\Omega,\sigma)}=O(1)
\,\text{ as }\,\varepsilon\to 0^{+},
\\[4pt]
\|{\mathcal{N}}_{\kappa}^{\varepsilon}\eta\|_{L^{p}(\partial\Omega,\sigma)}=o(1)\,\text{ as }\,\varepsilon\to 0^{+}, 
\end{array}
\end{equation}
this proves that
\begin{equation}\label{Dir-Lap11}
\lim_{\varepsilon\to 0^{+}}{\rm I}_\varepsilon=0.
\end{equation}
Likewise, based on \eqref{VXT-12}-\eqref{VXT-13}, Proposition~\ref{kiGa-615} (with $p:=1$), \eqref{Int-Es-G.1-yy.2ecb.CHI}, 
H\"older's inequality, and \cite[Corollary~8.4.8, (8.4.96), p.\,712]{GHA.I} we see that for each $\varepsilon>0$ small we have
\begin{align}\label{Dir-Lap9.gFD}
|{\rm II}_\varepsilon| & \leq\frac{C}{\varepsilon^2}
\int_{{\mathcal{O}}_\varepsilon}|G_{\Omega}(y,x_0)||\eta(y)|\,dy
\leq\frac{C}{\varepsilon}\int_{\partial\Omega}{\mathcal{N}}_{\kappa}^{\varepsilon}\big(G_{\Omega}(\cdot,x_0)\eta\big)\,d\sigma
\nonumber\\[4pt]
&\leq\frac{C}{\varepsilon}\int_{\partial\Omega}{\mathcal{N}}_{\kappa}^{\varepsilon}\big(G_{\Omega}(\cdot,x_0)\big)
{\mathcal{N}}_{\kappa}^\varepsilon\eta\,d\sigma
\nonumber\\[4pt]
& \leq C\int_{\partial\Omega}{\mathcal{N}}^{\lambda\varepsilon}_{\widetilde\kappa}\big(\nabla G_{\Omega}(\cdot,x_0)\big)
{\mathcal{N}}_{\kappa}^\varepsilon\eta\,d\sigma
\nonumber\\[4pt]
& \leq C\big\|{\mathcal{N}}^{\lambda\varepsilon}_{\widetilde\kappa}\big(\nabla G_{\Omega}(\cdot,x_0)\big)\big\|_{L^{p'}(\partial\Omega,\sigma)}
\|{\mathcal{N}}^{\varepsilon}_{\kappa}\eta\|_{L^{p}(\partial\Omega,\sigma)}
\nonumber\\[4pt]
& \leq C\big\|{\mathcal{N}}_{\kappa}^{\lambda\varepsilon}\big(\nabla G_{\Omega}(\cdot,x_0)\big)\big\|_{L^{p'}(\partial\Omega,\sigma)}
\|{\mathcal{N}}^{\varepsilon}_{\kappa}\eta\|_{L^{p}(\partial\Omega,\sigma)},
\end{align}
where $\lambda\in(1,\infty)$ and $\widetilde{\kappa}>\kappa$ are as in \eqref{Int-Es-G.1-yy.2ecb.CHI}, 
and where the constant $C\in(0,\infty)$ is independent of $\varepsilon$. Much as before, this ultimately permits us to conclude that
\begin{equation}\label{Dir-Lap11.BBB}
\lim_{\varepsilon\to 0^{+}}{\rm II}_\varepsilon=0.
\end{equation}
Together, \eqref{Dir-Lap8.pLb}, \eqref{Dir-Lap11}, and \eqref{Dir-Lap11.BBB} imply $\eta(x_0)\leq 0$.
Hence $w(x_0)\leq u(x_0)$, and since $x_0\in\Omega$ has been arbitrarily chosen, the desired conclusion follows in the case when $\Omega$ is bounded.

Let us now treat the case when $\Omega$ is an exterior domain. In this scenario, pick a nonnegative function $\Psi\in{\mathscr{C}}^\infty_c({\mathbb{R}}^n)$ 
with $\Psi\equiv 1$ on $B(0,1)$ and $\Psi\equiv 0$ on ${\mathbb{R}}^n\setminus B(0,2)$. Also, for each $R>0$ define $\Psi_R(x):=\Psi(x/R)$ 
at each $x\in{\mathbb{R}}^n$. The idea is to run the same argument as above with $\Phi_\varepsilon$ replaced by the product 
$\Phi_\varepsilon\cdot\Psi_R$, with $\varepsilon>0$ small and $R>0$ large. This time, in place of \eqref{Dir-Lap8.pLb}-\eqref{Dir-Lap8.pL.3} we arrive at 
\begin{align}\label{Dir-Lap8.pLb.II}
\eta(x_0)\leq{\rm I}_{\varepsilon,R}+{\rm II}_{\varepsilon,R}\,\,\text{ for each $\varepsilon>0$ small and $R>0$ large},
\end{align}
where, for each $\varepsilon>0$ small and $R>0$ large, we have set
\begin{align}\label{Dir-Lap8.pL.22.II}
{\rm I}_{\varepsilon,R} &:=2\int_{\Omega}\big\langle\nabla_yG_{\Omega}(y,x_0),\nabla(\Phi_\varepsilon\cdot\Psi_R)(y)\big\rangle\eta(y)\,dy,
\\[6pt]
{\rm II}_{\varepsilon,R} &:=\int_{\Omega}G_{\Omega}(y,x_0)\Delta(\Phi_\varepsilon\cdot\Psi_R)(y)\eta(y)\,dy.
\label{Dir-Lap8.pL.3.II}
\end{align}
Note that 
\begin{equation}\label{eq:kYGaff.1}
\nabla(\Phi_\varepsilon\cdot\Psi_R)=\Psi_R\cdot\nabla\Phi_\varepsilon+\Phi_\varepsilon\cdot\nabla\Psi_R=\nabla\Phi_\varepsilon+\nabla\Psi_R
\end{equation}
and
\begin{equation}\label{eq:kYGaff.2}
\Delta(\Phi_\varepsilon\cdot\Psi_R)=\Psi_R\cdot\Delta\Phi_\varepsilon+\Phi_\varepsilon\cdot\Delta\Psi_R
+2\langle\nabla\Phi_\varepsilon,\nabla\Psi_R\rangle=\Delta\Phi_\varepsilon+\Delta\Psi_R
\end{equation}
with the final equalities valid whenever $\varepsilon>0$ is small and $R>0$ is large. 
Our goal remains to show that $\eta(x_0)\leq 0$. The contributions of $\nabla\Phi_\varepsilon$ 
and $\Delta\Phi_\varepsilon$ in the context of ${\rm I}_{\varepsilon,R}$ and ${\rm II}_{\varepsilon,R}$
are handled as before. As for the contributions of $\nabla\Psi_R$ and $\Delta\Psi_R$ 
in the context of ${\rm I}_{\varepsilon,R}$ and ${\rm II}_{\varepsilon,R}$, in the case when $n\geq 3$ we have
\begin{align}\label{Dir-Lap8.pL.22.III}
\Big|\int_{\Omega}\big\langle &\nabla_yG_{\Omega}(y,x_0), (\nabla\Psi_R)(y)\big\rangle\eta(y)\,dy\Big|
\nonumber\\[4pt]
&\leq R^{-1}\int_{B(0,2R)\setminus B(0,R)}|\nabla_yG_{\Omega}(y,x_0)||\eta(y)|\,dy
\nonumber\\[4pt]
&\leq CR^{-n}\int_{B(0,2R)\setminus B(0,R)}|\eta(y)|\,dy
\nonumber\\[4pt]
&\leq C\fint_{B(0,2R)\setminus B(0,R)}|u|\,d{\mathcal{L}}^n+C\fint_{B(0,2R)\setminus B(0,R)}|w|\,d{\mathcal{L}}^n
\nonumber\\[4pt]
&=o(1)\,\,\text{ as }\,\,R\to\infty,
\end{align}
by virtue of the additional assumptions made in \eqref{utrFC-XXX-2D.777.WACO.FDD}, and similarly
\begin{align}\label{Dir-Lap8.pL.22.IV}
&\Big|\int_{\Omega}G_{\Omega}(y,x_0)(\Delta\Psi_R)(y)\eta(y)\,dy\Big|
\nonumber\\[4pt]
&\quad\leq R^{-2}\int_{B(0,2R)\setminus B(0,R)}|G_{\Omega}(y,x_0)||\eta(y)|\,dy
\nonumber\\[4pt]
&\quad\leq CR^{-n}\int_{B(0,2R)\setminus B(0,R)}|\eta(y)|\,dy
\nonumber\\[4pt]
&\quad\leq C\fint_{B(0,2R)\setminus B(0,R)}|u|\,d{\mathcal{L}}^n+C\fint_{B(0,2R)\setminus B(0,R)}|w|\,d{\mathcal{L}}^n
\nonumber\\[4pt]
&\quad=o(1)\,\,\text{ as }\,\,R\to\infty.
\end{align}
In an analogous fashion, when $n=2$ we have
\begin{align}\label{Dir-Lap8.pL.22.V}
&\Big|\int_{\Omega}\big\langle\nabla_yG_{\Omega}(y,x_0),(\nabla\Psi_R)(y)\big\rangle\eta(y)\,dy\Big|
\nonumber\\[4pt]
&\quad\leq R^{-1}\int_{B(0,2R)\setminus B(0,R)}|\nabla_yG_{\Omega}(y,x_0)||\eta(y)|\,dy
\nonumber\\[4pt]
&\quad\leq CR^{-3}\int_{B(0,2R)\setminus B(0,R)}|\eta(y)|\,dy
\nonumber\\[4pt]
&\quad\leq CR^{-1}\fint_{B(0,2R)\setminus B(0,R)}|u|\,d{\mathcal{L}}^2
+CR^{-1}\fint_{B(0,2R)\setminus B(0,R)}|w|\,d{\mathcal{L}}^2
\nonumber\\[4pt]
&\quad=o(1)\,\,\text{ as }\,\,R\to\infty,
\end{align}
and
\begin{align}\label{Dir-Lap8.pL.22.VI}
&\Big|\int_{\Omega}G_{\Omega}(y,x_0)(\Delta\Psi_R)(y)\eta(y)\,dy\Big|
\nonumber\\[4pt]
&\quad\leq\Big|\int_{B(0,2R)\setminus B(0,R)}[G_{\Omega}(y,x_0)-c](\Delta\Psi_R)(y)\eta(y)\,dy\Big|
\nonumber\\[4pt]
&\qquad+|c|\Big|\int_{B(0,2R)\setminus B(0,R)}(\Delta\Psi_R)(y)\eta(y)\,dy\Big|
\nonumber\\[4pt]
&\quad\leq CR^{-3}\int_{B(0,2R)\setminus B(0,R)}|u|\,d{\mathcal{L}}^2
+CR^{-3}\int_{B(0,2R)\setminus B(0,R)}|w|\,d{\mathcal{L}}^2
\nonumber\\[4pt]
&\qquad+C\Big|\int_{B(0,2R)\setminus B(0,R)}(\Delta\Psi_R)u\,d{\mathcal{L}}^2\Big|
+CR^{-2}\int_{B(0,2R)\setminus B(0,R)}|w|\,d{\mathcal{L}}^2
\nonumber\\[4pt]
&\quad\leq CR^{-1}\fint_{B(0,2R)\setminus B(0,R)}|u|\,d{\mathcal{L}}^2
+C\fint_{B(0,2R)\setminus B(0,R)}|w|\,d{\mathcal{L}}^2
\nonumber\\[4pt]
&\qquad+C\Big|\int_{B(0,2R)\setminus B(0,R)}(\Delta\Psi_R)[u-c]\,d{\mathcal{L}}^2\Big|
\nonumber\\[4pt]
&\quad\leq CR^{-1}\fint_{B(0,2R)\setminus B(0,R)}|u|\,d{\mathcal{L}}^2
+C\fint_{B(0,2R)\setminus B(0,R)}|w|\,d{\mathcal{L}}^2+CR^{-1}
\nonumber\\[4pt]
&\quad=o(1)\,\,\text{ as }\,\,R\to\infty.
\end{align}
Above, the penultimate inequality uses the fact that since $\Psi_R\in{\mathscr{C}}^\infty_c({\mathbb{R}}^n)$ is supported in 
$B(0,2R)\setminus B(0,R)$ then
\begin{equation}\label{eq:yudgg.667}
\int_{B(0,2R)\setminus B(0,R)}\Delta\Psi_R\,d{\mathcal{L}}^2=\int_{{\mathbb{R}}^n}\Delta\Psi_R\,d{\mathcal{L}}^2=0,
\end{equation}
via integration by parts. All together, 
\begin{align}\label{Dir-Lap8.pLb.IIBM}
\lim_{\varepsilon\to 0^{+}}\lim_{R\to\infty}{\rm I}_{\varepsilon,R}=0\,\,\text{ and }\,\,
\lim_{\varepsilon\to 0^{+}}\lim_{R\to\infty}{\rm II}_{\varepsilon,R}=0
\end{align}
which in concert with \eqref{Dir-Lap8.pLb.II} proves that $w(x_0)-u(x_0)=\eta(x_0)\leq 0$, hence in this case we again have 
\begin{equation}\label{eq:WACO.2025}
w(x_0)\leq u(x_0), 
\end{equation}
as wanted.
\end{proof}

Before stating our next result, we recall some conventions and notation. First, 
given an open set $\Omega\subset{\mathbb{R}}^{n}$ of locally finite perimeter with 
geometric measure theoretic outward unit normal $\nu=(\nu_1,\dots,\nu_n)$,
under the canonical embedding of ${\mathbb{R}}^{n}$ into the Clifford algebra 
${\mathcal{C}}\!\ell_{n}$ (cf. \eqref{embed}), it follows that $\nu$ may be identified with 
the ${\mathcal{C}}\!\ell_{n}$-valued function defined ${\mathcal{H}}^{\,n-1}$-a.e. on 
$\partial_\ast\Omega$ according to 
\begin{equation}\label{utggGYHNN.iii}
\nu=\nu_1{\mathbf{e}}_1+\cdots+\nu_n{\mathbf{e}}_n.
\end{equation}
Second, recall the standard Dirac operator $D=\sum_{j=1}^n{\mathbf{e}}_j\odot\partial_j$ in ${\mathbb{R}}^n$.

\begin{proposition}\label{yrFCC-TT}
Let $\Omega$ be an infinitesimally flat {\rm AR} domain in $\mathbb{R}^{n}$ and set 
$\sigma:={\mathcal{H}}^{n-1}\lfloor\partial\Omega$. Also, denote by $\nu$ the geometric measure 
theoretic outward unit normal to $\Omega$, canonically identified with a Clifford algebra-valued 
function on $\partial\Omega$ as in \eqref{utggGYHNN.iii}, and fix some aperture parameter $\kappa>0$.
In this context, suppose the function $u\in{\mathscr{C}}^{\,\infty}(\Omega)\otimes{\mathcal{C}}\!\ell_n$ 
satisfies
\begin{equation}\label{ih6gAAA}
\begin{array}{c}
Du=0\,\,\text{ in }\,\,\Omega,\quad{\mathcal{N}}_{\kappa}u\in L^{1,\infty}(\partial\Omega,\sigma),\,\,\text{ and}
\\[4pt]
\text{$u(x)=o(1)$ as $|x|\to\infty$ if $\Omega$ is an exterior domain}.
\end{array}
\end{equation} 

Then the nontangential boundary trace $u\big|^{{}^{\kappa-{\rm n.t.}}}_{\partial\Omega}$ exists 
{\rm (}in ${\mathcal{C}}\!\ell_n${\rm )} at $\sigma$-a.e. point on $\partial\Omega$ and 
\begin{align}\label{ih6g6FFF}
{\mathcal{N}}_{\kappa}u\in L^{1}(\partial\Omega,\sigma)&\Longleftrightarrow
u\big|^{{}^{\kappa-{\rm n.t.}}}_{\partial\Omega}\in L^{1}(\partial\Omega,\sigma)\otimes{\mathcal{C}}\!\ell_n
\nonumber\\[6pt]
&\Longleftrightarrow\nu\odot\big(u\big|^{{}^{\kappa-{\rm n.t.}}}_{\partial\Omega}\big)\in 
H^{1}(\partial\Omega,\sigma)\otimes{\mathcal{C}}\!\ell_n
\nonumber\\[6pt]
&\Longleftrightarrow\nu\odot\big(u\big|^{{}^{\kappa-{\rm n.t.}}}_{\partial\Omega}\big)\in 
L^{1}(\partial\Omega,\sigma)\otimes{\mathcal{C}}\!\ell_n.
\end{align} 
Moreover, if either of the above memberships materializes then
\begin{align}\label{ih6g6GGG}
&\|{\mathcal{N}}_{\kappa}u\|_{L^{1}(\partial\Omega,\sigma)}\approx
\big\|u\big|^{{}^{\kappa-{\rm n.t.}}}_{\partial\Omega}\big\|_{L^{1}(\partial\Omega,\sigma)\otimes{\mathcal{C}}\!\ell_n}
\\[6pt]
&\qquad\approx\big\|\nu\odot\big(u\big|^{{}^{\kappa-{\rm n.t.}}}_{\partial\Omega}\big)
\big\|_{H^{1}(\partial\Omega,\sigma)\otimes{\mathcal{C}}\!\ell_n}
\approx\big\|\nu\odot\big(u\big|^{{}^{\kappa-{\rm n.t.}}}_{\partial\Omega}\big)
\big\|_{L^{1}(\partial\Omega,\sigma)\otimes{\mathcal{C}}\!\ell_n}
\nonumber
\end{align} 
where the implicit proportionality constants are independent of $u$.
\end{proposition}

\begin{proof}
The fact that the nontangential boundary trace $u\big|^{{}^{\kappa-{\rm n.t.}}}_{\partial\Omega}$ 
exists (in ${\mathcal{C}}\!\ell_n$) at $\sigma$-a.e. point in $\partial\Omega$ follows 
from the Fatou-type result for monogenic functions recorded in \eqref{iah-a8-DDir}, since 
having ${\mathcal{H}}^{\,n-1}(\partial\Omega)<\infty$ permits us to invoke \cite[Lemma~6.2.4, (6.2.36), p.\,504]{GHA.I} to 
conclude from \eqref{ih6gAAA} that ${\mathcal{N}}_{\kappa}u$ belongs to $L^{p}(\partial\Omega,\sigma)$ 
whenever $p\in\big(\tfrac{n-1}{n}\,,\,1\big)$. 

Moving on, the right-pointing implication in the first line of \eqref{ih6g6FFF} and the 
right-pointing estimate in the first line of \eqref{ih6g6GGG} are clear from \eqref{abxi-gEE.u6f} 
and Proposition~\ref{nont-ind-11-PP}. Consider next the left-pointing implication in the 
first line of \eqref{ih6g6FFF} and the naturally accompanying estimate. To this end, assume 
$u\in{\mathscr{C}}^{\,\infty}(\Omega)\otimes{\mathcal{C}}\!\ell_n$ is as in \eqref{ih6gAAA} 
and, in addition,  
\begin{equation}\label{ih6g6EEE}
u\big|^{{}^{\kappa-{\rm n.t.}}}_{\partial\Omega}\in L^{1}(\partial\Omega,\sigma)\otimes{\mathcal{C}}\!\ell_n.
\end{equation} 
The fact that $u$ is monogenic (i.e., a null-solution of $D$) in $\Omega$ implies the existence of some 
purely dimensional exponent $\theta\in(0,1)$ with the property that the function $w:=|u|^\theta$ is subharmonic 
in $\Omega$ (e.g., any $\theta\in(0,1)$ satisfying $\theta\geq\tfrac{n-2}{n-1}$ will do; see \cite{SW2}, 
or \cite[Theorem~3.35, p.\,106]{GiMu91}). In addition, $w$ is continuous and nonnegative in $\Omega$, and 
\begin{equation}\label{ih6g6Fii}
{\mathcal{N}}_{\kappa}w=({\mathcal{N}}_{\kappa}u)^\theta\in L^{1/\theta,\infty}(\partial\Omega,\sigma).
\end{equation} 
Also, $w(x)=o(1)$ as $|x|\to\infty$ if $\Omega$ is an exterior domain.
To proceed, choose $p\in(1,1/\theta)$ and observe that, since $\partial\Omega$ is compact, we have
\begin{equation}\label{ih6g6Fii.b}
L^{1/\theta,\infty}(\partial\Omega,\sigma)\hookrightarrow L^{p}(\partial\Omega,\sigma).
\end{equation} 
In particular, 
\begin{equation}\label{ih6g6Fii.c}
{\mathcal{N}}_{\kappa}w\in L^{p}(\partial\Omega,\sigma).
\end{equation} 
Let us now introduce
\begin{equation}\label{ih6g6E22}
f:=w\big|^{{}^{\kappa-{\rm n.t.}}}_{\partial\Omega}
=\Big|u\big|^{{}^{\kappa-{\rm n.t.}}}_{\partial\Omega}\Big|^\theta\in L^{1/\theta}(\partial\Omega,\sigma).
\end{equation} 
Since $\Omega$ is an infinitesimally flat {\rm AR} domain and $1/\theta\in(1,\infty)$, applying \cite[Theorem~8.1.6, pp.\,594-596]{GHA.V}
with $L:=\Delta$ ensures the solvability of the Dirichlet Problem
\begin{equation}\label{Dir-Lap1aaf}
\left\{
\begin{array}{l}
v\in{\mathscr{C}}^{\infty}(\Omega),\,\,\Delta v=0\,\,\text{ in }\,\,\Omega,
\quad{\mathcal{N}}_{\kappa}v\in L^{1/\theta}(\partial\Omega,\sigma),
\\[6pt]
v\big|^{{}^{\kappa-{\rm n.t.}}}_{\partial\Omega}=f\,\,\text{ at $\sigma$-a.e. point on }\,\partial\Omega,
\\[6pt]
v(x)=O(|x|^{2-n})\,\,\text{ as $|x|\to\infty$ if $\Omega$ is an exterior domain}.
\end{array}
\right.
\end{equation}
Moreover, there exists a constant $C\in(0,\infty)$ so that 
\begin{equation}\label{Dir-Lap1aaf.EEE}
\|{\mathcal{N}}_{\kappa}v\|_{L^{1/\theta}(\partial\Omega,\sigma)}\leq\|f\|_{L^{1/\theta}(\partial\Omega,\sigma)}.
\end{equation}
Given that $\partial\Omega$ has finite measure, we also have 
\begin{equation}\label{ih6g6Fii.d}
{\mathcal{N}}_{\kappa}v\in L^{p}(\partial\Omega,\sigma).
\end{equation} 
Finally, by virtue of the last property in \cite[Theorem~1.5.9, pp.\,151--153]{GHA.III}, in the case when $n=2$
the last condition in \eqref{Dir-Lap1aaf} is equivalent to the fact that
\begin{equation}\label{eq:fSWSSf.43}
\parbox{8.00cm}{if $\Omega$ is an exterior domain in ${\mathbb{R}}^2$ then there exists $c\in{\mathbb{R}}$ such that 
$v(x)=c+O(|x|^{-1})$ as $|x|\to\infty$.} 
\end{equation}

At this stage, since the subharmonic function $w$ and the harmonic function $v$ have matching nontangential boundary traces, 
behave at infinity in a compatible fashion with \eqref{utrFC-XXX-2D.777.WACO.FDD} if $\Omega$ is an exterior domain,  
and since ${\mathcal{N}}_{\kappa}w,\,{\mathcal{N}}_{\kappa}v\in L^{p}(\partial\Omega,\sigma)$ with $p\in(1,\infty)$, 
Proposition~\ref{yrFCC} applies (with the role of $u$ currently played by $v$) and gives that $0\leq w\leq v$ in $\Omega$. 
Consequently,
\begin{equation}\label{ih6g6FiKH}
0\leq{\mathcal{N}}_{\kappa}w\leq{\mathcal{N}}_{\kappa}v\,\,\text{ at every point on }\,\,\partial\Omega.
\end{equation} 
In concert with the estimate in \eqref{Dir-Lap1aaf.EEE}, naturally accompanying \eqref{Dir-Lap1aaf},  
and the knowledge that the function ${\mathcal{N}}_{\kappa}u$ is nonnegative and $\sigma$-measurable on 
$\partial\Omega$ (cf. \cite[Proposition~8.2.3, pp.\,684--685]{GHA.I}), this permits us to estimate 
\begin{align}\label{ih6g6FiKH.22}
\int_{\partial\Omega}{\mathcal{N}}_{\kappa}u\,d\sigma 
&=\int_{\partial\Omega}\big[{\mathcal{N}}_{\kappa}\big(|u|^\theta\big)\big]^{1/\theta}\,d\sigma
=\int_{\partial\Omega}\big[{\mathcal{N}}_{\kappa}w\big]^{1/\theta}\,d\sigma
\nonumber\\[4pt]
&\leq\int_{\partial\Omega}\big[{\mathcal{N}}_{\kappa}v\big]^{1/\theta}\,d\sigma
\leq C\int_{\partial\Omega}|f|^{1/\theta}\,d\sigma
\nonumber\\[4pt]
&=C\int_{\partial\Omega}\big|u\big|^{{}^{\kappa-{\rm n.t.}}}_{\partial\Omega}\big|\,d\sigma<+\infty,
\end{align} 
for some finite constant $C>0$ independent of the function $u$. 
This finishes the proof of the left-pointing implication in the first line 
of \eqref{ih6g6FFF}, in a quantitative sense. 

Let us record our progress. At this point, the equivalence in the first line of \eqref{ih6g6FFF} and 
the first norm-equivalence in \eqref{ih6g6GGG} have been established. Granted these, all other claims 
follow upon recalling from \cite[(10.1.103)-(10.1.104), p.\,625]{GHA.II} that 
if $\Omega\subseteq{\mathbb{R}}^n$ is an Ahlfors regular domain, $\nu$ denotes its geometric measure theoretic outward unit normal, 
$\sigma:={\mathcal{H}}^{\,n-1}\lfloor\partial\Omega$ is the surface measure on $\partial\Omega$, and 
if $F\in{\mathscr{C}}^{\,\infty}(\Omega)\otimes{\mathcal{C}}\!\ell_n$ 
is a function which, for some $\kappa>0$, satisfies (with $D$ denoting the associated Dirac operator) 
\begin{equation}\label{u6ggBly5-HAR.HHH.LL}
\begin{array}{c}
DF=0\,\,\text{ in }\,\,\Omega,\,\,\,{\mathcal{N}}_{\kappa}F\in L^1(\partial\Omega,\sigma)\,\,\text{ and}
\\[4pt]
F\big|^{{}^{\kappa-{\rm n.t.}}}_{\partial\Omega}\,\,\text{ exists (in ${\mathcal{C}}\!\ell_n$) at $\sigma$-a.e. point on }\,\,\partial\Omega,
\end{array}
\end{equation}
then 
\begin{equation}\label{u6ggBly5-HAR.HHH.jj}
\begin{array}{c}
\nu\odot\big(F\big|^{{}^{\kappa-{\rm n.t.}}}_{\partial\Omega}\big)
\in H^1(\partial\Omega,\sigma)\otimes{\mathcal{C}}\!\ell_n\,\,\text{ and}
\\[6pt]
\big\|\nu\odot\big(F\big|^{{}^{\kappa-{\rm n.t.}}}_{\partial\Omega}\big)
\big\|_{H^1(\partial\Omega,\sigma)\otimes{\mathcal{C}}\!\ell_n}\leq 
C\|{\mathcal{N}}_{\kappa}F\|_{L^1(\partial\Omega,\sigma)},
\end{array}
\end{equation}
for some $C\in(0,\infty)$ independent of $F$. Here, we have also used that 
\begin{equation}\label{u7yggGFDD.gffc.IIII.L1}
H^1(\partial\Omega,\sigma)\hookrightarrow L^1(\partial\Omega,\sigma)\,\,\text{ continuously},
\end{equation}
and that $\nu\odot\nu=-1$ (cf. \eqref{X-sqr}).
\end{proof}

Let $\Omega\subset{\mathbb{R}}^{n}$ be an open set of locally finite perimeter, let $\nu$ be its geometric 
measure theoretic outward unit normal, canonically identified with a Clifford algebra-valued function on $\partial\Omega$
as in \eqref{utggGYHNN.iii}, and set $\sigma:={\mathcal{H}}^{\,n-1}\lfloor\,\partial\Omega$. 
Define the action of the (boundary-to-domain) Cauchy-Clifford 
integral operator ${\mathcal{C}}$ on any 
$f\in L^1\big(\partial_\ast\Omega,\frac{\sigma(x)}{1+|x|^{n-1}}\big)\otimes{\mathcal{C}}\!\ell_{n}$ according to 
\begin{equation}\label{Cau-C1.iii}
{\mathcal{C}}f(x):=\frac{1}{\omega_{n-1}}\int_{\partial_\ast\Omega}
\frac{x-y}{|x-y|^n}\odot\nu(y)\odot f(y)\,d\sigma(y),\qquad\forall\,x\in\Omega, 
\end{equation}
and define the action of its principal-value (or, boundary-to-boundary) version ${\mathfrak{C}}$ on $f$, at 
$\sigma$-a.e.$x\in\partial\Omega$, by
\begin{equation}\label{Cau-C3.iii}
{\mathfrak{C}}f(x):=\lim_{\varepsilon\to 0^{+}}\frac{1}{\omega_{n-1}}
\int\limits_{\substack{y\in\partial_\ast\Omega\\ |x-y|>\varepsilon}}
\frac{x-y}{|x-y|^n}\odot\nu(y)\odot f(y)\,d\sigma(y).
\end{equation}
A key feature of \eqref{Cau-C1.iii} is that for each function 
$f\in L^1\big(\partial_\ast\Omega,\frac{\sigma(x)}{1+|x|^{n-1}}\big)\otimes{\mathcal{C}}\!\ell_{n}$ one has
\begin{equation}\label{dble-layer-PDE-cc}
{\mathcal{C}}f\in{\mathscr{C}}^{\infty}(\Omega)\,\,\text{ and }\,\,D({\mathcal{C}}f)=0\,\,\text{ in }\,\,\Omega.
\end{equation}
Moreover, under the assumption that actually $\partial\Omega$ is a {\rm UR} set, 
for each aperture parameter $\kappa>0$ one has (cf. \cite[(2.4.10), p.\,384]{GHA.III})
\begin{equation}\label{eq:u7766Rdqarf}
\begin{array}{c}
{\mathcal{N}}_{\kappa}\big({\mathcal{C}}f\big)\in L^{1,\infty}(\partial\Omega,\sigma)\,\,\text{ and}
\\[4pt]
\big\|{\mathcal{N}}_{\kappa}\big({\mathcal{C}}f\big)
\big\|_{L^{1,\infty}(\partial\Omega,\sigma)}\leq C\|f\|_{L^1(\partial_\ast\Omega,\sigma)\otimes{\mathcal{C}}\!\ell_{n}}
\end{array}
\end{equation}
for each $f\in L^1(\partial_\ast\Omega,\sigma)\otimes{\mathcal{C}}\!\ell_{n}$, for some constant $C=C(\Omega,\kappa)\in(0,\infty)$ independent of $f$, 
and the nontangential boundary trace formula 
\begin{equation}\label{R=R-cc}
\big({\mathcal{C}}f\big)\big|^{{}^{\kappa-{\rm n.t.}}}_{\partial\Omega}
=\big(\tfrac{1}{2}I+{\mathfrak{C}}\big)f\,\,\text{ at $\sigma$-a.e. point on }\,\,\partial_\ast\Omega
\end{equation}
holds for all $f\in L^1\big(\partial_\ast\Omega,\frac{\sigma(x)}{1+|x|^{n-1}}\big)\otimes{\mathcal{C}}\!\ell_{n}$
(cf. \cite[Example~1.4.12, pp.\,62--63]{GHA.IV} and \cite[(1.5.20), p.\,89]{GHA.IV}). 

The Clifford algebra formalism also allows us to consider the Riesz transforms bundled together, into a single entity, we 
call the boundary-to-domain {\tt Clifford}-{\tt Riesz} {\tt transform}. Its action on 
$f\in L^1\big(\partial\Omega\,,\,\frac{\sigma(x)}{1+|x|^{n-1}}\big)\otimes{\mathcal{C}}\!\ell_{n}$ is defined as 
\begin{equation}\label{7hg7gv-VVV.a}
{\mathcal{R}}_{{\mathcal{C}}\!\ell}f(x):=\frac{2}{\omega_{n-1}}\int_{\partial\Omega}
\frac{x-y}{|x-y|^n}\odot f(y)\,d\sigma(y),\quad\forall\,x\in\Omega. 
\end{equation}
It has been shown in \cite[Theorem~2.1.6, pp.\,251--252]{GHA.IV} that whenever $\Omega$ is a {\rm UR} domain, and for any 
$\kappa\in(0,\infty)$, the Clifford-Riesz transform ${\mathcal{R}}_{{\mathcal{C}}\!\ell}$ satisfies
\begin{equation}\label{7hg7gv-VVV.h}
\big\|{\mathcal{N}}_{\kappa}\big({\mathcal{R}}_{{\mathcal{C}}\!\ell}f\big)
\big\|_{L^1(\partial\Omega,\sigma)}\leq C\|f\|_{H^1(\partial\Omega,\sigma)\otimes{\mathcal{C}}\!\ell_{n}}
\end{equation}
for all $f\in H^1(\partial\Omega,\sigma)\otimes{\mathcal{C}}\!\ell_{n}$, for some 
$C=C(\Omega,\kappa)\in(0,\infty)$ independent of $f$. 

Our next result indicates that, within the context of infinitesimally flat {\rm AR} domains, if a Clifford algebra-valued function 
leads to a better membership than predicted in \eqref{eq:u7766Rdqarf} it necessarily belongs to a smaller space, specifically $H^1$.

\begin{theorem}\label{yrFCC-TT.aat}
Let $\Omega$ be an infinitesimally flat {\rm AR} domain in $\mathbb{R}^{n}$. Denote by $\nu$ its geometric measure 
theoretic outward unit normal, canonically identified with a Clifford algebra-valued function on $\partial\Omega$
as in \eqref{utggGYHNN.iii}, and set $\sigma:={\mathcal{H}}^{n-1}\lfloor\partial\Omega$. Also, fix some 
aperture parameter $\kappa>0$. Then for any function $f\in L^{1}(\partial\Omega,\sigma)\otimes{\mathcal{C}}\!\ell_n$ 
one has
\begin{equation}\label{ih6gAtfvv}
{\mathcal{N}}_{\kappa}({\mathcal{C}}f)\in L^{1}(\partial\Omega,\sigma)\Longleftrightarrow
\nu\odot f\in H^{1}(\partial\Omega,\sigma)\otimes{\mathcal{C}}\!\ell_n,
\end{equation} 
in a quantitative fashion. 
\end{theorem}

\begin{proof}
To get started, observe that the left-pointing implication in \eqref{ih6gAtfvv} is a consequence 
of \eqref{Cau-C1.iii}, \eqref{7hg7gv-VVV.a}, and \eqref{7hg7gv-VVV.h}. There remains to 
establish that, given $f\in L^{1}(\partial\Omega,\sigma)\otimes{\mathcal{C}}\!\ell_n$ 
with ${\mathcal{N}}_{\kappa}({\mathcal{C}}f)\in L^{1}(\partial\Omega,\sigma)$, we 
necessarily have $\nu\odot f\in H^{1}(\partial\Omega,\sigma)\otimes{\mathcal{C}}\!\ell_n$. 
To this end, define 
\begin{equation}\label{eq:py5r42rf}
\Omega_{+}:=\Omega\,\,\text{ and }\,\,\Omega_{-}:={\mathbb{R}}^n\setminus\overline{\Omega},
\end{equation}
and recall from Theorem~\ref{6trrf.TT.ccc.WACO.222.NEW.AR} that $\Omega_{-}$ is also an infinitesimally flat {\rm AR} domain whose 
boundary is $\partial\Omega$, and whose geometric measure theoretic outward unit normal is $-\nu$. If we define 
\begin{equation}\label{Cau-C1.iii.UUU}
u_{\pm}(x):=\frac{1}{\omega_{n-1}}\int_{\partial\Omega}
\frac{x-y}{|x-y|^n}\odot\nu(y)\odot f(y)\,d\sigma(y),\qquad\forall\,x\in\Omega_{\pm},
\end{equation}
then from \eqref{dble-layer-PDE-cc}, the working assumptions, \eqref{eq:u7766Rdqarf}, and \eqref{R=R-cc} we have
\begin{equation}\label{eq:UTYRfa7tr5fc}
\begin{array}{c}
u_{\pm}\in{\mathscr{C}}^{\,\infty}(\Omega)\otimes{\mathcal{C}}\!\ell_n,\quad Du_{\pm}=0\,\,\text{ in }\,\,\Omega_{\pm},
\\[4pt]
{\mathcal{N}}_{\kappa}u_{+}\in L^1(\partial\Omega,\sigma),\quad{\mathcal{N}}_{\kappa}u_{-}\in L^{1,\infty}(\partial\Omega,\sigma),
\\[4pt]
u_{\pm}\big|^{{}^{\kappa-{\rm n.t.}}}_{\partial\Omega}=\big(\pm\tfrac{1}{2}I+{\mathfrak{C}}\big)f\,\,\text{ at $\sigma$-a.e. point on $\partial\Omega$},
\\[4pt]
\text{and $u_{\pm}(x)=o(1)$ as $|x|\to\infty$ if $\Omega_{\pm}$ is an exterior domain}.
\end{array}
\end{equation}
In particular, ${\mathcal{N}}_{\kappa}u_{+}\in L^1(\partial\Omega,\sigma)$ implies that
\begin{equation}\label{eq:975er5er5we}
u_{+}\big|^{{}^{\kappa-{\rm n.t.}}}_{\partial\Omega}\in L^{1}(\partial\Omega,\sigma)\otimes{\mathcal{C}}\!\ell_n.
\end{equation}
Based on this, \eqref{eq:UTYRfa7tr5fc}, and Proposition~\ref{yrFCC-TT} we then see that
\begin{equation}\label{ih6g6Fyf.FVV.kj.A}
\nu\odot\Big(\big(\tfrac{1}{2}I+{\mathfrak{C}}\big)f\Big)
=\nu\odot\Big(u_{+}\big|^{{}^{\kappa-{\rm n.t.}}}_{\partial\Omega}\Big)\in H^{1}(\partial\Omega,\sigma)\otimes{\mathcal{C}}\!\ell_n.
\end{equation} 
Also, the first jump-formula in \eqref{eq:UTYRfa7tr5fc} and \eqref{eq:975er5er5we} imply
\begin{equation}\label{ih6g6Fyf.1}
(\tfrac{1}{2}I+{\mathfrak{C}}\big)f=u_{+}\big|^{{}^{\kappa-{\rm n.t.}}}_{\partial\Omega}\in L^{1}(\partial\Omega,\sigma)\otimes{\mathcal{C}}\!\ell_n
\end{equation} 
which (in light of the last jump-formula in \eqref{eq:UTYRfa7tr5fc}) further entails 
\begin{equation}\label{ih6g6Fyf.2}
u_{-}\big|^{{}^{\kappa-{\rm n.t.}}}_{\partial\Omega}=
\big(-\tfrac{1}{2}I+{\mathfrak{C}}\big)f=\big(\tfrac{1}{2}I+{\mathfrak{C}}\big)f-f\in L^{1}(\partial\Omega,\sigma)\otimes{\mathcal{C}}\!\ell_n.
\end{equation} 
In turn, based on \eqref{ih6g6Fyf.2} and Proposition~\ref{yrFCC-TT} we conclude that
\begin{equation}\label{ih6g6Fyf.FVV.kj.Ai8t}
\nu\odot\Big(\big(-\tfrac{1}{2}I+{\mathfrak{C}}\big)f\Big)
=\nu\odot\Big(u_{-}\big|^{{}^{\kappa-{\rm n.t.}}}_{\partial\Omega}\Big)\in H^{1}(\partial\Omega,\sigma)\otimes{\mathcal{C}}\!\ell_n.
\end{equation}
Ultimately, from \eqref{ih6g6Fyf.FVV.kj.A} and \eqref{ih6g6Fyf.FVV.kj.Ai8t} we see that 
\begin{equation}\label{ih6g6Fyf.FVV.kj.Ai8t.2}
\nu\odot f=\nu\odot\Big(\big(\tfrac{1}{2}I+{\mathfrak{C}}\big)f\Big)
-\nu\odot\Big(\big(-\tfrac{1}{2}I+{\mathfrak{C}}\big)f\Big)\in H^{1}(\partial\Omega,\sigma)\otimes{\mathcal{C}}\!\ell_n,
\end{equation}
plus a naturally accompanying estimate. 
\end{proof}

Here is a version of Theorem~\ref{yrFCC-TT.aat} which features the principal-value Cauchy-Clifford 
integral operator in place of its boundary-to-domain version, appearing in \eqref{ih6gAtfvv}. 
Specifically, while ordinarily ${\mathfrak{C}}$ maps $L^1$ into $L^{1,\infty}$, if a function in $L^1$ is mapped by 
${\mathfrak{C}}$ into $L^1$ then necessarily the function actually belongs to the Hardy space $H^1$.

\begin{theorem}\label{yrFCC-TT.aat.2}
Let $\Omega$ be an infinitesimally flat {\rm AR} domain in $\mathbb{R}^{n}$. Denote by $\nu$ its geometric 
measure theoretic outward unit normal, canonically identified with a Clifford algebra-valued function on $\partial\Omega$
as in \eqref{utggGYHNN.iii}, and set $\sigma:={\mathcal{H}}^{\,n-1}\lfloor\partial\Omega$. Then
for each function $f\in L^{1}(\partial\Omega,\sigma)\otimes{\mathcal{C}}\!\ell_n$ one has
\begin{equation}\label{ih6gAtfvv.cc}
{\mathfrak{C}}f\in L^{1}(\partial\Omega,\sigma)\otimes{\mathcal{C}}\!\ell_n
\Longleftrightarrow\nu\odot f\in H^{1}(\partial\Omega,\sigma)\otimes{\mathcal{C}}\!\ell_n,
\end{equation} 
in a quantitative fashion. 
\end{theorem}

\begin{proof}
Let $f\in L^{1}(\partial\Omega,\sigma)\otimes{\mathcal{C}}\!\ell_n$ be such that
$\nu\odot f\in H^{1}(\partial\Omega,\sigma)\otimes{\mathcal{C}}\!\ell_n$. Introduce $\Omega_{\pm}$ as in \eqref{eq:py5r42rf}
and define $u_{\pm}$ as in \eqref{Cau-C1.iii.UUU}. From \eqref{7hg7gv-VVV.h} we see that 
${\mathcal{N}}_{\kappa}u_{\pm}\in L^1(\partial\Omega,\sigma)$ which, together with \eqref{R=R-cc}, implies
\begin{align}\label{eq:uyttwrf5r3}
{\mathfrak{C}}f &=\frac{1}{2}\Big\{\big(\tfrac{1}{2}I+{\mathfrak{C}}\big)f+\big(-\tfrac{1}{2}I+{\mathfrak{C}}\big)f\Big\}
\nonumber\\[4pt]
&=\tfrac{1}{2}\big\{u_{+}\big|^{{}^{\kappa-{\rm n.t.}}}_{\partial\Omega}+u_{-}\big|^{{}^{\kappa-{\rm n.t.}}}_{\partial\Omega}\big\}
\in L^{1}(\partial\Omega,\sigma)\otimes{\mathcal{C}}\!\ell_n,
\end{align}
as desired. In the opposite direction, let $f\in L^{1}(\partial\Omega,\sigma)\otimes{\mathcal{C}}\!\ell_n$ be such that
${\mathfrak{C}}f\in L^{1}(\partial\Omega,\sigma)\otimes{\mathcal{C}}\!\ell_n$. Then 
\begin{equation}\label{eq:uitwrf8909}
\big(\pm\tfrac{1}{2}I+{\mathfrak{C}}\big)\in L^{1}(\partial\Omega,\sigma)\otimes{\mathcal{C}}\!\ell_n.
\end{equation}
With $\Omega_{\pm}$ as in \eqref{eq:py5r42rf} and $u_{\pm}$ as in \eqref{Cau-C1.iii.UUU}, we then have 
\begin{equation}\label{eq:UTYRfa7tr5fc.D}
\begin{array}{c}
u_{\pm}\in{\mathscr{C}}^{\,\infty}(\Omega)\otimes{\mathcal{C}}\!\ell_n,\quad Du_{\pm}=0\,\,\text{ in }\,\,\Omega_{\pm},
\\[4pt]
{\mathcal{N}}_{\kappa}u_{+}\in L^{1,\infty}(\partial\Omega,\sigma),\quad{\mathcal{N}}_{\kappa}u_{-}\in L^{1,\infty}(\partial\Omega,\sigma),
\\[4pt]
u_{\pm}\big|^{{}^{\kappa-{\rm n.t.}}}_{\partial\Omega}=\big(\pm\tfrac{1}{2}I+{\mathfrak{C}}\big)f\in L^{1}(\partial\Omega,\sigma)\otimes{\mathcal{C}}\!\ell_n,
\\[4pt]
\text{and $u_{\pm}(x)=o(1)$ as $|x|\to\infty$ if $\Omega_{\pm}$ is an exterior domain},
\end{array}
\end{equation}
thanks to \eqref{dble-layer-PDE-cc}, \eqref{R=R-cc}, and \eqref{eq:u7766Rdqarf}.
Granted these properties, Proposition~\ref{yrFCC-TT}  applies and gives that
\begin{equation}\label{eq:861t51rPLL}
\nu\odot\Big(\big(\pm\tfrac{1}{2}I+{\mathfrak{C}}\big)f\Big)
=\nu\odot\big(u_{\pm}\big|^{{}^{\kappa-{\rm n.t.}}}_{\partial\Omega}\big)
\in H^{1}(\partial\Omega,\sigma)\otimes{\mathcal{C}}\!\ell_n
\end{equation}
plus natural estimates. As a consequence, 
\begin{equation}\label{eq:861t51rPLL.2}
\nu\odot f=\nu\odot\Big(\big(\tfrac{1}{2}I+{\mathfrak{C}}\big)f\Big)-\nu\odot\Big(\big(-\tfrac{1}{2}I+{\mathfrak{C}}\big)f\Big)
\in H^{1}(\partial\Omega,\sigma)\otimes{\mathcal{C}}\!\ell_n
\end{equation}
in a quantitative fashion. 
\end{proof}

Moving on, recall the ``principal-value'' Riesz transforms associated with a given {\rm UR} set as in \eqref{Cau-RRj}.
One of the principal results for the portion of the present work pertaining to domains with compact boundaries is the following 
theorem asserting that the Riesz transforms characterize the Hardy space $H^1$ in the more inclusive Clifford algebra formalism. 

\begin{theorem}\label{yrFCC-TT.aat.TRa}
Let $\Omega$ be an infinitesimally flat {\rm AR} domain in $\mathbb{R}^{n}$ and 
set $\sigma:={\mathcal{H}}^{\,n-1}\lfloor\partial\Omega$. Consider the Riesz transforms 
$\{R_j\}_{1\leq j\leq n}$ associated with the {\rm UR} set $\partial\Omega$ as in \eqref{Cau-RRj}
and glue them together into a singular integral operator acting on functions 
in $L^1(\partial\Omega,\sigma)\otimes{\mathcal{C}}\!\ell_n$ according to 
\begin{equation}\label{utggG-TRFF}
R:={\mathbf{e}}_1\odot R_1+\cdots+{\mathbf{e}}_n\odot R_n.
\end{equation}

Then the following statements are equivalent {\rm (}in a natural quantitative fashion{\rm )}:

\begin{enumerate}
\item[(1)] $f\in L^{1}(\partial\Omega,\sigma)\otimes{\mathcal{C}}\!\ell_n$ and 
$R_jf\in L^{1}(\partial\Omega,\sigma)\otimes{\mathcal{C}}\!\ell_n$ for $1\leq j\leq n$;
\item[(2)] $f\in L^{1}(\partial\Omega,\sigma)\otimes{\mathcal{C}}\!\ell_n$ and 
$Rf\in L^{1}(\partial\Omega,\sigma)\otimes{\mathcal{C}}\!\ell_n$;
\item[(3)] $f\in H^{1}(\partial\Omega,\sigma)\otimes{\mathcal{C}}\!\ell_n$.
\end{enumerate}
\end{theorem}

\begin{proof}
That (1) $\Rightarrow$ (2) is clear from \eqref{utggG-TRFF}. To show that 
(2) $\Rightarrow$ (3), denote by $\nu$ the geometric measure theoretic outward 
unit normal to $\Omega$, canonically identified with a Clifford algebra-valued function 
on $\partial\Omega$ as in \eqref{utggGYHNN.iii}, and consider a function $f$ as in (2).  
If we define
\begin{equation}\label{utggG-TRFF.2}
g:=\nu\odot f\in L^{1}(\partial\Omega,\sigma)\otimes{\mathcal{C}}\!\ell_n,
\end{equation}
then $\nu\odot g=-f$ (since $\nu\odot\nu=-1$, as seen from \eqref{X-sqr}) hence 
\begin{equation}\label{utggG-TRFF.3}
{\mathfrak{C}}g=\tfrac{1}{2}R(\nu\odot g)=-\tfrac{1}{2}Rf\in L^{1}(\partial\Omega,\sigma)\otimes{\mathcal{C}}\!\ell_n.
\end{equation}
Granted \eqref{utggG-TRFF.2}-\eqref{utggG-TRFF.3}, Theorem~\ref{yrFCC-TT.aat.2} applies and gives that
\begin{equation}\label{utggG-TRFF.rt}
f=-\nu\odot g\in H^{1}(\partial\Omega,\sigma)\otimes{\mathcal{C}}\!\ell_n,
\end{equation}
proving (3). Finally, that (3) $\Rightarrow$ (1) is seen from \eqref{u7yggGFDD.gffc.IIII.L1}
and \cite[Theorem~2.3.2, pp.\,348--357, item (6)]{GHA.III}.
\end{proof}

When specialized to scalar-valued functions, Theorem~\ref{yrFCC-TT.aat.TRa} shows that
\begin{equation}\label{7y5t6tt}
\parbox{8.70cm}{assuming $\Omega\subset\mathbb{R}^{n}$ is an infinitesimally flat {\rm AR} domain
and $\sigma={\mathcal{H}}^{n-1}\lfloor\partial\Omega$, then for each function $f\in L^{1}(\partial\Omega,\sigma)$ 
one has $R_jf\in L^{1}(\partial\Omega,\sigma)$ for all $j\in\{1,\dots,n\}$
if and only if the function $f$ belongs to the Hardy space $H^{1}(\partial\Omega,\sigma)$.}
\end{equation}
This proves \eqref{7ggVV}. In concert with the Hahn-Banach Theorem and the duality result 
\begin{equation}\label{KDLE20-Bis}
\big(H^1(\partial\Omega,\sigma)\big)^\ast={\rm BMO}(\partial\Omega,\sigma)
\end{equation}
(cf. \cite[Theorem~4.6.1, pp.\,182--184]{GHA.II}), 
this characterization implies the following version of the classical Fefferman-Stein representation theorem (cf. \cite[Theorem~3, p.\,145]{FeSt72}).

\begin{theorem}\label{u6g5rfffc}
Let $\Omega\subset\mathbb{R}^{n}$ be an infinitesimally flat {\rm AR} domain
and set $\sigma:={\mathcal{H}}^{\,n-1}\lfloor\partial\Omega$. Then each 
$f\in{\rm BMO}(\partial\Omega,\sigma)$ may be expressed as 
\begin{equation}\label{utggG-TRFF.rt.A}
f=f_0+\sum_{j=1}^n R_jf_j\,\,\text{ on }\,\,\partial\Omega, 
\end{equation}
for some $f_0,f_1,\dots,f_n\in L^\infty(\partial\Omega,\sigma)$ with 
\begin{equation}\label{utggG-TRFF.rt.B}
\|f_0\|_{L^{\infty}(\partial\Omega,\sigma)}+\sum_{j=1}^n\|f_j\|_{L^{\infty}(\partial\Omega,\sigma)}
\leq C\|f\|_{{\rm BMO}(\partial\Omega,\sigma)}
\end{equation}
where the constant $C\in(0,\infty)$ is independent of $f$. 
\end{theorem}

\begin{proof}
Take
\begin{equation}\label{yy6yhh/1}
{\mathscr{V}}:=\underbrace{L^{1}(\partial\Omega,\sigma)\oplus L^{1}(\partial\Omega,\sigma)\oplus\cdots\oplus
L^{1}(\partial\Omega,\sigma)}_{\text{$n+1$ copies}},
\end{equation}
canonically viewed as a Banach space equipped with the norm
\begin{equation}\label{yy6yhh/2b}
\|(g_0,g_1,\dots,g_n)\|_{\mathscr{V}}:=\|g_0\|_{L^{1}(\partial\Omega,\sigma)}
+\sum_{j=1}^n\|g_j\|_{L^{1}(\partial\Omega,\sigma)}.
\end{equation}
Let us also define 
\begin{equation}\label{yy6yhh/2}
{\mathscr{W}}:=\big\{(g_0,g_1,\dots,g_n)\in{\mathscr{V}}:\,g_j=R_jg_0\,\text{ for }\,1\leq j\leq n\big\}.
\end{equation}
Then ${\mathscr{W}}$ is a closed linear subspace of ${\mathscr{V}}$, and \eqref{7y5t6tt} implies that 
the linear mapping 
\begin{equation}\label{yy6yhh/3}
H^{1}(\partial\Omega,\sigma)\ni g_0\longmapsto (g_0,R_1g_0,\dots,R_ng_0)\in{\mathscr{W}}
\end{equation}
is a well-defined Banach space isomorphism of $H^{1}(\partial\Omega,\sigma)$ onto ${\mathscr{W}}$. 

Consider now an arbitrary function $f\in{\rm BMO}(\partial\Omega,\sigma)$. In view of \eqref{KDLE20-Bis}
we may regard $f$ as a continuous linear functional $\Lambda:H^{1}(\partial\Omega,\sigma)\to{\mathbb{C}}$.
Via \eqref{yy6yhh/3}, this may be subsequently identified with a continuous linear functional on 
${\mathscr{W}}$ which, thanks to the Hahn-Banach Theorem, may be further extended to a continuous
linear functional $\widetilde{\Lambda}$ on ${\mathscr{V}}$, i.e.,
\begin{equation}\label{yy6yhh/4}
\widetilde{\Lambda}\in{\mathscr{V}}^\ast=
\underbrace{L^{\infty}(\partial\Omega,\sigma)\oplus L^{\infty}(\partial\Omega,\sigma)
\oplus\cdots\oplus L^{\infty}(\partial\Omega,\sigma)}_{\text{$n+1$ copies}}.
\end{equation}
After unraveling notation, this argument proves that there exist functions
$f_0,f_1,\dots,f_n\in L^{\infty}(\partial\Omega,\sigma)$ with the property that
\begin{equation}\label{yy6yhh/5}
\Lambda(h)=\int_{\partial\Omega}hf_0\,d\sigma-\sum_{j=1}^n\int_{\partial\Omega}(R_jh)f_j\,d\sigma
\,\,\text{ for each }\,\,h\in H^{1}(\partial\Omega,\sigma),
\end{equation}
and
\begin{equation}\label{yy6yhh/6}
\|\Lambda\|_{(H^{1}(\partial\Omega,\sigma))^\ast}\leq\|f_0\|_{L^{\infty}(\partial\Omega,\sigma)}
+\sum_{j=1}^n\|f_j\|_{L^{\infty}(\partial\Omega,\sigma)}.
\end{equation}
In particular, for each $h\in L^{2}(\partial\Omega,\sigma)\subset H^{1}(\partial\Omega,\sigma)$ we may
use the anti-Hermitian character of the Riesz transforms on $\partial\Omega$ (cf. \cite[(2.3.25), p.\,350]{GHA.III}) 
to write 
\begin{align}\label{yy6yhh/7}
\Lambda(h) &=\int_{\partial\Omega}hf_0\,d\sigma+\sum_{j=1}^n\int_{\partial\Omega}h(R_jf_j)\,d\sigma
\nonumber\\[6pt]
&=\int_{\partial\Omega}h\Big\{f_0+\sum_{j=1}^nR_jf_j\Big\}\,d\sigma.
\end{align}
In light of the fact that the space $L^{2}(\partial\Omega,\sigma)$ is a dense subspace of the Hardy space $H^{1}(\partial\Omega,\sigma)$ 
(see, e.g., \cite[Proposition~4.4.4, p.\,164]{GHA.II} in this regard), this ultimately proves that the functionals
$f\equiv\Lambda\in\big(H^{1}(\partial\Omega,\sigma)\big)^\ast$ 
and $f_0+\sum_{j=1}^nR_jf_j\in{\rm BMO}(\partial\Omega,\sigma)=\big(H^{1}(\partial\Omega,\sigma)\big)^\ast$
coincide as functions in the space ${\rm BMO}(\partial\Omega,\sigma)$. 
\end{proof}

We conclude by giving the proof of Theorem~\ref{u6g5rfffc-CCC}:

\vskip 0.08in
\begin{proof}[Proof of Theorem~\ref{u6g5rfffc-CCC}]
The left-to-right inclusion follows from Theorem~\ref{u6g5rfffc}, while the right-to-left 
inclusion is a consequence of \cite[(2.3.39), p.\,353]{GHA.III}.
\end{proof}

\section{Proofs of Main Results: Domains with Unbounded Boundaries}
\setcounter{equation}{0}
\label{S-4}

Compared with our work carried out in \S{S-3}, dealing with domains possessing unbounded boundaries brings in novel difficulties. 
As such, a number of tools have to be sharpened while other results need to be work out from scratch. In the latter category we have
the following geometric lemma, which is designed to help correlate the behavior of a function at infinity in the domain with that on the boundary.  

\begin{lemma}\label{NEW.COL.LEMMA}
Suppose $\Omega\subseteq{\mathbb{R}}^n$ is an open set with the property that $\partial\Omega$ is an unbounded Ahlfors regular set, 
and define $\sigma:={\mathcal{H}}^{n-1}\lfloor\partial\Omega$. Fix a reference point $z\in\partial\Omega$ and for each $R>0$ 
abbreviate 
\begin{equation}\label{eq:uygt.8t5f}
A_R:=\Omega\cap[B(z,2R)\setminus B(z,R)],\quad\Delta(z,R):=B(z,R)\cap\partial\Omega.
\end{equation}

Then there exist $C,M,\widetilde{\kappa}\in(0,\infty)$ which depend only on $n$ and the Ahlfors regularity constants of $\partial\Omega$ 
with the property that for each Lebesgue measurable function $u:\Omega\to{\mathbb{R}}$ and each $R\in(0,\infty)$ one has
\begin{equation}\label{NEW.COL.LEM}
\int_{A_R}|u|\,d{\mathcal{L}}^n\leq CR\cdot\int_{\Delta(z,MR)\setminus\Delta(z,R/2)}{\mathcal{N}}_{\widetilde{\kappa}}^{2R}u\,d\sigma
\end{equation} 
\end{lemma}

\begin{proof}
With $C_{\partial\Omega}$ and $c_{\partial\Omega}$ denoting, respectively, 
the upper and lower Ahlfors regularity constants of $\partial\Omega$, fix
\begin{equation}\label{eq:PR.1}
M>\max\Big\{4\,,\,\frac{1}{2}\Big(\frac{C_{\partial\Omega}}{c_{\partial\Omega}}\Big)^{\frac{1}{n-1}}\Big\}\,\,
\text{ and }\,\,\widetilde{\kappa}>4(M+2).
\end{equation} 
Also, select some aperture parameter $\kappa\in(0,1)$, and pick an arbitrary $R>0$. 
Elementary geometric considerations then show that 
\begin{equation}\label{eq:PR.2}
\Gamma_{\kappa}(x)\cap\big[A_R\cap{\mathcal{O}}_{R/4}\big]=\varnothing\,\,\text{ for all }\,\,
x\in\partial\Omega\setminus\big[\Delta(z,MR)\setminus\Delta(z,R/2)\big],
\end{equation}
and that
\begin{equation}\label{eq:PR.3}
A_R\setminus{\mathcal{O}}_{R/4}\subseteq\Gamma_{\widetilde{\kappa}}(x)\cap B(z,2R)\,\,
\text{ for all }\,\,x\in\Delta(z,MR)\setminus\Delta(z,R/2).
\end{equation}

Suppose a Lebesgue measurable function $u:\Omega\to{\mathbb{R}}$ has been given. Then, by \eqref{eq:PR.3},
\begin{equation}\label{eq:PR.4}
\begin{array}{c}
|u(y)|\leq\big({\mathcal{N}}_{\widetilde{\kappa}}^{2R}u\big)(x)\,\text{ for all }\,y\in A_R\setminus{\mathcal{O}}_{R/4},
\\[4pt]
\text{and all }\,\,x\in\Delta(z,MR)\setminus\Delta(z,R/2).
\end{array}
\end{equation}
Integrating in the variable $y\in A_R\setminus{\mathcal{O}}_{R/4}$ with respect to ${\mathcal{L}}^n$ and in 
the variable $x\in\Delta(z,MR)\setminus\Delta(z,R/2)$ with respect to 
$\sigma$ therefore yields
\begin{align}\label{eq:PR.5}
\sigma\Big(\Delta(z,MR) & \setminus\Delta(z,R/2)\Big)\cdot\int_{A_R\setminus{\mathcal{O}}_{R/4}}|u|\,d{\mathcal{L}}^n
\nonumber\\[4pt]
&\leq{\mathcal{L}}^n\Big(A_R\setminus{\mathcal{O}}_{R/4}\Big)\cdot\int_{\Delta(z,MR)\setminus\Delta(z,R/2)}
{\mathcal{N}}_{\widetilde{\kappa}}^{2R}u\,d\sigma.
\end{align}
Note that, by Ahlfors regularity,
\begin{align}\label{eq:PR.6}
\sigma\Big(\Delta(z,MR)\setminus\Delta(z,R/2)\Big)
&\geq\sigma\Big(\Delta(z,MR)\Big)-\sigma\Big(\Delta(z,R/2)\Big)
\nonumber\\[4pt]
&\geq c_{\partial\Omega}(MR)^{n-1}-C_{\partial\Omega}(R/2)^{n-1}
\nonumber\\[4pt]
&=R^{n-1}\Big\{c_{\partial\Omega}M^{n-1}-C_{\partial\Omega}(1/2)^{n-1}\Big\}
\end{align}
and that $c_{\partial\Omega}M^{n-1}-C_{\partial\Omega}(1/2)^{n-1}>0$ thanks to our choice of $M$ in \eqref{eq:PR.1}. Since we also have
\begin{equation}\label{eq:PR.7}
{\mathcal{L}}^n\Big(A_R\setminus{\mathcal{O}}_{R/4}\Big)\leq{\mathcal{L}}^n\Big(B(z,2R)\Big)=CR^n,
\end{equation}
we conclude from \eqref{eq:PR.5}-\eqref{eq:PR.7} that
\begin{equation}\label{eq:PR.8}
\int_{A_R\setminus{\mathcal{O}}_{R/4}}|u|\,d{\mathcal{L}}^n
\leq CR\cdot\int_{\Delta(z,MR)\setminus\Delta(z,R/2)}{\mathcal{N}}_{\widetilde{\kappa}}^{2R}u\,d\sigma.
\end{equation} 
In addition, we have
\begin{align}\label{eq:PR.9}
\int_{A_R\cap{\mathcal{O}}_{R/4}}|u|\,d{\mathcal{L}}^n&=\int_{{\mathcal{O}}_{2R}}
\Big|u\cdot{\mathbf{1}}_{A_R\cap{\mathcal{O}}_{R/4}}\Big|\,d{\mathcal{L}}^n
\nonumber\\[4pt]
&\leq CR\cdot\int_{\partial\Omega}{\mathcal{N}}_{\kappa}\Big(u\cdot{\mathbf{1}}_{A_R\cap{\mathcal{O}}_{R/4}}\Big)\,d\sigma
\nonumber\\[4pt]
&=CR\cdot\int_{\Delta(z,MR)\setminus\Delta(z,R/2)}{\mathcal{N}}_{\kappa}
\Big(u\cdot{\mathbf{1}}_{A_R\cap{\mathcal{O}}_{R/4}}\Big)\,d\sigma
\nonumber\\[4pt]
&\leq CR\cdot\int_{\Delta(z,MR)\setminus\Delta(z,R/2)}{\mathcal{N}}^{2R}_{\kappa}u\,d\sigma,
\end{align}
where the first inequality follows from Proposition~\ref{kiGa-615} and the second equality is a consequence of \eqref{eq:PR.2}.
At this stage, the estimate claimed in \eqref{NEW.COL.LEM} is obtained by combining \eqref{eq:PR.8} with \eqref{eq:PR.9}, 
bearing in mind that $\widetilde{\kappa}>\kappa$. 
\end{proof}

Next, we formally present the class of domains that are most relevant to this section.
In line with \cite[Definition~3.1.8, p.\,121]{GHA.V} and \cite[Definition~2.15, p.\,85]{M5}, we make the following definition. 

\begin{definition}\label{def:USKT}
Fix $n\in{\mathbb{N}}$ with $n\geq 2$ and consider a parameter $\delta>0$. Call a nonempty, proper subset $\Omega$ 
of ${\mathbb{R}}^n$ a $\delta$-{\tt flat} {\tt Ahlfors} {\tt regular} {\tt domain} {\rm (}or $\delta$-{\tt flat} 
{\tt AR} {\tt domain}, or simply $\delta$-{\tt AR} {\tt domain}{\rm )} provided $\Omega$ is an 
Ahlfors regular domain {\rm (}in the sense of Definition~\ref{def:AR}{\rm )} whose geometric measure 
theoretic outward unit normal $\nu$ satisfies {\rm (}with $\sigma:={\mathcal{H}}^{n-1}\lfloor\partial\Omega${\rm )}
\begin{equation}\label{trfca-hj7i}
\|\nu\|_{[{\rm BMO}(\partial\Omega,\sigma)]^n}<\delta.
\end{equation}
\end{definition}

In particular, whenever $\delta\in(0,1)$, work in \cite[Lemma~2.8, p.\,52]{M5} guarantees that $\partial\Omega$ is an unbounded set.
The next order of business is to adapt and further sharpen Proposition~\ref{yrFCC} to the class of $\delta$-{\rm AR} 
domains with $\delta\in(0,1)$ small. 
One improvement, namely the ability to ask for nontangential maximal function control only in {\it weak} Lebesgue spaces, is dictated by how we shall 
employ this later, in the proof of Proposition~\ref{yrFCC-TT.UNB}, where we can no longer embed weak Lebesgue spaces into genuine global Lebesgue spaces
given that we presently work on non-compact boundaries. 

\begin{proposition}\label{yrFCC.UND}
Fix a dimension $n\in{\mathbb{N}}$ with $n\geq 2$ along with an integrability exponent $p\in(1,\infty)$.
Let $\Omega$ be a $\delta$-{\rm AR} domain in $\mathbb{R}^{n}$ {\rm (}cf. Definition~\ref{def:USKT}{\rm )} 
where $\delta\in(0,1)$ is sufficiently small {\rm (}relative to $n$, $p$, and the Ahlfors regularity constants of $\partial\Omega${\rm )}.
Abbreviate $\sigma:={\mathcal{H}}^{\,n-1}\lfloor\partial\Omega$, and fix some aperture parameter $\kappa>0$. 
Consider a harmonic function $u$ in $\Omega$ satisfying ${\mathcal{N}}_{\kappa}u\in L^{p,\infty}(\partial\Omega,\sigma)$, and consider
a continuous subharmonic function $w$ in $\Omega$ with the property that ${\mathcal{N}}_{\kappa}w\in L^{p,\infty}(\partial\Omega,\sigma)$ and
\begin{equation}\label{ih6g6gfr5.UND}
\begin{array}{c}
\Big(w\big|^{{}^{\kappa-{\rm n.t.}}}_{\partial\Omega}\Big)(x)\,\,\text{ and }\,\,
\Big(u\big|^{{}^{\kappa-{\rm n.t.}}}_{\partial\Omega}\Big)(x)\,\,\text{ exist}
\\[2pt] 
\text{and are equal at $\sigma$-a.e. point }\,\,x\in\partial\Omega.
\end{array}
\end{equation} 

Then $w\leq u$ in $\Omega$.
\end{proposition}

\begin{proof}
Let $p'\in(1,\infty)$ be such that $1/p+1/p'=1$, and associate with it the Lorentz space $L^{p',1}(\partial\Omega,\sigma)$.
Also, fix an integrability exponent $q\in(n-1,\infty)$ and define $\alpha:=1-\frac{n-1}{q}\in(0,1)$. Given any set $E\subseteq{\mathbb{R}}^n$, 
denote by $\dot{\mathscr{C}}^\alpha(E)$ the (homogeneous) space of H\"older continuous functions with exponent $\alpha$ in $E$.
Next, fix an arbitrary point $x_0\in\Omega$, and bring in the following compact neighborhood of $x_0$ 
\begin{equation}\label{eq:POS.1}
K:=\overline{B\big(x_0,\,\tfrac{1}{2}\cdot{\rm dist}\,(x_0,\partial\Omega)\big)}\subseteq\Omega.
\end{equation}
The goal is to construct a Green function $G_{\Omega}(\cdot,x_0)$ for the Laplacian in $\Omega$ with pole at $x_0$ satisfying
\begin{align}\label{Dir-Lap4.UNB}
& G_{\Omega}(\cdot,x_0)\in{\mathscr{C}}^{\infty}(\Omega\setminus\{x_0\}),\qquad
\Delta G_{\Omega}(\cdot,x_0)=-\delta_{x_0}\,\,\text{ in }\,\,{\mathcal{D}}'(\Omega),
\\[4pt]
& {\mathcal{N}}^{\Omega\setminus K}_{\kappa}\big(\nabla G_{\Omega}(\cdot,x_0)\big),\,
{\mathcal{N}}^{\Omega\setminus K}_{\kappa}\big(G_{\Omega}(\cdot,x_0)\big)\in L^{p',1}(\partial\Omega,\sigma)\subseteq L^{p'}(\partial\Omega,\sigma),
\label{Dir-Lap5.UNB}
\end{align}
as well as
\begin{equation}\label{eq:POS.4}
G_{\Omega}(\cdot,x_0)\in\dot{\mathscr{C}}^\alpha(\overline{\Omega}\setminus K)\,\,\text{ and }\,\,G_{\Omega}(\cdot,x_0)\big|_{\partial\Omega}=0,
\end{equation}
plus the positivity property
\begin{equation}\label{eq:POS.11}
G_{\Omega}(\cdot,x_0)>0\,\,\text{ in }\,\,\Omega\setminus\{x_0\}.
\end{equation}

The idea is to fix $x_\ast\in{\mathbb{R}}^n\setminus\overline{\Omega}$ and, with $E_\Delta$ as in \eqref{j7ggGYg.EEE}, take
\begin{equation}\label{eq:POS.2}
G_{\Omega}(\cdot,x_0):=-E_{\Delta}(\cdot-x_0)+E_{\Delta}(\cdot-x_\ast)-w\,\,\text{ in }\,\,\Omega,
\end{equation}
where $w$ solves the regularity boundary value problem
\begin{equation}\label{eq:POS.3}
\left\{
\begin{array}{l}
w\in{\mathscr{C}}^\infty(\Omega),\quad\Delta w=0\,\,\text{ in }\,\,\Omega,
\\[4pt]
{\mathcal{N}}_\kappa w,\,{\mathcal{N}}_\kappa(\nabla w)\in L^{p',1}(\partial\Omega,\sigma)\cap L^q(\partial\Omega,\sigma),
\\[4pt]
w\big|^{{}^{\kappa-{\rm n.t.}}}_{\partial\Omega}=\big[-E_{\Delta}(\cdot-x_0)+E_{\Delta}(\cdot-x_\ast)\big]\Big|_{\partial\Omega}.
\end{array}
\right.
\end{equation}
Since $\Omega$ is a $\delta$-{\rm AR} domain with $\delta\in(0,1)$ sufficiently small (relative to $n$, $p$, 
and the Ahlfors regularity constants of $\partial\Omega$), and since the boundary datum is sufficiently regular 
(specifically, it simultaneously belongs to the Lorentz-based Sobolev space of order one
$L^{p',1}_1(\partial\Omega,\sigma)$ and well as the Lebesgue-based Sobolev space of order one $L^q_1(\partial\Omega,\sigma)$, 
as defined in \cite[\S11]{GHA.II}), 
from \cite[Remark~6.2, p.\,381]{M5} (cf. also \cite[Theorem~6.5, pp.\,379-381]{M5} and \cite[Theorem~8.4.1, pp.\,802-804]{GHA.V}) 
we deduce that such a solution exists. Thanks to this and \cite[Lemma~8.3.7, pp.\,693-694]{GHA.I} we therefore have the 
memberships required in 
\eqref{Dir-Lap5.UNB}. That the properties stipulated in \eqref{Dir-Lap4.UNB} are true is clear from \eqref{eq:POS.2}. 

Moreover, according to \cite[Corollary~8.6.8, p.\,743]{GHA.I}, the fact that 
${\mathcal{N}}_\kappa(\nabla w)$ belongs to $L^q(\partial\Omega,\sigma)$ guarantees that $w$ extends to a function in 
$\dot{\mathscr{C}}^\alpha(\overline{\Omega})$. Since $-E_{\Delta}(\cdot-x_0)+E_{\Delta}(\cdot-x_\ast)$ 
is also H\"older continuous of order $\alpha$ in ${\mathbb{R}}^n$ away from $x_0$ and $x_\ast$, we further conclude from \eqref{eq:POS.2}-\eqref{eq:POS.3} 
that the properties demanded in \eqref{eq:POS.4} hold. As a consequence, there exists a constant $C\in(0,\infty)$ such that  
\begin{equation}\label{eq:POS.5}
\sup_{x\in\partial{\mathcal{O}}_\varepsilon}|G_{\Omega}(x,x_0)|\leq C\varepsilon^\alpha\,\,\text{ for each }\,\,
\varepsilon\in\big(0,\,{\rm dist}\,(x_0,\partial\Omega)\big).
\end{equation}
In particular,
\begin{equation}\label{eq:POS.5bis}
\lim_{\varepsilon\to 0^{+}}\sup_{\substack{x\in\Omega\setminus K\,\,\text{ with}\\
{\rm dist}(x,\partial\Omega)=\varepsilon}}|G_{\Omega}(x,x_0)|=0
\end{equation}

To proceed, fix an arbitrary threshold $\varepsilon>0$ and observe that for each $x\in\Omega\setminus K$ 
with ${\rm dist}(x,\partial\Omega)\geq\varepsilon$ we have 
\begin{equation}\label{eq:POS.6}
x\in\Gamma_\kappa(z)\setminus K\,\,\text{ whenever }\,\,z\in\Delta(\pi(x),\kappa\varepsilon)
\end{equation}
where $\pi(x)$ is a point on $\partial\Omega$ with the property that ${\rm dist}(x,\partial\Omega)=|x-\pi(x)|$. This implies
\begin{equation}\label{eq:POS.7}
\begin{array}{c}
|G_{\Omega}(x,x_0)|\leq\big({\mathcal{N}}^{\Omega\setminus K}_\kappa G_{\Omega}(\cdot,x_0)\big)(z)\,\,\text{ for all points}
\\[4pt]
x\in\Omega\setminus K\,\,\text{ with }\,\,{\rm dist}(x,\partial\Omega)\geq\varepsilon\,\,
\text{ and }\,\,z\in\Delta(\pi(x),\kappa\varepsilon).
\end{array}
\end{equation}
Taking the integral average in $z$ over $\Delta(\pi(x),\kappa\varepsilon)$ and applying H\"older's inequality then shows that 
for each $x\in\Omega\setminus K$ with ${\rm dist}(x,\partial\Omega)\geq\varepsilon$ we have
\begin{align}\label{eq:POS.8}
|G_{\Omega}(x,x_0)|&\leq\fint_{\Delta(\pi(x),\kappa\varepsilon)}{\mathcal{N}}^{\Omega\setminus K}_\kappa G_{\Omega}(\cdot,x_0)\,d\sigma
\nonumber\\[4pt]
&\leq\Big(\fint_{\Delta(\pi(x),\kappa\varepsilon)}\big|{\mathcal{N}}^{\Omega\setminus K}_\kappa G_{\Omega}(\cdot,x_0)\big|^{p'}\,d\sigma\Big)^{1/p'}
\nonumber\\[4pt]
&\leq C\varepsilon^{-(n-1)/p'}\Big(\int_{\partial\Omega}{\mathbf{1}}_{\Delta(\pi(x),\kappa\varepsilon)}\cdot
\big|{\mathcal{N}}^{\Omega\setminus K}_\kappa G_{\Omega}(\cdot,x_0)\big|^{p'}\,d\sigma\Big)^{1/p'}.
\end{align}
In particular, from \eqref{eq:POS.8}, \eqref{Dir-Lap5.UNB}, and Lebesgue's Dominated Convergence Theorem we conclude that
\begin{equation}\label{eq:POS.9}
\parbox{9.80cm}{whenever $0<\varepsilon<R<\infty$ are fixed, if the points $\{x_j\}_{j\in{\mathbb{N}}}\subseteq\Omega\setminus K$ 
are such that $R\geq{\rm dist}(x_j,\partial\Omega)\geq\varepsilon$ for each index 
$j\in{\mathbb{N}}$ and also $\lim\limits_{j\to\infty}|x_j|=\infty$,
then one necessarily has $\lim\limits_{j\to\infty}G_{\Omega}(x_j,x_0)=0$.}
\end{equation}
Proceeding analogously to \eqref{eq:POS.8}, or directly invoking \cite[Lemma~8.6.6, p.\,742]{GHA.I}, 
also gives that for all $R\in(0,\infty)$ we have
\begin{equation}\label{eq:POS.10}
\sup_{\substack{x\in\Omega\setminus K\,\,\text{ with}\\{\rm dist}(x,\partial\Omega)\geq R}}|G_{\Omega}(x,x_0)|\leq CR^{-(n-1)/p'}
\big\|{\mathcal{N}}^{\Omega\setminus K}_\kappa G_{\Omega}(\cdot,x_0)\big\|_{L^{p'}(\partial\Omega,\sigma)},
\end{equation}
hence
\begin{equation}\label{eq:POS.10bis}
\lim_{R\to\infty}\sup_{\substack{x\in\Omega\setminus K\,\,\text{ with}\\{\rm dist}(x,\partial\Omega)=R}}|G_{\Omega}(x,x_0)|=0.
\end{equation}
Finally, the function $h:=G_{\Omega}(\cdot,x_0)+E_\Delta(\cdot-x_0)$ is harmonic near $x_0$, hence bounded near $x_0$, 
and so for each $x\in\partial B(x_0,r)$ with $r\in(0,1)$ we have
\begin{equation}\label{eq:POS.10bisbis}
G_{\Omega}(x,x_0)=-E_\Delta(x-x_0)+h(x)=
\left\{
\begin{array}{ll}
\frac{1}{(n-2)\omega_{n-1}}\frac{1}{r^{n-2}}+O(1) & \text{ if }\,\,n\geq 3,
\\[4pt]
\frac{1}{2\pi}\ln(1/r)+O(1) & \text{ if }\,\,n=2.
\end{array}
\right.
\end{equation}
This goes to show that 
\begin{equation}\label{eq:POS.10bisbisbis}
G_{\Omega}(x,x_0)>0\,\,\text{ for each $x\in\partial B(x_0,r)$ with $r\in(0,1)$ small}.
\end{equation}
The end-game in the proof of \eqref{eq:POS.11} is as follows. Pick a small radius $r\in\big(0,\,\tfrac{1}{4}{\rm dist}\,(x_0,\partial\Omega)\big)$ 
along with a threshold $\varepsilon\in\big(0,\,\tfrac{1}{4}{\rm dist}\,(x_0,\partial\Omega)\big)$ and some number $R>2{\rm dist}\,(x_0,\partial\Omega)$.
Having fixed a reference point $z_\ast\in\partial\Omega$ and a radius $\rho>2R$, define  
\begin{equation}\label{eq:D+M.Feb}
D_{r,\varepsilon,R,\rho}:=\Big(B(z_\ast,\rho)\setminus\overline{B(x_0,r)}\Big)\cap\Big({\mathcal{O}}_R\setminus\overline{{\mathcal{O}}_\varepsilon}\Big).
\end{equation}
This is an open bounded set whose closure is contained in $\Omega\setminus\{x_0\}$. In particular, $G(\cdot,x_0)$ is harmonic in 
$D_{r,\varepsilon,R,\rho}$ and continuous on its closure. As such, the Maximum Principle gives
\begin{equation}\label{eq:D+M.Feb.2}
\min_{x\in D_{r,\varepsilon,R,\rho}}G(x,x_0)=\min_{x\in\partial D_{r,\varepsilon,R,\rho}}G(x,x_0).
\end{equation}
Collectively, from \eqref{eq:POS.5bis}, \eqref{eq:POS.9}, \eqref{eq:POS.10bis}, \eqref{eq:POS.10bisbisbis}, 
we see that the limit of $\min_{x\in\partial D_{r,\varepsilon,R,\rho}}G(x,x_0)$ as $\rho\to\infty$, then $R\to\infty$, and finally 
$\varepsilon\to 0^{+}$, is nonnegative whenever $r$ is sufficiently small. This shows that $G_{\Omega}(\cdot,x_0)\geq 0$ in 
$\Omega\setminus\{x_0\}$. 
Since $\Omega\setminus\{x_0\}$ is an open connected set and the harmonic function $G_{\Omega}(\cdot,x_0)$ 
is not identically zero, the latter self-improves
to the positivity condition claimed in \eqref{eq:POS.11}.

Pressing on, recall the family of functions $\{\Phi_{\varepsilon}\}_{\varepsilon>0}$
associated with the set $\Omega$ as in Lemma~\ref{LhkC}. Next, fix a reference point $z\in\partial\Omega$, pick a nonnegative function 
$\Psi\in{\mathscr{C}}^\infty_c({\mathbb{R}}^n)$ with $\Psi\equiv 1$ on $B(0,1)$ and $\Psi\equiv 0$ on ${\mathbb{R}}^n\setminus B(0,2)$, 
then for each $R>0$ define $\Psi_R(x):=\Psi\big((x-z)/R\big)$ at each $x\in{\mathbb{R}}^n$. In particular, given any $\varepsilon>0$ and $R>0$, 
the support of the function $\Phi_{\varepsilon}\cdot\Psi_R$ is a compact subset of $\Omega$. As before, 
bring in a sequence $\{w_j\}_{j\in{\mathbb{N}}}$ 
such that for each $j\in{\mathbb{N}}$ the function $w_j$ belongs to ${\mathscr{C}}^{\infty}(\Omega_j)$, is subharmonic in 
$\Omega_j:=\{x\in\Omega:\,{\rm dist}(x,\partial\Omega)>1/j\}$ and satisfies $w_j\to w$ uniformly on compact subsets of 
$\Omega$ as $j\to\infty$.
With the Green function constructed above, we re-run the portion of the proof of Proposition~\ref{yrFCC} which starting with 
\eqref{Dir-Lap8.aaTa.1} has produced \eqref{eq:WACO.2025}, now considering the vector field 
\begin{align}\label{Dir-Lap8.aaTa.1.UNB}
\vec{F}(y) &:=(\Phi_{\varepsilon}\cdot\Psi_R)(y)w_j(y)\nabla_yG_{\Omega}(y,x_0)
\nonumber\\[6pt]
&\quad +w_j(y)G_{\Omega}(y,x_0)\nabla(\Phi_{\varepsilon}\cdot\Psi_R)(y)
\nonumber\\[6pt]
&\quad -(\Phi_{\varepsilon}\cdot\Psi_R)(y)G_{\Omega}(y,x_0)(\nabla w_j)(y),\qquad\forall\,y\in\Omega\setminus\{x_0\},
\end{align}
with $\varepsilon>0$ small, $R>0$ large, and $j\in{\mathbb{N}}$ such that $j>N/\varepsilon$. Once again, we are arrive at 
\eqref{Dir-Lap8.pLb.II}
with ${\rm I}_{\varepsilon,R}$ and ${\rm II}_{\varepsilon,R}$ as in \eqref{Dir-Lap8.pL.22.II}-\eqref{Dir-Lap8.pL.3.II}. 
It is at this stage that
special attention is needed. For one thing, in view of the fact that $\partial\Omega$ is no longer a compact set, 
the final equalities in 
\eqref{eq:kYGaff.1}-\eqref{eq:kYGaff.2} fail to materialize, so we shall have to work with the initial equalities instead. 
Second, we need to re-examine the terms obtained by expanding the derivatives in the context of ${\rm I}_{\varepsilon,R}$ 
and ${\rm II}_{\varepsilon,R}$ 
and produce suitable replacements of \eqref{Dir-Lap8.pL.22.III}-\eqref{Dir-Lap8.pL.22.VI} in the scenario in which the set 
$\partial\Omega$ is unbounded.
In the case of ${\rm I}_{\varepsilon,R}$ there are two such terms. For one of them we now employ Lemma~\ref{NEW.COL.LEMMA} to the 
function $u:={\mathbf{1}}_{\Omega\setminus K}\cdot\nabla G_{\Omega}(\cdot,x_0)$ to estimate (using notation introduced in its statement)
\begin{align}\label{Dir-Lap8.pL.22.III.UNB}
\Big|\int_{\Omega}\big\langle\nabla_yG_{\Omega}(y,x_0), & \Phi_{\varepsilon}(y)(\nabla\Psi_R)(y)\big\rangle\eta(y)\,dy\Big|
\nonumber\\[4pt]
&\leq CR^{-1}\int_{A_R}|\nabla_yG_{\Omega}(y,x_0)||\eta(y)|\,dy
\nonumber\\[4pt]
&\leq C\int_{\Delta(z,MR)\setminus\Delta(z,R)}{\mathcal{N}}^{\Omega\setminus K}_{\widetilde{\kappa}}\big(\nabla G_{\Omega}(\cdot,x_0)\big)
\cdot{\mathcal{N}}_{\widetilde{\kappa}}\eta\,d\sigma
\nonumber\\[4pt]
&=o(1)\,\,\text{ as }\,\,R\to\infty,
\end{align}
by Lebesgue's Dominated Convergence Theorem, since 
\begin{equation}\label{eq:uiYGGV.1}
\begin{array}{c}
{\mathcal{N}}^{\Omega\setminus K}_{\widetilde{\kappa}}\big(\nabla(G_{\Omega}(\cdot,x_0))\big)\in L^{p',1}(\partial\Omega,\sigma),
\quad{\mathcal{N}}_{\widetilde{\kappa}}\eta\in L^{p,\infty}(\partial\Omega,\sigma),
\\[4pt]
\text{and }\,\,L^{p',1}(\partial\Omega,\sigma)\cdot L^{p,\infty}(\partial\Omega,\sigma)\subseteq L^1(\partial\Omega,\sigma).
\end{array}
\end{equation}
The memberships in the first line of \eqref{eq:uiYGGV.1} are implied by \eqref{Dir-Lap5.UNB}, the fact that 
$\eta:=w-u$ now has ${\mathcal{N}}_{\kappa}\eta\in L^{p,\infty}(\partial\Omega,\sigma)$, and 
\cite[Proposition~8.4.1, p.\,698]{GHA.I}. The inclusion in the second line of \eqref{eq:uiYGGV.1}
is a consequence of (O'Neil's version of) H\"older's inequality for Lorentz spaces. 
The other term obtained by expending derivatives in the context of ${\rm I}_{\varepsilon,R}$ is handled by invoking Lemma~\ref{LhkC} and 
Proposition~\ref{kiGa-615} (with $p:=1$) which permit us to write 
\begin{align}\label{Dir-Lap8.pL.22.III.UNB.doi}
\Big|\int_{\Omega}\big\langle\nabla_yG_{\Omega}(y,x_0), & (\Psi_R)(y)(\nabla\Phi_{\varepsilon})(y)\big\rangle\eta(y)\,dy\Big|
\nonumber\\[4pt]
&\leq C\varepsilon^{-1}\int_{{\mathcal{O}}_\varepsilon}|\nabla_yG_{\Omega}(y,x_0)||\eta(y)|\,dy
\nonumber\\[4pt]
&\leq C\int_{\partial\Omega}{\mathcal{N}}^{\Omega\setminus K}_{\widetilde{\kappa}}\big(\nabla G_{\Omega}(\cdot,x_0)\big)
\cdot{\mathcal{N}}^\varepsilon_{\kappa}\eta\,d\sigma
\nonumber\\[4pt]
&=o(1)\,\,\text{ as }\,\,\varepsilon\to 0^{+},
\end{align}
by Lebesgue's Dominated Convergence Theorem, whose applicability is ensured by \eqref{eq:uiYGGV.1} and Proposition~\ref{nont-ind-11-PP} 
(bearing in mind that $\eta\big|^{{}^{\kappa-{\rm n.t.}}}_{\partial\Omega}=0$ at $\sigma$-a.e. point on $\partial\Omega$). 

In the case of ${\rm II}_{\varepsilon,R}$ there are three terms to consider (cf. \eqref{eq:kYGaff.2}), all of which 
require the pointwise estimate  
\begin{equation}\label{eq:YTFGF.644}
{\mathcal{N}}^{{\mathcal{O}}_{2R}\setminus K}_{\widetilde{\kappa}}\big(G_{\Omega}(\cdot,x_0)\big)
\leq CR\cdot{\mathcal{N}}^{\Omega\setminus K}_{\widetilde{\widetilde{\kappa}}}
\big(\nabla(G_{\Omega}(\cdot,x_0))\big)\,\,\text{ for each }\,\,R>0,
\end{equation}
valid for some $\widetilde{\widetilde{\kappa}}\in(0,\infty)$. This follows as in the proof of 
\cite[Proposition~8.9.17, pp.\,816-817]{GHA.I} with a minor adjustment. The only aspect which requires an adjustment is 
the ability of joining any point $y\in\Gamma_{\widetilde{\kappa}}(x)$ with $x$ via a rectifiable curve 
$\gamma\in\Gamma_{\widetilde{\widetilde{\kappa}}}(x)$ whose length is controlled by a fixed multiple of $|x-y|$.
Specifically, now starting with $y\in\Gamma_{\widetilde{\kappa}}(x)\setminus K$ we would like to ensure that 
$\gamma\in\Gamma_{\widetilde{\widetilde{\kappa}}}(x)\setminus K$ while retaining the aforementioned length condition. 
However, this is easily remedied since if there is a portion of $\gamma$ contained in $K$ we replace it with a suitable arc whose
length is under control. 

For one of these terms, we once more call upon Lemma~\ref{LhkC} and Proposition~\ref{kiGa-615} (with $p:=1$) to estimate  
\begin{align}\label{Dir-Lap8.pL.22.III.UNB.DOI}
\Big|\int_{\Omega}G_{\Omega}(y,x_0) & (\Psi_R)(y)(\Delta\Phi_{\varepsilon})(y)\eta(y)\,dy\Big|
\leq C\varepsilon^{-2}\int_{{\mathcal{O}}_\varepsilon}|G_{\Omega}(\cdot,x_0)||\eta|\,d{\mathcal{L}}^n
\nonumber\\[4pt]
&\leq C\varepsilon^{-1}\int_{\partial\Omega}{\mathcal{N}}^{{\mathcal{O}}_{\varepsilon}\setminus K}_{\widetilde{\kappa}}\big(G_{\Omega}(\cdot,x_0)\big)
\cdot{\mathcal{N}}^\varepsilon_{\widetilde{\kappa}}\eta\,d\sigma
\nonumber\\[4pt]
&\leq C\int_{\partial\Omega}{\mathcal{N}}^{\Omega\setminus K}_{\widetilde{\widetilde{\kappa}}}\big(\nabla(G_{\Omega}(\cdot,x_0))\big)
\cdot{\mathcal{N}}^\varepsilon_{\widetilde{\kappa}}\eta\,d\sigma
\nonumber\\[4pt]
&=o(1)\,\,\text{ as }\,\,\varepsilon\to 0^{+}.
\end{align}
The last equality is a consequence of Lebesgue's Dominated Convergence Theorem, whose applicability is guaranteed by 
Proposition~\ref{nont-ind-11-PP} (keeping in mind that $\eta\big|^{{}^{\kappa-{\rm n.t.}}}_{\partial\Omega}=0$ at $\sigma$-a.e. point on $\partial\Omega$), 
and the fact that, much as in \eqref{eq:uiYGGV.1}, we have
\begin{equation}\label{eq:uiYGGV.2}
\begin{array}{c}
{\mathcal{N}}^{\Omega\setminus K}_{\widetilde{\widetilde{\kappa}}}\big(\nabla(G_{\Omega}(\cdot,x_0))\big)\in L^{p',1}(\partial\Omega,\sigma),
\quad{\mathcal{N}}_{\widetilde{\kappa}}\eta\in L^{p,\infty}(\partial\Omega,\sigma),
\\[4pt]
\text{and }\,\,L^{p',1}(\partial\Omega,\sigma)\cdot L^{p,\infty}(\partial\Omega,\sigma)\subseteq L^1(\partial\Omega,\sigma).
\end{array}
\end{equation}
To treat the second term alluded to above, we once again rely on Lemma~\ref{NEW.COL.LEMMA} and then appeal to \eqref{eq:YTFGF.644} to write
\begin{align}\label{Dir-Lap8.pL.22.IV.UNB}
\Big|\int_{\Omega}G_{\Omega}(y,x_0) & \Phi_{\varepsilon}(y)(\Delta\Psi_R)(y)\eta(y)\,dy\Big|
\leq R^{-2}\int_{A_R}|G_{\Omega}(y,x_0)||\eta(y)|\,dy
\nonumber\\[4pt]
&\leq CR^{-1}\int_{\Delta(z,MR)\setminus\Delta(z,R)}{\mathcal{N}}^{{\mathcal{O}}_{2R}\setminus K}_{\widetilde{\kappa}}\big(G_{\Omega}(\cdot,x_0)\big)
\cdot{\mathcal{N}}_{\widetilde{\kappa}}\eta\,d\sigma
\nonumber\\[4pt]
&\leq C\int_{\Delta(z,MR)\setminus\Delta(z,R)}{\mathcal{N}}^{\Omega\setminus K}_{\widetilde{\widetilde{\kappa}}}
\big(\nabla(G_{\Omega}(\cdot,x_0))\big)\cdot{\mathcal{N}}_{\widetilde{\kappa}}\eta\,d\sigma
\nonumber\\[4pt]
&=o(1)\,\,\text{ as }\,\,R\to\infty,
\end{align}
by virtue of Lebesgue's Dominated Convergence Theorem and \eqref{eq:uiYGGV.2}.
Moreover, given that for large values of $R$ the function $\langle\nabla\Phi_\varepsilon,\nabla\Psi_R\rangle$ no longer vanishes 
identically on $\Omega$ (as it did in the case when $\partial\Omega$ was compact), we presently need to estimate an additional term, namely
\begin{align}\label{Dir-Lap8.pL.22.IV.UNB.new}
\Big|2\int_{\Omega}G_{\Omega}(y,x_0) & \langle\nabla\Phi_\varepsilon(y),\nabla\Psi_R(y)\rangle\eta(y)\,dy\Big|
\nonumber\\[4pt]
&\leq C(\varepsilon R)^{-1}\int_{A_R}{\mathbf{1}}_{{\mathcal{O}}_{\varepsilon}}(y)|G_{\Omega}(y,x_0)||\eta(y)|\,dy
\nonumber\\[4pt]
&\leq C\varepsilon^{-1}\int_{\Delta(z,MR)\setminus\Delta(z,R)}{\mathcal{N}}^\varepsilon_{\widetilde{\kappa}}\big(G_{\Omega}(\cdot,x_0)\big)
\cdot{\mathcal{N}}_{\widetilde{\kappa}}\eta\,d\sigma
\nonumber\\[4pt]
&\leq C\int_{\Delta(z,MR)\setminus\Delta(z,R)}{\mathcal{N}}^{\Omega\setminus K}_{\widetilde{\widetilde{\kappa}}}
\big(\nabla(G_{\Omega}(\cdot,x_0))\big)\cdot{\mathcal{N}}_{\widetilde{\kappa}}\eta\,d\sigma
\nonumber\\[4pt]
&=o(1)\,\,\text{ as }\,\,R\to\infty,
\end{align}
thanks to Lemma~\ref{NEW.COL.LEMMA}, \eqref{eq:YTFGF.644} (used with $R:=\varepsilon/2$), Lebesgue's Dominated Convergence Theorem, and \eqref{eq:uiYGGV.2}.

As in the end-game of the proof of Proposition~\ref{yrFCC}, from \eqref{Dir-Lap8.pL.22.III.UNB}, \eqref{Dir-Lap8.pL.22.III.UNB.doi}, 
\eqref{Dir-Lap8.pL.22.III.UNB.DOI}, \eqref{Dir-Lap8.pL.22.IV.UNB}, and \eqref{Dir-Lap8.pL.22.IV.UNB.new} we conclude that $\eta(x_0)\leq 0$, 
hence $w(x_0)\leq u(x_0)$. Given the arbitrariness of the point $x_0$, the desired conclusion follows. 
\end{proof}

In turn, Proposition~\ref{yrFCC.UND} is a key ingredient in the proof of the following version of Proposition~\ref{yrFCC-TT}.

\begin{proposition}\label{yrFCC-TT.UNB}
Let $\Omega$ be a $\delta$-{\rm AR} domain in $\mathbb{R}^{n}$ {\rm (}cf. Definition~\ref{def:USKT}{\rm )} 
where $\delta\in(0,1)$ is sufficiently small {\rm (}relative to the dimension $n$ and the Ahlfors regularity constants of $\partial\Omega${\rm )}.
Abbreviate $\sigma:={\mathcal{H}}^{n-1}\lfloor\partial\Omega$ and denote by $\nu$ the geometric measure 
theoretic outward unit normal to $\Omega$, canonically identified with a Clifford algebra-valued 
function on $\partial\Omega$ as in \eqref{utggGYHNN.iii}. Finally, fix an aperture parameter $\kappa>0$ and 
assume $u\in{\mathscr{C}}^{\,\infty}(\Omega)\otimes{\mathcal{C}}\!\ell_n$ is a function satisfying
\begin{equation}\label{ih6gAAA.UNB}
\begin{array}{c}
Du=0\,\,\text{ in }\,\,\Omega,\quad{\mathcal{N}}_{\kappa}u\in L^{1,\infty}(\partial\Omega,\sigma),\,\,\text{ and}
\\[4pt]
\text{the trace $u\big|^{{}^{\kappa-{\rm n.t.}}}_{\partial\Omega}$ exists 
{\rm (}in ${\mathcal{C}}\!\ell_n${\rm )} at $\sigma$-a.e. point on $\partial\Omega$}.
\end{array}
\end{equation} 

Then 
\begin{align}\label{ih6g6FFF.UNB}
{\mathcal{N}}_{\kappa}u\in L^{1}(\partial\Omega,\sigma)&\Longleftrightarrow
u\big|^{{}^{\kappa-{\rm n.t.}}}_{\partial\Omega}\in L^{1}(\partial\Omega,\sigma)\otimes{\mathcal{C}}\!\ell_n
\nonumber\\[6pt]
&\Longleftrightarrow\nu\odot\big(u\big|^{{}^{\kappa-{\rm n.t.}}}_{\partial\Omega}\big)\in 
H^{1}(\partial\Omega,\sigma)\otimes{\mathcal{C}}\!\ell_n
\nonumber\\[6pt]
&\Longleftrightarrow\nu\odot\big(u\big|^{{}^{\kappa-{\rm n.t.}}}_{\partial\Omega}\big)\in 
L^{1}(\partial\Omega,\sigma)\otimes{\mathcal{C}}\!\ell_n.
\end{align} 
In addition, if either of the above memberships materializes then
\begin{align}\label{ih6g6GGG.UNB}
&\|{\mathcal{N}}_{\kappa}u\|_{L^{1}(\partial\Omega,\sigma)}\approx
\big\|u\big|^{{}^{\kappa-{\rm n.t.}}}_{\partial\Omega}\big\|_{L^{1}(\partial\Omega,\sigma)\otimes{\mathcal{C}}\!\ell_n}
\\[6pt]
&\quad\approx\big\|\nu\odot\big(u\big|^{{}^{\kappa-{\rm n.t.}}}_{\partial\Omega}\big)
\big\|_{H^{1}(\partial\Omega,\sigma)\otimes{\mathcal{C}}\!\ell_n}
\approx\big\|\nu\odot\big(u\big|^{{}^{\kappa-{\rm n.t.}}}_{\partial\Omega}\big)
\big\|_{L^{1}(\partial\Omega,\sigma)\otimes{\mathcal{C}}\!\ell_n}
\nonumber
\end{align} 
where the implicit proportionality constants are independent of $u$.
\end{proposition}

\begin{proof}
We reason much as in the proof of Proposition~\ref{yrFCC-TT}, now making use of Proposition~\ref{yrFCC.UND} in place of 
Proposition~\ref{yrFCC}. This permits us to carry out the same argument without having to resort to 
\eqref{ih6g6Fii.b}-\eqref{ih6g6Fii.c} and \eqref{ih6g6Fii.d} (which were crucially reliant on the boundedness of $\partial\Omega$). 
\end{proof}

The stage is now set for extending the scope of Theorem~\ref{yrFCC-TT.aat.TRa} by now allowing $\delta$-{\rm AR} domains in $\mathbb{R}^{n}$ 
with $\delta\in(0,1)$ sufficiently small. 

\begin{theorem}\label{yrFCC-TT.aat.TRa.UNB}
Suppose $\Omega$ is a $\delta$-{\rm AR} domain in $\mathbb{R}^{n}$ {\rm (}cf. Definition~\ref{def:USKT}{\rm )} 
where $\delta\in(0,1)$ is sufficiently small {\rm (}relative to the dimension $n$ and the Ahlfors regularity constants of $\partial\Omega${\rm )}, 
and set $\sigma:={\mathcal{H}}^{n-1}\lfloor\partial\Omega$. Bring in the Riesz transforms 
$\{R_j\}_{1\leq j\leq n}$ associated with the {\rm UR} set $\partial\Omega$ as in \eqref{Cau-RRj}, and  
glue them together into a singular integral operator acting on functions in $L^1(\partial\Omega,\sigma)\otimes{\mathcal{C}}\!\ell_n$ 
as in \eqref{utggG-TRFF}. Then the following statements are equivalent {\rm (}in a natural quantitative fashion{\rm )}:

\begin{enumerate}
\item[(1)] $f\in L^{1}(\partial\Omega,\sigma)\otimes{\mathcal{C}}\!\ell_n$ and 
$R_jf\in L^{1}(\partial\Omega,\sigma)\otimes{\mathcal{C}}\!\ell_n$ for $1\leq j\leq n$;
\item[(2)] $f\in L^{1}(\partial\Omega,\sigma)\otimes{\mathcal{C}}\!\ell_n$ and 
$Rf\in L^{1}(\partial\Omega,\sigma)\otimes{\mathcal{C}}\!\ell_n$;
\item[(3)] $f\in H^{1}(\partial\Omega,\sigma)\otimes{\mathcal{C}}\!\ell_n$.
\end{enumerate}
\end{theorem}

\begin{proof}
Once Proposition~\ref{yrFCC-TT.UNB} has been established, we can produce natural counterparts of 
Theorem~\ref{yrFCC-TT.aat} and Theorem~\ref{yrFCC-TT.aat.2} in the class of $\delta$-{\rm AR} domains with $\delta\in(0,1)$ sufficiently small
by reasoning in exactly the same fashion as before. In turn, the aforementioned versions of these theorems may be now employed to obtain all 
conclusions we presently seek by repeating the proof of Theorem~\ref{yrFCC-TT.aat.TRa} verbatim. 
\end{proof}

We are now prepared to present the following version of Theorem~\ref{u6g5rfffc} in the category of $\delta$-{\rm AR} domains in ${\mathbb{R}}^n$ 
with $\delta\in(0,1)$ sufficiently small, which contains as a particular case (corresponding to the scenario in which the domain in question 
is simply the upper-half space ${\mathbb{R}}^n_{+}$) the classical Fefferman-Stein representation theorem (cf. \cite[Theorem~3, p.\,145]{FeSt72}).

\begin{theorem}\label{u6g5rfffc.UNB}
Let $\Omega$ be a $\delta$-{\rm AR} domain in $\mathbb{R}^{n}$ {\rm (}cf. Definition~\ref{def:USKT}{\rm )} 
where $\delta\in(0,1)$ is sufficiently small {\rm (}relative to the dimension $n$ and the Ahlfors regularity constants of $\partial\Omega${\rm )}, 
and abbreviate $\sigma:={\mathcal{H}}^{n-1}\lfloor\partial\Omega$. Recall the modified Riesz transforms 
$\{R^{{}^{\rm mod}}_j\}_{1\leq j\leq n}$ associated with the {\rm UR} set $\partial\Omega$ as in \eqref{eq:MOD.RZ.1}-\eqref{eq:MOD.RZ.2}.
Then each $f\in{\rm BMO}(\partial\Omega,\sigma)$ may be expressed as 
\begin{equation}\label{utggG-TRFF.rt.A.UNB}
f=f_0+\sum_{j=1}^n R^{{}^{\rm mod}}_jf_j\,\,\text{ on }\,\,\partial\Omega, 
\end{equation}
for some $f_0,f_1,\dots,f_n\in L^\infty(\partial\Omega,\sigma)$ with natural quantitative control. 
\end{theorem}

\begin{proof}
We follow the proof of Theorem~\ref{u6g5rfffc} mutatis mutandis, now employing Theorem~\ref{yrFCC-TT.aat.TRa.UNB} in lieu of 
Theorem~\ref{yrFCC-TT.aat.TRa}, and relying on the duality result 
\begin{equation}\label{eq:Dual.333}
{}_{{\rm BMO}(\partial\Omega,\sigma)}\big\langle R^{{}^{\rm mod}}_jf,g\big\rangle_{H^1(\partial\Omega,\sigma)}
=\int_{\partial\Omega}fR_jg\,d\sigma
\end{equation}
valid for each $j\in\{1,\dots,n\}$, each $f\in{\rm BMO}(\partial\Omega,\sigma)$, and each $g\in H^1(\partial\Omega,\sigma)$
(see \cite[(2.3.40), p.\,353]{GHA.III}).
\end{proof}

Finally, here is the proof of Theorem~\ref{u6g5rfffc-CCC.1111.UNB}:

\vskip 0.08in
\begin{proof}[Proof of Theorem~\ref{u6g5rfffc-CCC.1111.UNB}]
First, \eqref{7ggVV.UNB} is seen from the equivalence (1) $\Leftrightarrow$ (2) in Theorem~\ref{yrFCC-TT.aat.TRa.UNB} 
(written for scalar-valued functions). Next, the left-to-right inclusion in \eqref{utggG-TRFF.rt.A-CCC.UNB}
is implied by Theorem~\ref{u6g5rfffc.UNB}, while the right-to-left inclusion in \eqref{utggG-TRFF.rt.A-CCC.UNB} 
is deduced from \cite[(2.3.38), p.\,353]{GHA.III}.
\end{proof}

\end{document}